\documentclass[11pt,twoside]{article} 

\usepackage{amsthm,amsfonts,amsmath,amssymb,epsfig,color,float,graphicx,verbatim}

\usepackage[utf8]{inputenc} 
\usepackage[T1]{fontenc}    
\usepackage{hyperref}       
\hypersetup{colorlinks,allcolors=black}
\usepackage{url}            
\usepackage{booktabs}       
\usepackage{amsfonts}       
\usepackage{nicefrac}       
\usepackage{microtype}      
\usepackage{caption}
\usepackage{subfigure}
\usepackage{graphicx}
\usepackage{epsfig}
\usepackage{float}
\usepackage{paralist}

\usepackage{cleveref}
\usepackage[vlined,boxed]{algorithm2e}

\usepackage{lipsum}
\usepackage{tikz}
\usetikzlibrary{calc}
\usepackage{xcolor}

\newcommand{\ER}{Erd\H{o}s-R\'{e}nyi\xspace}

\usepackage{multirow}
\usepackage{hyperref}
\usepackage{enumitem}
\usepackage{framed}
\usepackage{nicefrac}
\newtheorem{theorem}{Theorem}[section]
\newtheorem{definition}[theorem]{Definition}
\newtheorem{assumption}[theorem]{Assumption}
\numberwithin{equation}{section}
\newtheorem{lemma}[theorem]{Lemma}
\newtheorem{proposition}[theorem]{Proposition}
\newtheorem{coro}[theorem]{Corollary}
\newtheorem{property}[theorem]{Property}
\newtheorem{remark}[theorem]{Remark}

\newtheorem{claim}[theorem]{Claim}

\numberwithin{equation}{section}

\renewcommand{\phi}{\varphi}
\renewcommand{\epsilon}{\varepsilon}

\newcommand{\DS}{\displaystyle}

\renewcommand{\emptyset}{\varnothing}

\newcommand{\tn}{|\kern-.1em|\kern-0.1em|}

\mathchardef\ordinarycolon\mathcode`\:
\mathcode`\:=\string"8000
\begingroup \catcode`\:=\active
\gdef:{\mathrel{\mathop\ordinarycolon}}
\endgroup

\newcommand\be{\begin{equation}}
\newcommand\ee{\end{equation}}

\newcommand{\mystar}{\frac{Ln}{3C}}
\newcommand{\mysquare}{Lrn}
\newcommand{\myregion}{[-R-4Cr,R+4Cr]}

\title{Optimal Private Median Estimation\\ under Minimal Distributional Assumptions}

\author{
  Christos Tzamos \thanks{Supported by the NSF grant CCF-2008006.}\\
  Department of Computer Science\\
  University of Wisconsin‐Madison\\
  \texttt{tzamos@wisc.edu } \\
  \and
  Emmanouil V. Vlatakis-Gkaragkounis\thanks{ Supported by NSF grants CCF-1703925, CCF-1763970,
CCF-1814873, CCF-1563155, the Simons Collaboration on Algorithms and Geometry, and partially by the Onassis Foundation - Scholarship ID: F ZN 010-1/2017-2018.
} \\
  Department of Computer Science\\
  Columbia University\\
  \texttt{emvlatakis@cs.columbia.edu} \\
  \and
  Ilias Zadik\thanks{Supported by a CDS-Moore-Sloan Postdoctoral Fellowship.}\\
  Center for Data Science\\
  New York University\\
  \texttt{zadik@nyu.edu } \\
}

\begin{document}

\maketitle

\begin{abstract}

We study the fundamental task of estimating the median of an underlying distribution from a finite number of samples, under pure differential privacy constraints. We focus on distributions satisfying the minimal assumption that they have a positive density at a small neighborhood around the median. In particular, the distribution is allowed to output unbounded values and is not required to have finite moments. We compute the exact, up-to-constant terms, statistical rate of estimation for the median by providing nearly-tight upper and lower bounds. Furthermore, we design a polynomial-time differentially private algorithm which provably achieves the optimal performance. At a technical level, our results leverage a Lipschitz Extension Lemma which allows us to design and analyze differentially private algorithms solely on appropriately defined ``typical'' instances of the samples.
\end{abstract}
\clearpage
\tableofcontents
\clearpage

\section{Introduction}
\label{section:Main-Introduction}
Aggregating large amounts of data plays an increasingly important role across various scientific fields.
However, what makes the aggregated data so important is the same element that renders the privacy of the individuals critical.
For example, in the medical field while patient records may be used aggregately to offer faster and better diagnoses, it also amplifies concerns regarding individual privacy in how this data is being used.
Similar concerns permeate nowadays various fields of modern society, forcing even  governments to adopt new and sweeping privacy laws such as the EU’s General Data Protection Regulation 2016/679 (GDPR). From a statistics point of view, the natural and urging associated statistical question is what information of the data can be released while respecting, to the maximum possible extent, the privacy of the individuals involved. Straightforward approaches, such as removing obvious identifiers or releasing only the summaries that concern at least a certain number of members, i.e in a social network or in a HIV-patient health-care database, can be easily broken \cite{dwork2017exposed,rigaki2020survey,borgs2018revealing}.

To understand the emerging trade-off between information and privacy, researchers have employed the notion of differential privacy \cite{Dwork06,DworkMNS06,DworkR14}. In this framework, researchers reported the outcome of their analyses having firstly
injected a well-calibrated noise. More precisely, the level of the noise has been chosen
in such a way that it would be computationally arduous to reconstruct any personal micro data
but  would simultaneously permit an accurate estimation of group statistics.

In this work, we study the optimal statistical performance of a differential private estimator for the task of \emph{median estimation}, one of the simplest and most preferred  statistical index
in various scientific fields associated with an one-dimensional data-set. Despite the fact that the introduction of the median even as a term is relatively recent (the earliest trace in English literature appears to be in 1881
by Francis Galton in one of his surveys 
\cite{galton1882report,david1995first}), it seems that, as a measure of
central tendency, it behaves much more robustly \cite{Huber81,hampel2011robust} than many other statistical indices, such as mean and mode, especially due to its high break-down point \cite{davies2007breakdown,huberbreakdown}. 
As a simple, yet relevant example, the nominal median income reflects much better the real-life setting than
the mean corresponding index. 
Indeed, the median household income -- the income cut-off where half of the households earn more, and half earn less --
is much more robust and illustrative than the mean estimator which could easily have been biased by a small number of billionaires. Furthermore, even the task of determining the critical high-risk age groups during the ongoing COVID-19 pandemic, median estimator plays a significant role in balancing out some possible outliers \cite{dicker2006principles,mediancdc}.

In all these cases, median gives crucial information causing probably several social, political and financial implications.
More importantly, however, safeguarding the individual privacy of the members of a survey could prevent acts of  discrimination
based on the leaked personal data. In social networks 
\cite{task2012guide,Kearns913,zhu2017differentially,liu2020local} 
or medical applications 
\cite{DankarE13,MalinEO13,FredriksonLJLPR14,ZhangLZHS17,YukselKO17} 
for instance, samples may correspond to human subjects and it is crucial that the estimator maintains their privacy while still producing accurate results. 

\textbf{Our model.}
We focus on the statistical problem of estimating \emph{the median of a probability distribution} under privacy constraints. Given sample access to an unknown distribution $\mathcal{D}$ with median $m\left(\mathcal{D}\right)$, the goal is to produce an estimate $\hat m$ that satisfies $|m\left(\mathcal{D}\right) - \hat m| \le \alpha$ with probability $1-\beta$ while respecting the privacy of the samples. An important property of the median estimation task is that it can be efficiently estimated using few samples even under heavy tailed distributions whose mean may not even be well-defined. A minimal statistical assumption for estimation is a mild concentration of the distribution around the median, i.e. that for some parameters $L$ and $r$, the distribution has density at least $L$ in an interval $[m\left(\mathcal{D}\right)-r,m\left(\mathcal{D}\right)+r]$ around the median. In such a case it can be shown that $O(\log(1/\beta) / (\alpha^2 L^2) )$ samples suffice to learn the median within $\alpha < r$ with probability $1-\beta$ \cite{ptr20}.

As mentioned above, differential privacy is the leading measure for quantifying privacy guarantees for an arbitrary randomized function of many input points. It is parameterized by a value $\epsilon$ and it requires that changing the value of any input point may change the probability of any output by at most a factor of $e^{\epsilon}.$ More formally we use the Hamming distance $d_H\left(X, X'\right):= |\{i \in [n] \Big| X_i \neq
X_i'\}|,$ defined for two $n$-tuples of samples $X,X' \in \mathbb{R}^n.$
\begin{definition}\label{definition:epsilon-privacy}
 A randomized algorithm $\mathcal{A}$ is $\epsilon$-differential private if for all subsets $S \in \mathcal{F}$ of the output space $(\Omega, \mathcal{F})$ and $n$-tuples of samples $X_1,X_2 \in \mathbb{R}^n$, it holds  \begin{equation} \label{privdfn}\mathbb{P}\left(\mathcal{A}(X_1) \in S\right) \leq e^{\epsilon d_H\left(X_1,X_2\right)}\mathbb{P}\left(\mathcal{A}(X_2) \in S\right).\end{equation}
\end{definition}

While there has been extensive work in statistical estimation under privacy constraints, obtaining private estimators with optimal sample complexity was only recently shown to be possible. Karwa and Vadhan \cite{KarwaV18} studied the basic problem of learning the parameters of a univariate Gaussian and obtained the optimal rate for estimation. Subsequently, sample efficient estimators were obtained for learning multi-variate Gaussians and product distributions~\cite{Mixtures,MRF,Kamath0SU19}. Further work obtained minimax rates for various other statistical problems~\cite{DiscreteGaussian,HeavyTailed,DuchiWJ16}, including mean estimation for heavy tailed distributions \cite{HeavyTailed,DuchiWJ16}, parameter estimation for \ER graphs and stochastic block models~\cite{BorgsNeurIPS15, BorgsCSZ18,SU19} and general hypothesis selection \cite{PHS,LocallyPrivateHypothesisSelection}.

\subsection{Our contributions}\label{intro:contrib}

Our work considers the setting of the recent work by Avella-Medina and Brunel \cite{Subgaussian,ptr20} who studied the problem of median estimation under arbitrary distributions $\mathcal{D}$ that satisfy a mild concentration condition and whose median lies in a bounded range.

\begin{assumption} \label{assumption:median-mass}
Let  $r,R,L>0$ be fixed positive constants with $Lr \leq \frac{1}{2}.$ The distribution $\mathcal{D}$ has a unique median $m\left(\mathcal{D}\right) \in [-R,R]$. Moreover, it admits a density with respect to the Lebesgue measure $f(u)$ when $u\in [m\left(\mathcal{D}\right)-r,m\left(\mathcal{D}\right)+r]$ and furthermore it holds  $f(u)\geq L$, for all $u\in [m\left(\mathcal{D}\right)-r,m\left(\mathcal{D}\right)+r]$.
\end{assumption}Note that, under the Assumption~\ref{assumption:median-mass} the distribution can output arbitrarily large values and is not required to have finite moments.

We call distributions that satisfy this Assumption~\ref{assumption:median-mass} \emph{admissible}. For the class of admissible distributions, \cite{Subgaussian} obtained two private estimators that in $\tilde{O}(n)$-time produce an estimate $\hat m$ for the median such that $|m\left(\mathcal{D}\right) - \hat m| \le \alpha$ with probability $1-\beta$ using a sample size $n$. The estimators achieve sample complexity
$$\Omega\left(\frac{\log \frac{1}{\beta}}{L^2 \alpha^2}+\frac{\log^2 \frac{1}{\beta} \log \frac{1}{\delta} }{\epsilon^2L\alpha}+\frac{\log( \frac R \alpha + 1 ) \log \frac{1}{\delta}}{\epsilon L r} \right) \quad \text{ and } \quad \Omega\left(\frac{\log \frac{1}{\beta}}{L^2 \alpha^2}+\frac{\log^2 \frac{1}{\beta} (\log\frac{1}{\beta} + \log \frac{1}{\delta}) }{\epsilon^2L\alpha} \right)$$

They both work under the $(\epsilon,\delta)$-differential privacy setting (see Definition~\ref{definition:epsilon-delta-privacy}) which is a weaker guarantee than pure differential privacy.

While, the estimators from \cite{Subgaussian} establish indeed the finite sample complexity on learning the median privately, our main result is a significantly improved estimator which learns the median using less samples and works under the stricter pure $\epsilon$-differential privacy.
\clearpage
\begin{theorem}
There exists an $\epsilon$-differentially~private algorithm that draws $$n=O\left( \frac{\log(\frac{1}{\beta})}{L^2 \alpha^2} +\frac{ \log \left(\frac{1}{\beta}\right)}{\epsilon L\alpha }+\frac{\log \left(\frac{R}{\alpha}+1\right)}{\epsilon Lr} \right)$$ samples from any admissible distribution $\mathcal{D}$  
and in $\operatorname{poly}(n)$ time produces an estimate $\hat m$ that satisfies $|m(\mathcal{D}) - \hat m| \le \alpha$ with probability $1-\beta$.
\end{theorem}

Our result improves upon several aspects of the prior bound and gives stronger privacy guarantees. It has linear dependence in both $1/\epsilon$ and  $\log(1/\beta)$ as opposed to quadratic/cubic. Importantly, we show that the sample complexity we obtain is tight in all possible parameters up to absolute constants. We obtain matching lower-bounds showing that any estimator that produces an $\alpha$-accurate estimate of the median must use at least a constant fraction of the samples used by our estimator.

\begin{theorem}
Consider any $\epsilon$-differentially~private algorithm. There exists an admissible  distribution $\mathcal{D}$ from which $$n=\Omega\left( \frac{\log(\frac{1}{\beta})}{L^2 \alpha^2} +\frac{ \log \left(\frac{1}{\beta}\right)}{\epsilon L\alpha }+\frac{\log \left(\frac{R}{\alpha}+1\right)}{\epsilon Lr} \right)$$
samples are required to produce an estimate $\hat m$ that satisfies $|m(\mathcal{D}) - \hat m| \le \alpha$ with probability $1-\beta$.
\end{theorem}

Our results can also be applied for location estimation in the case of a Gaussian distribution $N(\mu,\sigma^2)$ with known variance $\sigma^2$. In this case, the median is equal to the mean and we recover \emph{the exact tight bounds} obtained in recent work of \cite{Kamath0SU19,KarwaV18}. Thus, despite our minimal distributional assumptions, the trade-off between samples, accuracy and privacy that we obtain is the same as the simpler parametric setting of Gaussian distributions. Finally, one of the two estimators constructed by Avella-Medina and Brunel \cite{Subgaussian,ptr20} are leveraging the notion of smooth sensitivity to create the private median estimator. Interestingly, some recent work by Asi and Duchi \cite[Lemma 5.2.]{Duch20near} proved that \textit{any successful $\epsilon$-private median estimator} which is adding noise based on the smooth sensitivity of the median function requires the use of $n=\Omega\left( \frac{1}{\alpha \epsilon^2}\right)$ samples. Hence, the result of Asi and Duchi combined with our result that in some regime the correct scaling for the sample complexity is $\tilde{\Theta}(\frac{1}{\epsilon \alpha}),$ allows us to conclude the \textit{suboptimality} of the smooth sensitivity method in the private median estimation task. A similar notion of suboptimality of smooth sensitivity for median estimation has been establish by Asi and Duchi themselves in the slightly different context of instance-optimality \cite{Duch20near}.

\subsection{Our approach and techniques}\label{intro:tech}
To obtain our main result we devise a more general principled framework that can be applied to many other problems in private statistical estimation.
Our framework splits the design problem in two steps, and is inspired by earlier work on learning \ER graphs and, their generalization, graphons \cite{borgs2018private,BorgsCSZ18}:
\begin{enumerate}
    \item We obtain private algorithms for a subset of the domain. We focus our attention to a restricted class of instances that are ``typical'' with respect to the distribution. We define a set of constraints that samples should satisfy with high probability and focus only on instances that meet those constraints. In this restricted setting, we obtain an estimator that is accurate and private for these instances only.
    \item We extend our estimator to all instances and guarantee privacy globally. We want to do this while preserving the output of our estimator in the set of typical instances. 
\end{enumerate}

In the {\bf first step}, we make careful use of the well-studied Laplace mechanism, that is a common method for a turning a non-private algorithm to a private one.
It works by adding Laplace noise to the deterministic final output of the algorithm in order to ensure privacy. The amount of noise that needs to be added crucially depends on the sensitivity of the algorithm to changes in the input. 

For the task of computing the median of $n$ points, the sensitivity can be quite high as even if all points are bounded in $[-R,R]$ there are instances where changing a single input may change the median by $R$\footnote{For example if $\lceil n/2 \rceil$ points are located at $-R$ and $\lfloor n/2 \rfloor$ points are located in $R$ changing a single point can move the median from $-R$ to $R$.}. 

While a worst case instance might have large sensitivity, as argued above, one expects that, in a ``typical instance'' drawn randomly from a concentrated distribution, the median does not change so drastically when few of the points move. Indeed, we quantify the notion of a typical instance as the family of instances that have many points in various distances from the sample median and show that for such ``typical" instance, even changing a small fraction of input points cannot move the median of the sample dramatically.

To obtain our main result, we exploit this observation and initially focus only on typical instances of the distribution. For those instances, we show that a variant of the Laplace mechanism, which we call \textit{``flattened'' Laplace mechanism} ensures that differential privacy is guaranteed when focusing only on typical instances. Importantly, while working on these instances this simple variant suffices for obtaining high probability accuracy guarantees, a more delicate approach is necessary to ensure privacy guarantees, which should work even under worst case instances. 

In the {\bf second step}, we aim to define the output distribution of our estimator in the atypical instances to ensure privacy globally. To achieve this, we employ a method for Lipschitz extensions that has been developed in prior work for estimating parameters of \ER graphs, called the ``Extension Lemma'' (Check Proposition 2.1, \cite{borgs2018private}) The ``Extension Lemma'' shows that it is always possible to extend a private algorithm from a smaller domain $A$ to a larger domain $B$ without changing the output distribution for instances in $A$. It achieves this by explicitly defining what the output distribution should be for instances in $B\setminus A$ and only worsens the privacy guarantee by a factor of 2. This allows to obtain worst-case private estimators that are guaranteed to produce accurate estimates with high probability over the instances drawn from the distribution. 

Finally, an important technical contribution of our work is showing that the extension can be computed in polynomial time. We do this by characterizing the structure of the resulting extended output distributions showing that they are piece-wise ``constant'' or ``exponential'' with $\operatorname{poly}(n)$ number of pieces. We can identify all these pieces through a simple greedy algorithm and sample exactly from the resulting distribution. This provides the first instance of a natural problem for which the extension given by the Extension Lemma \cite{borgs2018private} can be computed in polynomial time making progress on a direction suggested by the authors of \cite{borgs2018private}.

\subsection{Further Related Work}
Lipschitz-extensions is a popular technique for designing private algorithms that was used for example in \cite{BlockiBDS13,KasiviswanathanNRS13}. A common theme in those methods is that they develop a Lipschitz estimator for \textcolor{black}{a quantity of interest in} a small domain that is extended to be Lipschitz throughout the whole domain and can then be made private through the Laplace mechanism. Developing a Lipschitz median estimator for our setting would lead to suboptimal rates that scale linearly with the range of possible values $R$\footnote{This is because the Lipschitz-constant must be at least $R/n$ as changing all the input should produce any possible median.}. In contrast, our framework \textcolor{black}{does not simply add noise to a Lipschitz extended estimator but} constructs directly a private mechanism for the small ``typical'' domain and extends it to 
the whole domain via the more general Extension Lemma that has been developed in recent work \cite{borgs2018private,BorgsCSZ18}.

The framework of Lipschitz-extensions was also recently used for estimation of the median along with other statistics such as the variance and the trimmed mean \cite{CummingsD20}. Their paper focuses on arbitrary data-sets without statistical assumptions. Applied to our setting, their methods would yield sub-optimal sample complexity, scaling linearly with the range possible values $R$, again as opposed to the optimal logarithmic dependence we obtain in our work. Median estimation has also been studied in \cite{DworkL09}, who gave an $(\epsilon,\delta)$-differentially private mechanism that is consistent in the limit but did not provide explicit non-asymptotic rates.

\section{Preliminaries: The Extension Lemma}
\label{section:Main-Preliminaries}

In this section for the reader's convenience, we present briefly a main tool behind of our approach, which we refer to from now on as the Extension Lemma. The Extension Lemma is originally stated and proved in \cite{BorgsCSZ18, borgs2018private}. 

Using the notation of Definition \ref{definition:epsilon-privacy} let us consider an arbitrary $\epsilon$-differentially private algorithm defined on input belonging in some set $\mathcal{H} \subseteq \mathbb{R}^n$. Then the Extension Lemma guarantees that the algorithm \emph{can be always extended} to a $2\epsilon$-differentially private algorithm defined for arbitrary input data from $\mathcal{M}$ with the property that if the input data belongs in $\mathcal{H}$, the distribution of output values is the same with the original algorithm. The result in \cite{borgs2018private} is generic, in the sense of applying to any input metric space, but here we present it for simplicity only when the input space if the input space is $\mathbb{R}^n$ and is equipped with the Hamming distance $d_H.$ Formally the result is as following,

\begin{proposition}["The Extension Lemma'' Proposition 2.1, \cite{borgs2018private}] \label{extension} 
Let $\hat{\mathcal{A}}$ be an $\epsilon$-differentially private algorithm designed for input from $\mathcal{H} \subseteq \mathbb{R}^n$ with arbitrary output measure space $(\Omega,\mathcal{F})$. Then there exists a randomized algorithm $\mathcal{A}$ defined on the whole input space $\mathbb{R}^n$ with the same output space which is $2\epsilon$-differentially private and satisfies that for every $X \in \mathcal{H}$, $\mathcal{A}(X) \overset{d}{=}  \hat{\mathcal{A}}(X)$.
\end{proposition}

\section{A Rate-Optimal Estimator}
\label{section:Main-Statistical}

In this section, we present an optimal $\epsilon$-differentially private algorithm for median estimation. We defer the proofs of the stated results to the appendix.  We start with the model of estimation.

\subsection{The Model of Estimation.}
In our statistical model, for some $n \in \mathbb{N}$ one is given $n$ independent identically distributed (i.i.d.) samples $X_1,\cdots,X_n$ from a distribution $\mathcal{D}$. We make the minimal assumption that the distribution $\mathcal{D}$ is \emph{admissible} per Assumption \ref{assumption:median-mass} for some fixed, but arbitrary, parameters $L,R,r>0$ satisfying $Lr \leq \frac{1}{2}.$

We are interested in  estimating the median of $\mathcal{D}$, denoted by $m\left(\mathcal{D} \right),$ from the $n$ samples using an $\epsilon$-differentially private algorithm. Importantly, we assume that the algorithm designer has access to the exact values of the parameters $L,r,R>0$ for which the distribution $\mathcal{D}$ satisfies the admissibility assumptions. Our main interest is to study the following minimax rate:
\begin{equation}
   \mathcal{R}\left(n,\epsilon,L,R,r\right)\left(\alpha\right):= \min_{\mathcal{A} \text{ is } \epsilon-\mathrm{D.P.} }  \max_{(L,R,r)-\mathcal{D} \text{ admissible}} \mathbb{P}_{x_1,\ldots,x_n \sim^{\text{i.i.d.}} \mathcal{D}}\left[ |\mathcal{A}\left(x_1,x_2,\ldots,x_n\right)-m\left(\mathcal{D}\right)| \geq \alpha \right],
\end{equation}where in the above $\min$ we use the abbreviation D.P. for differential privacy and $(L,R,r)-\mathcal{D} \text{ admissible}$ for being admissible with parameters $L,R,r$. Our focus is on \emph{the sample complexity} defined for any $\alpha \in (0,r), \beta \in (0,1)$ and $R,L,r$ with $Lr \leq \frac{1}{2}$ as following.
\begin{equation}\label{eq:sc}
    n_{\mathrm{sc}}\left(\alpha, \beta,R,L,r,\epsilon\right):=\inf\{n \in \mathbb{N}: \mathcal{R}\left(n,\epsilon,L,R,r\right)\left(\alpha\right) \leq \beta\}. 
\end{equation}In words, we want to understand the minimum number of samples that are required for an $\epsilon$-differentially private estimator to estimate the median with accuracy $\alpha$ with probability $1-\beta.$
\subsection{The Optimal Sample Complexity}
We offer a tight, up to absolute constants, characterization of the sample complexity.
Our results follow by proposing and analyzing the performance of an $\epsilon$-differentially private algorithm and also proving the corresponding minimax lower bound.
\begin{theorem}\label{thm:rate}
Assume $\epsilon \in (0,1)$. Then for any $L,R,r>0$ with $Lr \leq \frac{1}{2}$ and $\alpha \in (0,\min\{R,r\}),$ $\beta \in (0,\frac{1}{2})$ it holds
\begin{equation*}
n_{\mathrm{sc}}\left(\alpha, \beta,R,L,r,\epsilon\right) =\Theta\left(\frac{\log(\frac{1}{\beta})}{L^2\alpha^2} +\frac{ \log \left(\frac{1}{\beta}\right)}{\epsilon L\alpha }+\frac{\log \left(\frac{R}{ \alpha}+1\right)}{\epsilon Lr}\right).
 \end{equation*} 
\end{theorem}
 
\subsection{The Optimal Algorithm}\label{sec:optimalalg}We describe here the construction of the optimal $\epsilon$-differentially private algorithm. We use the Extension Lemma and follow the path described in items 1,2 in Section \ref{intro:tech}. Our high-level approach is to first define a a ``typical'' subset of the input space and a restricted differentially private algorithm defined only on the ``typical'' subset, and then use the Extension Lemma (\Cref{extensionApp}) to extend it to a private algorithm on the whole input space. All stated results are proven in the Appendix.

\paragraph{The typical set}We first define a ``typical'' subset $\mathcal{H} \subseteq \mathbb{R}^n$ of the input space. Let $C>1$ be a constant we are going to choose sufficiently large. We define \begin{equation} \label{equation:sensitivity-set-median}
    \mathcal{H}=\mathcal{H}_C=\left\{ X \in \mathbb{R}^n: 
\begin{cases}
\sum_{i\in [n]}\mathbf{1}\{X_i-m(X)\in[0,\frac{\kappa C }{Ln}]\}\ge \kappa +1\\
\sum_{i\in [n]}\mathbf{1}\{m(X)-X_i\in[0,\frac{\kappa C }{Ln}]\}\ge \kappa +1\\
\kappa\in\{1,\cdots,\frac{Lnr}{2C}\}\\
m(X) \in [-R-r/2,R+r/2]
\end{cases}  \right\},
\end{equation} where by $m(X)$ we denote the left empirical median of the set $(X_1,\cdots,X_n)$ (see Definition \ref{definition:median}). In words, the first two families of constraints assume that there are sufficiently many samples falling sufficiently close to the empirical median $m(X)$, at distances which are multiples of $C/Ln$. The motivation for this choice comes from the fact that each interval $[0,\kappa C/Ln]$ is assigned from any admissible distribution $\mathcal{D}$ at least $C \kappa /n$ probability mass based on the Assumption \ref{assumption:median-mass}. Hence, in expectation, it contains at least $C\kappa>\kappa $ out of the $n$ samples. The last constraint of $\mathcal{H}$ assumes that the $m(X)$ falls sufficiently close to the interval the population median $m\left(\mathcal{D}\right)$ is assumed to belong to, based on the Assumption \ref{assumption:median-mass}. 

\paragraph{The restricted algorithm} We now define the algorithm on inputs from $\mathcal{H}$ as a randomised algorithm with density  given by \begin{equation} \label{restricted} f_{\hat{\mathcal{A}}\left(X\right)} (\omega)=\frac{1}{\hat{Z}} \exp\left(-\frac{\epsilon}{4} \min\left\{ \mystar  \left|m(X)-\omega\right|,\mysquare \right\} \right), \omega \in \mathcal{I}:=\myregion \end{equation}
where the normalizing constant is \begin{equation} \label{norm_restr} \hat{Z}=\int_{\mathcal{I}} \exp\left(-\frac{\epsilon}{4} \min\left\{  \mystar \left|m(X)-\omega\right|,\mysquare \right\} \right) \mathrm{d}\omega.\end{equation} We call this distribution a ``flattened'' Laplacian mechanism (see \Cref{flatLap}).
\begin{figure}[t]
    \centering
\begin{tikzpicture}[scale=0.8]
\def\truncatedlaplace{\x,{5/exp(min(abs(\x-3),2.5)/2)}}
\def\xmax{6}
\def\xcen{3}
\def\xmin{-6}
\def\fmax{5/exp(min(abs(\xmax-3),2.5)/2)}
\def\fmin{5/exp(min(abs(\xmin-3),2.5)/2)}
\def\fcen{5/exp(min(abs(\xcen-3),2.5)/2)}
\fill [fill=blue!60,opacity=0.4] (\xmin,0) -- plot[domain=\xmin:\xmax] (\truncatedlaplace) -- ({\xmax},0) -- cycle;
\draw[thick,color=blue,domain=\xmin:\xmax] plot (\truncatedlaplace) node[right] {};
\draw[dashed] ({\xmax},{\fmax}) -- ({\xmax},0) node[below] {$ $};
\draw[dashed] ({\xmin},{\fmin}) -- ({\xmin},0) node[below] {$ $};
\draw[dashed] ({\xcen},{\fcen}) -- ({\xcen},0) node[below] {$m(X)$};
\draw[->] (0,0) -- (\xmax+1,0) node[right] {};
\draw[->] (0,0) -- (\xmin-1,0) node[left] {};
\draw[->] (0,0) -- (0,5) node[above] {};
\end{tikzpicture}
\vspace{-0.45cm}
\caption{``The flattened''-Laplacian Mechanism}
\label{flatLap}
\end{figure}
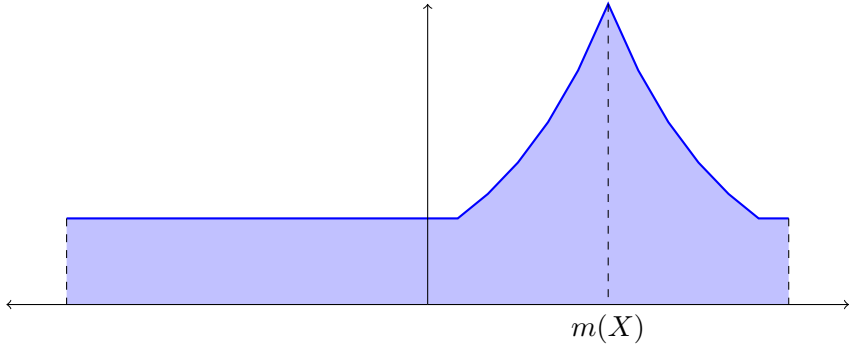
Note that the normalizing constant $\hat{Z}$ does not need to be indexed by the $n$-tuple $X$ as it can be easily proven that for $C>1/2$ the integral on the left hand side of \eqref{norm_restr} gives the same value for all $X \in \mathcal{H}$. Furthermore, observe that to implement the restricted algorithm given as input a data-set $X \in \mathcal{H}$, we first compute the left empirical median $m(X)$ and then sample from the ``flattened'' Laplacian distribution, one part of which is a uniform distibution and the other one is a truncated Laplacian distribution. The exact sampling details from the distribution are defered to Section \ref{sec:algorithmic}.

The motivation to select this distribution is to ensure that the algorithm $\hat{\mathcal{A}}$ is $\epsilon/2$-differentially private. The reason is that the median $m(X)$ as a function from $(\mathcal{H},d_H)$ to the reals, is unavoidably highly sensitive (e.g. it can move distance $\Omega\left(R\right)$ with $ n/2+1$ changes) and therefore by simply adding Laplace noise, as customary in the design of private algorithms, would not suffice to obtain optimal rates. Nevertheless we show that it satisfies the following ``approximate'' Lipschitz constraint with a ``Lipschitz'' constant which is independent of $R$.

\begin{lemma}\label{lemma:sensitivity-median}
Suppose $X, Y \in \mathcal{H}$ with Hamming distance $d_H\left(X,Y\right) \leq Lnr$. Then $$  \left|m(X)-m(Y)\right| \leq  \frac{ 3C}{Ln} d_H\left(X,Y\right).$$
\end{lemma}Using this constraint and the definition of the ``flattened'' Laplacian mechanism, it can be shown using elementary arguments that indeed for inputs $X,Y \in \mathcal{H}$ it indeed holds for all real $q$, $f_{\hat{\mathcal{A}}\left(X\right)} (q) \leq e^{\frac{\epsilon}{2} d_H(X,Y)}f_{\hat{\mathcal{A}}\left(Y\right)} (q),$ which certifies the $\epsilon/2$-differential privacy. The following lemma holds.
\begin{lemma}\label{lem:restricted-DP}
The algorithm $\hat{\mathcal{A}}$ defined on $\mathcal{H}$ is $\frac{\epsilon}{2}$-differentially private.
\end{lemma}
With respect to accuracy guarantees the following elementary result can be proven by using the density of the ``flattened'' Laplace distribution, as defined in Equation \eqref{restricted}.
\begin{lemma}\label{lemma:accuracy-typical0}
Suppose $C>1$ and a fixed $X \in \mathcal{H}$. Then for any $L,R,r>0$ with $Lr \leq \frac{1}{2}$ and $\alpha \in (0,r)$, $\beta \in (0,1)$ for some $ n=O\left(\frac{ \log \left(\frac{1}{\beta}\right)}{\epsilon L\alpha }+\frac{\log \left(\frac{R}{ \alpha}+1\right)}{\epsilon Lr}\right)$ it holds \begin{equation*}\mathbb{P} [ |\hat{\mathcal{A}}(X)-m\left(X\right)|  \geq \alpha] \leq \beta,\end{equation*}where the probability is with respect to the randomness of the algorithm $\hat{\mathcal{A}}$.
\end{lemma}

\paragraph{The general algorithm} Now we construct the final algorithm using the Extension Lemma (Lemma \ref{extension}). First, note that the algorithm $\hat{\mathcal{A}}$ defined on $\mathcal{H}$ based on Lemma \ref{lem:restricted-DP} is $\frac{\epsilon}{2}$-differentially private. Hence, using the Extension Lemma there is an $\epsilon$-differentially private algorithm $\mathcal{A}$ defined on the whole input space $\mathbb{R}^n$ with the property that for any $X \in \mathcal{H}$ it holds $\mathcal{A}(X) \overset{d}{=}\hat{\mathcal{A}}(X).$ This is the algorithm we consider for the upper bound. 

An investigation on the proof of the Extension Lemma in \cite{borgs2018private}, deferred due to space constraints to the appendix and specifically Section \ref{prem:extension}, shows that the algorithm $\mathcal{A}$ can be described as receiving input $X \in \mathbb{R}^n$ and outputting a sample from the continuous distribution with density given by \begin{equation}\label{alg:general} f_{\mathcal{A}\left(X\right)}(\omega)=\frac{1}{Z_{X}}\exp\left(\displaystyle\inf_{X' \in \mathcal{H}} \left[ \frac{\epsilon}{2} d_H(X,X') -\frac{\epsilon}{4} \min\left\{ \mystar  \left|m(X')-\omega\right|,\mysquare \right\}\right] \right), \end{equation}where $\omega \in \myregion$ and $Z_X$ is the appropriate normalizing constant. 

We now briefly elaborate on its accuracy guarantee. The set $\mathcal{H}$, besides crucial for Lemma \ref{lemma:sensitivity-median}, is important to our construction because of the following technical ``typical'' guarantee it has.
\begin{lemma}\label{lemma:typical-set-median} Let $\mathcal{D}$ be an admissible distribution for some parameters $L,R,r>0$ with $Lr \leq \frac{1}{2}$, let $\beta \in (0,1),$ $n \geq 3$ and let $X=(X_1,\ldots,X_n)$ consisting of i.i.d. samples from $\mathcal{D}$. Suppose $C>5$ satisfies $4Ce\exp\left( -\frac{2C}{27}\right)<1/2$. Then for some $C'=C'(C)>0$ and $n=O\left( \frac{\log \frac{1}{\beta}}{L^2r^2}\right)$ it holds
$$\displaystyle \mathbb{P}\left( \exists X' \in \mathcal{H} \text{ s.t. } d_H\left(X , X' \right)  \leq C' \log \frac{1}{\beta}, m(X')=m(X) \right) \geq 1-\beta.$$
\end{lemma}In words, the lemma claims that with probability $1-\beta$ an $n$-tuple of i.i.d. samples from $\mathcal{D}$, call it $X$, has $O\left(\log \frac{1}{\beta}\right)$ Hamming distance from at least one element $X' \in \mathcal{H}$ with $m(X)=m(X').$

Now the accuracy argument for the performance of $\mathcal{A}$ follows from the following high-level idea. By the definition of $\epsilon$-differential privacy (Definition \ref{definition:epsilon-privacy}) we know that the algorithm $\mathcal{A}$ when applied to two inputs of sufficiently ``small'' Hamming distance the outputs are ``close'' in distribution. Hence, Lemma \ref{lemma:typical-set-median} implies that the law of $\mathcal{A}(X)$ with input $X$ is ``close'' to the law of $\mathcal{A}(X')$ where the input is some $X' \in \mathcal{H}$ with $m(X')=m(X)$. Now since $X' \in \mathcal{H}$ by construction the law of $\mathcal{A}(X')$  equals the law of $\hat{\mathcal{A}}(X')$. Furthermore from Lemma \ref{lem:restricted-DP} we know that $\hat{\mathcal{A}}(X')$, and therefore also $\mathcal{A}(X)$, concentrates around $m(X')=m(X)$.  Finally standard concentration results (see e.g. \cite[Lemma 3]{ptr20}) imply that with $n=\Omega\left(\log \frac{1}{\beta}/(L^2\alpha^2)\right)$ samples it holds $|m(X)-m(\mathcal{D})|\leq \alpha$ with probability $1-\beta.$ 

Using the above idea we obtain the following result.
\begin{theorem}\label{thm:upper_bound}
Suppose $C>0$ is a sufficiently large constant, 
$\epsilon \in (0,1)$, $\mathcal{D}$ is an admissible distribution with respect to some arbitrary $L,R,r>0$ with $Lr \leq \frac{1}{2}$ and $\mathcal{A}$ is the $\epsilon$-differentially private algorithm defined above. Then for any $\alpha \in (0,r)$ and $ \beta \in (0,1)$ for some $$ n=O\left( \frac{\log \left( \frac{1}{\beta}\right)}{L^2 \alpha^2}+\frac{ \log \left(\frac{1}{\beta}\right)}{\epsilon L\alpha }+\frac{\log \left(\frac{R}{ \alpha}+1\right)}{\epsilon Lr}\right)$$ it holds $\quad\quad \quad \mathbb{P}_{X_1,X_2,\ldots,X_n \overset{iid}{\sim} \mathcal{D}} [ |\mathcal{A}(X_1,\ldots,X_n)-m\left(\mathcal{D}\right)|  \geq \alpha] \leq \beta.$
\end{theorem}
Finally, the discussion on the computationally efficient implementation of the algorithm $\mathcal{A}$ is deferred to Section \ref{sec:algorithmic}.

\subsection{Lower Bounds}\label{sec:LB}

In this subsection we discuss the lower bound part on $n_{\mathrm{sc}}(\alpha,\beta)$ as stated in Theorem \ref{thm:rate}. We establish that all three terms are necessary for all values of the parameters, using a different method for each one of them. We omit the complete proofs for the appendix and provide only brief hints for the methods followed.  First, we start with the $\epsilon$-independent part (the ``non-private'' part).
\begin{proposition}\label{propLB1}
Let $L>0, \epsilon \in (0,1), R, r>0$ with $Lr \leq \frac{1}{2}$ and $\alpha \in (0,\min\{r,R\}), \beta \in (0,\frac{1}{2})$. Suppose $n=o\left( \log \left( \frac{1}{\beta}\right)/(L^2 \alpha^2)\right).$ Then for any algorithm $\mathcal{A}$  there exists an admissible distribution $\mathcal{D}$ for this specific value of the parameters with the property $$\mathbb{P}_{X_1,X_2,\ldots,X_n \overset{iid}{\sim} \mathcal{D}} [ |\mathcal{A}(X)-m\left(\mathcal{D}\right)|  \geq \alpha] > \beta.$$
\end{proposition}
The key argument behind the proposition is to show that learning a a Bernoulli random variable at accuracy $\gamma>0$ with probability $1-\beta$ reduces to learning the median of an admissible distribution at accuracy $\gamma /L$ with probability $1-\beta$. Then using that the sample complexity of learning a Bernoulli distribution is known to require $\Theta\left( \log \left( \frac{1}{\beta}\right)/\gamma^2\right)$ samples we conclude the proof.
We \textcolor{black}{are ready to} state the lower bound on the first term on the $\epsilon$-dependent part.

\begin{proposition}\label{propLB2}
Let $L>0, \epsilon \in (0,1), R, r>0$ with $Lr \leq \frac{1}{2}$ and $\alpha \in (0,\min\{r,\frac{R}{2}\}), \beta \in (0,\frac{1}{2})$. Suppose $n=o\left(  \log \left(\frac{1}{\beta}\right)/(\epsilon L\alpha)\right).$ Then for any $\epsilon$-differentially private algorithm $\mathcal{A}$  there exists an admissible distribution $\mathcal{D}$ for this specific value of the parameters with the property $$\mathbb{P}_{X_1,X_2,\ldots,X_n \overset{iid}{\sim} \mathcal{D}} [ |\mathcal{A}(X)-m\left(\mathcal{D}\right)|  \geq \alpha] >\beta.$$

\end{proposition} 
The proof of this lower bound follows by a ``local'' argument. We construct two admissible distribution $\mathcal{D},\mathcal{D}'$ with medians of distance $O\left(\alpha\right)$ which we can couple so that drawing $X,X'$ two i.i.d. $n$-tuples from $\mathcal{D},\mathcal{D}'$ respectively, it holds $d_H\left(X,X'\right)=O(\alpha L n)$ with high probability. Using the definition of differential privacy we conclude that $\mathcal{A}(X),\mathcal{A}(X')$ assign the same probability to each event up to a multiplicative factor $e^{O(\epsilon \alpha L n)}.$ Yet, we prove that an algorithm contradicting the assumptions of Proposition \ref{propLB2} will imply that while $\mathcal{A}(X)$ assigns mass at least $1-\beta$ on the interval of width $2\alpha$ around $m(\mathcal{D})$, the algorithm  $\mathcal{A}(X')$ needs to assign mass at most $\beta.$  Hence combining the above, we conclude that it must be true that $e^{-\epsilon \alpha Ln}=O\left(\beta \right)$ or $n=\Omega\left(  \log \left(\frac{1}{\beta}\right)/(\epsilon L\alpha)\right).$

\begin{proposition}\label{propLB3}
Let $L>0, \epsilon \in (0,1), R, r>0$ with $Lr \leq \frac{1}{2}$ and $\alpha \in (0,\min\{r,R\}), \beta \in (0,\frac{1}{2})$. Suppose $n=o\left( \log \left(\frac{R}{ \alpha}+1\right)/(\epsilon Lr)\right).$ Then for any $\epsilon$-differentially private algorithm $\mathcal{A}$  there exists an admissible distribution $\mathcal{D}$ for this specific value of the parameters with the property $$\mathbb{P}_{X_1,X_2,\ldots,X_n \overset{iid}{\sim} \mathcal{D}} [ |\mathcal{A}(X)-m\left(\mathcal{D}\right)|  \geq \alpha] >\beta.$$

\end{proposition}
This lower bound follows by a ``global'' argument. We show that any $\epsilon$-differentially private algorithm contradicting the conclusion of Proposition \ref{propLB3} needs to assign at least $\Omega\left(e^{-Lrn}\right)$ probability mass in at least $\frac{R}{\alpha}+1$ distinct intervals in $[-R,R]$. Since the total probability mass must sum up to one we obtain the desired lower bound.

\section{A Polynomial-Time Implementation of the Rate-Optimal Estimator}
\label{section:Main-Algorithm}
\label{sec:algorithmic} 
In the previous section we constructed an differential private estimator denoted by $\mathcal{A}$ and proved that it is up-to-constants rate-optimal with respect to sampling complexity. In this section we report the following results. In the first subsection we prove that it can be efficiently implemented in worst-case polynomial time and average-case almost-linear time. Specifically, in the first subsection we sketch the high-level ideas to provide intuition. In the following subsections we offer the exact implementation details as well as the statements and their proofs regarding its computational complexity guarantees. 

\subsection{Algorithm outline and main result }


We start with a pseudocode for the general algorithm sampling $\mathcal{A}$ as defined in \Cref{alg:general} and its crucial property that it extends the simple ``flattened Laplacian" algorithm given by $\hat{\mathcal{A}}$ as defined in Equation \ref{restricted}.

\begin{algorithm}[H]
Sort $X$ s.t $X_1\le\ldots\le X_n$.\\
$m(X)\gets\textsc{LeftMedian}(X)$\\
Let $\mathcal{I}\equiv\myregion\equiv[-B,B]  $\\
Let $\begin{cases}
\mathcal{D}_{\text{restricted}}\sim \exp\left(-\frac{\epsilon}{4} \min\left\{ \mystar  \left|m(X)-\omega\right|,\mysquare \right\} \right), \omega \in \mathcal{I}\\
\mathcal{D}_{\text{general}}\sim \exp\left(\displaystyle\inf_{X' \in \mathcal{H}} \left[ \left(\frac{\epsilon}{2} d_H(X,X')\right) -\frac{\epsilon}{4} \min\left\{ \mystar  \left|m(X')-\omega\right|,\mysquare \right\} \right]\right), \omega \in \mathcal{I}\\
\end{cases}$\\
\If{$X\in \mathcal{H}$}
{
      $m \gets \text{Sample from }\mathcal{D}_{\text{restricted}}$ (Case (A))
}
\Else
{
      $m \gets \text{Sample from }\mathcal{D}_{\text{general}}$ (Case (B))
}
\Return $m$
\caption{$\textsc{PrivateMedian}(X=(X_1,\ldots,X_n))$:
This is the pseudocode for the general algorithm of sampling from the extended mechanism based on \Cref{alg:general}. The notation $D \sim p(\omega)$ is used to denote that the density of the distribution $D$ is proportional to the function $p(\omega)$.}
\label{alg:Estimator}
\end{algorithm}

We start with describing the high level ideas behind our implementation and its worst-case computational complexity guarantees. We then conclude with its average-time computational complexity guarantees as well.

First, notice that in the case of a data-set $X$ that belongs to the typical
set $\mathcal{H}$, our estimator corresponds to sampling from $\mathcal{D}_{\text{restricted}} -$ the ``flattened-Laplacian'' continuous distribution of $\hat{\mathcal{A}}(X')$ as defined in Equation \eqref{restricted}.
Therefore, the implementation of the algorithm with input from $\mathcal{H}$ corresponds to a simple two-phase protocol. In the first round we flip an appropriately biased coin to decide between the two regions where the estimator behaves differently: \emph{i)} the region where the estimator samples proportional to the Laplacian density and \emph{ii)} the region where the estimator samples proportional to the uniform distribution. In the second round we  simply apply conditional sampling either from a Laplacian or a uniform distribution, depending on the outcome of the coin flip.

Now, notice that we can sort $X$ in $O(n \log n)$-time, the query $X\in\mathcal{H}$ can be easily answered with counting binary-searches in $O(n\log n)$ time and, as we just discussed, it is easy to sample in $\operatorname{poly}(n)$-time from $\mathcal{D}_{\text{restricted}}.$  Therefore, to implement efficiently the extended mechanism $\mathcal{A}$, 
a $\operatorname{poly}(n)$-time sample generator  
is sufficient and necessary for the (appropriately normalized) distribution $\mathcal{D}_{\text{general}}$ (See \Cref{alg:general}) whose unnormalized density is given by:
\begin{equation*}
\textsc{UnNormalized}(X,\omega)=\exp\left(\displaystyle\inf_{X' \in \mathcal{H}} \left[ \left(\frac{\epsilon}{2} d_H(X,X')\right) -\frac{\epsilon}{4} \min\left\{ \mystar  \left|m(X')-\omega\right|,\mysquare \right\} \right]\right)
\end{equation*}
defined in $[-B,B]=\myregion$. Notice that such a task is a-priori challenging as even the evaluation of $\textsc{UnNormalized}(X,\omega)$ at a single instance of $\omega$ appears non-trivial because it involves a complicated optimization problem over $X' \in \mathcal{H}.$

Our main technical contribution is to show that this underlying constrained optimization problem associated with the definition of the density of $\mathcal{D}_{\text{general}}$ can be solved in polynomial-time for every fixed $\omega$ and then show how to use it to obtain a polynomial-time sample generator for the desired distribution. A key observation is that $\textsc{UnNormalized}(X,\omega)$ depends only on the possible values of the median $m(X')$ and the distance $d_H(X,X'),$ for any data-set $X'\in \mathcal{H}$. 
To simplify our notation, we introduce in our analysis the notion of the $\textsc{TypicalHamming}$ distance of a data-set $X$ and the typical set $\mathcal{H}$, which we prove that it can be computed in $\operatorname{poly}(n)$-time. 

\begin{lemma}\label{definition:TypicalHamming}
Let us define $\textsc{TypicalHamming}(X,\xi)$ as:
\[\textsc{TypicalHamming}(X,\xi)=\displaystyle\min_{\begin{matrix}X'\in\mathcal{H},m(X')=\xi
\\\xi\in[-R-r/2,R+r/2]\end{matrix}
} d_H(X,X').\] Then for any data-set $X$, \textsc{TypicalHamming}$(X,\cdot)$ is piece-wise constant function in $[-R-r/2,R+r/2]$ with at most $\operatorname{poly}(n)$ changes.
Additionally, for any $\xi$ and data-set $X$, $\textsc{TypicalHamming}(X,\xi)$ can be computed exactly using the $\operatorname{poly}(n)$
\footnote{For the interested reader, both the number of changes is at most $3n^2$ and the  time complexity is $O(n^4)$, but we will keep for simplicity the general expression $\operatorname{poly}(n)$ in our formal statement}-time algorithm described in Algorithm \ref{alg:typicalHamming}.
\end{lemma}

Leveraging the partition of the interval $[-B,B]$ implied by the different constant parts of \Cref{definition:TypicalHamming} we are able to show the following for $\textsc{UnNormalized}(X,\omega)$:
\begin{lemma}\label{lemma:algo:concatenation-laplacian}
For any given data-set $X$, there is a partition of $[-B,B]$ to a collection $\mathbb{J}$ of $\operatorname{poly}(n)$ consecutive intervals, $\mathbb{J}=\{ J_1,\cdots,J_{\operatorname{poly}(n)} \}$ such that :
\begin{equation*} 
 \text{For any $J_i \in \mathbb{J}\quad :\textsc{UnNormalized}(X,\omega)=
 \exp\left(\alpha_i \omega+\beta_i\right) \quad \forall \omega\in J_i$}
 \end{equation*}
 Moreover, we can compute exactly $\mathbb{J}$ and the constants $(\alpha_i,\beta_i)$ for every $J_i\in\mathbb{J}$ in $\operatorname{poly}(n)$ time.
\end{lemma}
In words, \Cref{lemma:algo:concatenation-laplacian} implies that our extended mechanism consists a concatenation of $\operatorname{poly}(n)$  exponential or uniform distributions.
Therefore, its implementation corresponds again to a simple two-phase protocol, similar to the implementation of the much simpler $\mathcal{D}_{\text{restricted}}$. In the first round we  sample from a discrete distribution on a $\operatorname{poly}(n)$-cardinality domain  to decide among the regions $J_1,\cdots,J_{\operatorname{poly}(n)} $. In the second round we  simply apply conditional sampling from the corresponding either exponential or uniform  which will
be truncated on the region which was the outcome of the first round. Notice that all our decomposition results apply to any input vector $X$. Combining the above, we show that a \textit{worst-case} $\operatorname{poly}(n)$-implementation of $\textsc{PrivateMedian}$ is possible via the aforementioned decomposition of $\mathcal{D}_{\text{general}}$. 

Now, a potentially interesting remark is in order. Recall that the typical set $\mathcal{H}$ as defined in \eqref{equation:sensitivity-set-median}, which is used to define the algorithm $\mathcal{A}$, is a function of the parameter $C>1$. Furthermore, recall that based on Theorem \ref{thm:upper_bound}, the algorithm $\mathcal{A}$ with $C>0$ being a sufficiently large constant is proven to achieve, up-to-constants, the optimal sample complexity. Now, importantly, our implementation of $\mathcal{A}$ remains \textit{worst-case polynomial-time for any choice of $C,$ possibly scaling with $n$,} since such a change affects only the arithmetic operations which are assumed to be executed always in $O(1)$-time. Using this observation, we can show that by tweaking appropriately this parameter $C$ from an absolute constant $\Theta(1)$ to scaling logarithmically-in-$n$ factor $\Theta(\log n)$ that the average-case time complexity of the algorithm becomes up-to-logarithms linear. Furthermore, we can maintain for the algorithm for this choice of $C=\Theta( \log n)$ an, up-to-logarithms this time, optimal sample complexity guarantee.  

More precisely, combining all the above, the following theorem holds, which slightly generalizes Theorem \ref{thm:upper_bound}, as it allows $C$ to scale with $n$.
\begin{theorem}\label{thm:upp_gen}
Suppose $C=C_n>0$, possibly scaling with $n$, is bigger than a sufficiently large constant, $\epsilon \in (0,1)$, $\mathcal{D}$ is an admissible distribution. 
Then for any $\alpha \in (0,r)$ and $ \beta \in (0,1)$, there exists an $\epsilon$-differentially~private implementation of 
\textsc{PrivateMedian} such that if it holds  $$ n=O\left( \frac{\log \left( \frac{1}{\beta}\right)}{L^2 \alpha^2}+C_n\frac{ \log \left(\frac{1}{\beta}\right)}{\epsilon L\alpha }+\frac{\log \left(\frac{R}{ \alpha}+1\right)}{\epsilon Lr}\right),$$  it produces an estimate $\hat m$ that satisfies $|m(\mathcal{D}) - \hat m| \le \alpha$ with probability $1-\beta$ and runs 
\begin{enumerate}
    \item[(i)] at most in $\operatorname{poly}(n)$ time ,  for some sufficiently large constant $C>0$.
    \item[(ii)] on expectation in $\tilde{O}(n)$ time, for some sufficiently large constant $C_0>0$ and $C=C_0 \log n$.
\end{enumerate}
\end{theorem}

The following clarification remark is in order.
\begin{remark}
Using the elementary property that for two diverging sequences $A_n,B_n, n \in \mathbb{N}$, it holds $A_n/\log A_n = \Omega( B_n)$ if and only if $A_n=\Omega(B_n \log B_n)$, the $\epsilon$-differentially~private implementation of 
\textsc{PrivateMedian} mentioned in Theorem \ref{thm:upp_gen} in the case $C=C_0 \log n$ for some sufficiently large constant $C_0>0$, requires a sample size of order $$ n=\tilde{O}\left( \frac{\log \left( \frac{1}{\beta}\right)}{L^2 \alpha^2}+\frac{ \log \left(\frac{1}{\beta}\right)}{\epsilon L\alpha }+\frac{\log \left(\frac{R}{ \alpha}+1\right)}{\epsilon Lr}\right)$$to output an estimate $\hat m$ that satisfies $|m(\mathcal{D}) - \hat m| \le \alpha$ with probability $1-\beta$.
\end{remark}


\subsection{Implementation of \textsc{PrivateMedian} and Time Complexity}
 In this section we prove Theorem \ref{thm:upp_gen}. In \Cref{appendix:worst-case}, we focus on the, arguably pessimistic, model of worst-case analysis, and we show that we can implement the algorithm $\mathcal{A}$ at most in $\operatorname{poly}(n)$ time. The worst-case complexity guarantee of our implementation hold under arbitrary $C=C_n>0$ bigger than a sufficiently large constant. Notice that combining this time-complexity guarantee with the sample complexity results from Theorem \ref{thm:upper_bound} we can conclude the part (i) of Theorem \ref{thm:upp_gen}.  We then proceed with average-case analysis presented in \Cref{appendix:average-case} where we assume that $C$ is scaling logarithmically with $n$. We show that this sacrifices the optimality of our sample complexity just up to a logarithmic factors and in that case our implementation actually runs on expectation in $\tilde{O}(n)$ time.

\emph{Note:For our worst-case and average-case analysis, we will assume that there is an oracle $\mathcal{O}$ such that 
we can always sample from Bernoulli, Uniform, Laplace, (Negative) Exponential Distribution and their truncated versions in an interval $I$ in $O(1)$ time. The exact details are in \Cref{appendix:truncated-sampling}.}
\subsubsection{Worst-case Time Complexity of \textsc{PrivateMedian}  }\label{appendix:worst-case}
We start by presenting a simple random generator for $\mathcal{D}_{\text{restricted}}$ (Case (A) of Algorithm \ref{alg:Estimator}, $X \in \mathcal{H}$.):
\begin{lemma}\label{lem:restr_time}
For a given a date-set $X=(X_1,\cdots,X_n) \in \mathcal{H}$ and its median $m(X)$, there exists a $O(1)-$protocol that generates a sample from $\mathcal{D}_{\text{restricted}}$.
\end{lemma}
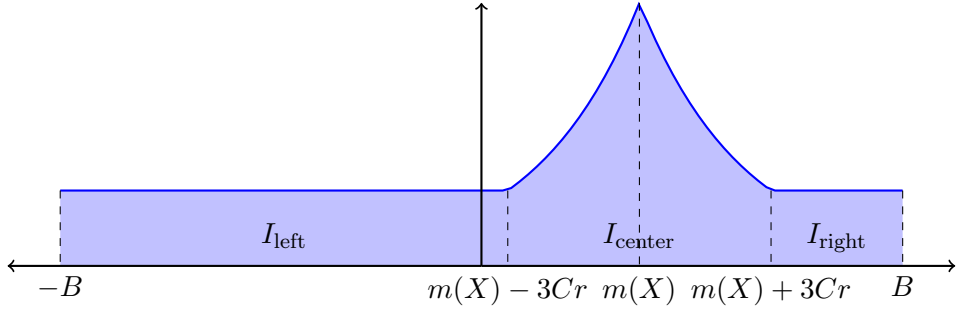
\begin{proof}

\begin{figure}[h!]
    \centering
\begin{tikzpicture}[scale=0.7]
\def\normaltwo{\x,{5/exp(min(abs(\x-3),2.5)/2)}}
\def\uniform{\x,{5/exp(2.5/2)}}

\def\xmax{8}
\def\xcen{3}
\def\xmin{-8}
\def\xlapmin{0.5}
\def\xlapmax{5.5}

\def\fmax{5/exp(min(abs(\xmax-3),2.5)/2)}
\def\fmin{5/exp(min(abs(\xmin-3),2.5)/2)}
\def\fcen{5/exp(min(abs(\xcen-3),2.5)/2)}
\def\flapmin{5/exp(min(abs(\xlapmin-3),2.5)/2)}
\def\flapmax{5/exp(min(abs(\xlapmax-3),2.5)/2)}

\fill[fill=blue!60,opacity=0.4,samples=100] (\xmin,0) -- plot[domain=\xmin:\xmax] (\normaltwo) -- ({\xmax},0) -- cycle;

\draw[thick,color=blue,domain=\xmin:\xmax,samples=100] plot (\normaltwo) node[right] {};

\draw[dashed] ({\xmax},{\fmax}) -- ({\xmax},0) node[below] {$B$};
\draw[dashed] ({\xmin},{\fmin}) -- ({\xmin},0) node[below] {$-B$};
\draw[dashed] ({\xcen},{\fcen}) -- ({\xcen},0) node[below] {$m(X)$};
\draw[dashed] ({\xlapmin},{\flapmin}) -- ({\xlapmin},0) node[below] {$m(X)-3Cr$};
\draw[dashed] ({\xlapmax},{\flapmax}) -- ({\xlapmax},0) node[below] {$m(X)+3Cr$};

\draw ({(\xmin +\xlapmin)/2},1) node[below] {$I_{\text{left}}$};
\draw ({(\xmax +\xlapmax)/2},1) node[below] {$I_{\text{right}}$};
\draw ({\xcen},1) node[below] {$I_{\text{center}}$};


\draw[->,thick] (0,0) -- (\xmax+1,0) node[right] {};
\draw[->,thick] (0,0) -- (\xmin-1,0) node[left] {};
\draw[->,thick] (0,0) -- (0,5) node[above] {};
\end{tikzpicture}
\caption*{``The flattened''-Laplacian Mechanism: $ \mathcal{D}_{\text{restricted}} \sim  \exp\left(-\frac{\epsilon}{4} \min\left\{ \mystar  \left|m(X)-\omega\right|,\mysquare \right\} \right)$}
\end{figure}
We define three intervals \footnote{One or two of them may equal to $\emptyset$.} which partition the support $[-B,B]$ of $\mathcal{D}_{\text{restricted}}$, with the property that when $\mathcal{D}_{\text{restricted}}$ is restricted on any of them, it reduces to a (truncated) standard distribution from which we assume we have oracle access to sample from. Specifically let us define \emph{i)}
$I_{\text{left}}=[-B,m(X)-3Cr]$, \emph{ii)} $I_{\text{center}}=[m(X)-3Cr,m(X)+3Cr]$ and \emph{iii)} $I_{\text{right}}=[m(X)+3Cr,B]$.
It is immediate to see that restricted in the interval $I_{\text{center}}$ $\mathcal{D}_{\text{restricted}}$ corresponds a (truncated) Laplacian distribution and in $I_{\text{left}},I_{\text{right}}$
it corresponds to a uniform distribution. 

Given this observation, in Algorithm \ref{alg:RestrictedEstimator} we present the pseudo-code of an evidently $O(1)$-time sample generator for $\mathcal{D}_{\text{restricted}}.$
\begin{algorithm}[h!]
Let $I_{\text{left}},I_{\text{center}},I_{\text{right}}$ be the decomposition of $\mathcal{I}=[-B,B]$ as described in the above figure\\
//\texttt{ For extreme values of $C,r,m(X)$,}\\//\texttt{ either $I_{\text{left}},I_{\text{center}}$ or $I_{\text{right}}$ may equal to $\emptyset$ }\\
Let $\begin{Bmatrix}p_{\text{center}}&=&\displaystyle\int_{I_{\text{center}}} \exp\left(-\frac{\epsilon}{4}  \mystar  \left|m(X)-\omega\right|  \right)\mathrm{d}\omega\\
 p_{\text{left}}&=&\displaystyle\int_{I_{\text{left}}} \exp\left(-\frac{\epsilon}{4} \mysquare  \right)\mathrm{d}\omega\\
 p_{\text{right}}&=&\displaystyle\int_{I_{\text{right}}} \exp\left(-\frac{\epsilon}{4} \mysquare  \right)\mathrm{d}\omega
 \end{Bmatrix}$\\
Let $p=p_{\text{left}}+p_{\text{center}}+p_{\text{right}}$\\
Toss a trinary coin $c:=\{\mathcal{L,C,R}\}$ with probability $(\tfrac{p_{\text{left}}}{p},\tfrac{p_{\text{center}}}{p},\tfrac{p_{\text{right}}}{p})$, correspondingly.\\
\If{$c$ outputs $\mathcal{L}$}
{
      $s \gets \text{Sample from }\textsc{Uniform}[I=I_{\text{left}}]$ 
}
\If{$c$ outputs $\mathcal{C}$}
{
      $s \gets \text{Sample from }\textsc{TruncatedLaplace}[\mu=m(X),\sigma=\frac{12C}{Ln\epsilon},I=I_{\text{center}} ]$ 
}
\If{$c$ outputs $\mathcal{R}$}
{
      $s \gets \text{Sample from }\textsc{Uniform}[I=I_{\text{right}}]$ 
}
\Return $s$
\caption{$\text{Sample from }\mathcal{D}_{\text{restricted}}$}
\label{alg:RestrictedEstimator}
\end{algorithm}
\end{proof}

We continue by presenting the construction of our sample generator for $\mathcal{D}_{\text{general}}$ ((Case (B) of Algorithm \ref{alg:Estimator}, $X \not \in \mathcal{H}$.)).
Firstly, let's recall the un-normalized term of $\mathcal{D}_{\text{general}}$, using  the definition of $f_{\mathcal{A}\left(X\right)}$ (\Cref{alg:general}):

\begin{equation*}
\textsc{UnNormalized}(X,\omega)=\exp\left(\displaystyle\inf_{X' \in \mathcal{H}} \left[ \left(\frac{\epsilon}{2} d_H(X,X')\right) -\frac{\epsilon}{4} \min\left\{ \mystar  \left|m(X')-\omega\right|,\mysquare \right\} \right]\right)
\end{equation*}
defined in $[-B,B]=\myregion$.

Our main technical contribution for this part is showing that the above  constrained optimization problem can be solved in polynomial-time and then using it to obtain a polynomial-time sample generator for the desired distribution. A key observation is that $\textsc{UnNormalized}(X,\omega)$ depends only on the possible values of the median $m(X')$ and the distance $d_H(X,X'),$ for any data-set $X'\in \mathcal{H}$. 

To motivate our first technical lemma let's re-write the above optimization expression with an equivalent form~ 
\begin{equation*}
\textsc{UnNormalized}(X,\omega)=\exp\left(\inf_{k\in[n]}\displaystyle\inf_{
\begin{matrix}X' \in \mathcal{H}\\m(X')=\xi\\\xi\in[-R-r/2,R+r/2]\\d_H(X,X')=k\end{matrix}}
\left[ \left(\frac{\epsilon}{2} k\right) -\frac{\epsilon}{4} \min\left\{ \mystar  \left|\xi-\omega\right|,\mysquare \right\} \right]\right)
\end{equation*}

In order to discretize the optimization space over $\mathcal{H}$, we use the following observation that we presented in \Cref{definition:TypicalHamming}.

\begin{lemma}[Restated \Cref{definition:TypicalHamming}]
\label{appendix:definition:TypicalHamming}
$\\$Let us define $\textsc{TypicalHamming}(X,\xi)$ as:
\[\textsc{TypicalHamming}(X,\xi)=\displaystyle\min_{\begin{matrix}X'\in\mathcal{H},m(X')=\xi
\\\xi\in[-R-r/2,R+r/2]\end{matrix}
} d_H(X,X').\] Then for any data-set $X$, \textsc{TypicalHamming}$(X,\cdot)$ is piece-wise constant function in $[-R-r/2,R+r/2]$ with at most $\operatorname{poly}(n)$ changes.
Additionally, for any $\xi$ and data-set $X$, $\textsc{TypicalHamming}(X,\xi)$ can be computed exactly using the $\operatorname{poly}(n)$
\footnote{For the interested reader, both the number of changes is at most $3n^2$ and the  time complexity is $O(n^4)$, but we will keep for simplicity the general expression $\operatorname{poly}(n)$ in our formal statement}-time algorithm described in Algorithm \ref{alg:typicalHamming}.
\end{lemma}
\begin{proof}
Our proof \textcolor{black}{follows by combining the following claims:}
\begin{enumerate}
    \item (\Cref{claim:algo:TypicalHamming}) 
We begin the proof by presenting a simple greedy algorithm, Algorithm \ref{alg:typicalHamming}, that outputs a data-set $X'$ such that it belongs to the typical set $\mathcal{H}$, and its median $m(X')$ equals to an input value $\xi$. From all the possible choices of $X'$ the greedy algorithm chooses the one that minimizes the Hamming distance from an input data-set $X=(X_1,X_2,\ldots,X_n)$.
    \item (\Cref{claim:partition:H}) Having established the correctness of the greedy method, we show that there exists a partition of $[-R-r/2,R+r/2]$ to a collection of  $\operatorname{poly}(n)$ disjoint subintervals such that in every subinterval, the output of the greedy algorithm remains constant.
\end{enumerate}

\begin{claim}\label{claim:algo:TypicalHamming}
The procedure Algorithm \ref{alg:typicalHamming} for solving $\textsc{TypicalHamming}(X=(X_1,\ldots,X_n),\xi)$ outputs a data-set $X'$ satisfying the conditions that $m(X')=\xi$ and $X'\in \mathcal{H}$
that minimizes the Hamming distance between $X,X'$ over all $X'\in\mathcal{H}$, i.e $\displaystyle \arg\min_{X'\in\mathcal{H}}\{d_H(X,X')\}$.
Moreover $\textsc{TypicalHamming}$ runs in $O(n^2)$ time.
\end{claim}

\begin{algorithm}
\textbf{If} $\xi\not\in [-R-r/2,R+r/2]$ \textbf{then} \Return \texttt{Impossible}\\
Create a copy $X'$ of data-set $X$ \texttt{i.e $X_i'\gets X_i \forall \ i \in [n]$}\\
Sort $X'$ s.t $X_1'\le\ldots\le X_n'$ \texttt{--Tie breaker: Maintain the order before the sorting.--}  \\
Extract the permutation $\pi: [n] \rightarrow [n]$ with $X'_{\pi(i)}=X_{i}$.\\
Part 1.\texttt{Rebalance the data-set s.t $\xi$ is the statistical median}

\If{$\Big(\textsc{LeftMedian}(X')\neq \xi \Big)$}
{
    Let $k= \arg\min_{\kappa\in[n]}\{\left |\lceil \frac{n}{2} \right \rceil-\kappa|$ s.t $X'_{\kappa}\le \xi \le X'_{\kappa+1}\}$ . \\
        \For{$i\in\left[\left |\lceil \frac{n}{2} \right \rceil-k|\right]$}
        {
        
            \If{$k<\left \lceil \frac{n}{2} \right \rceil$}
            {
                Set $X_{n-i+1}'\gets \xi$
            }
            \Else
            {
                Set $X_{i}'\gets \xi$
            }
        }
}
Part 2.\texttt{Rebalance the data-set to achieve concentration around $\xi$}\\
\For{ $\kappa\in\{\lfloor \frac{Lnr}{2C} \rfloor,\cdots,1\}$ }
{
    $ \textrm{Right}\gets  \kappa +1-\displaystyle\sum_{i\in [n]}\mathbf{1}\{X'_i-\xi\in[0,\frac{\kappa C }{Ln}]\} $\\
    \While{ $\textrm{Right} >0  $}
    {
        $X'_{\arg\max{X'}}\gets \xi$ \\
        $ \textrm{Right}\gets  \kappa +1-\displaystyle\sum_{i\in [n]}\mathbf{1}\{X'_i-\xi\in[0,\frac{\kappa C }{Ln}]\} $
    }
    $ \textrm{Left} \gets  \kappa +1-\displaystyle\sum_{i\in [n]}\mathbf{1}\{\xi-X'_i\in[0,\frac{\kappa C }{Ln}]\} $\\
    \While{ $\textrm{Left} >0  $}
    {   
        $X'_{\arg\min{X'}}   \gets \xi$\\
        $ \textrm{Left}     \gets  \kappa +1-\displaystyle\sum_{i\in [n]}\mathbf{1}\{\xi-X'_i\in[0,\frac{\kappa C }{Ln}]\} $
    }
}
For all $i=1,\ldots,n$ set $X'_i \leftarrow X'_{\pi(i)}$\\ 
\Return $d_H(X,X')$ 
\caption{\color{black}$\textsc{TypicalHamming}(X=(X_1,\ldots,X_n),\xi)$. In words, our greedy algorithm has three parts.
As preliminary steps, we firstly copy the data-set to a table $X'$, we sort it and compute its left-median. In the case that $\xi$ is not equal with its left median, we compute between which elements $\xi$ lies inside the ordered table $X'$. In the case of ties, we ``position'' $\xi$ to the closest possible to the median.
In Part 1. if $\xi$ is less than the median, then we change the highest possible values of the table to $\xi$ until $\xi$ becomes the median, symmetrically for the other case.
In Part 2. we follow the same strategy like Part 1 but our goal now to maintain the concentration requirements of $\sum_{i\in [n]}\mathbf{1}\{X_i-m(X)\in[0,\frac{\kappa C }{Ln}]\}\ge \kappa +1$ (Left-buckets)
$\sum_{i\in [n]}\mathbf{1}\{m(X)-X_i\in[0,\frac{\kappa C }{Ln}]\}\ge \kappa +1$ (Right-buckets)
for $\kappa\in\{\lfloor \frac{Lnr}{2C} \rfloor,\cdots,1\}$.
}
\label{alg:typicalHamming}
\end{algorithm}
\begin{proof}
Firstly, we will write down the list of the constraints that the above optimization algorithm needs to satisfy:
\[\footnotesize \begin{cases}
\displaystyle\sum_{i\in [n]}\mathbf{1}\{\xi-X_i\in[0,+\infty]\}\ge\frac{n}{2}            &\text{Median-Left Side (m-LS)}\\
\displaystyle\sum_{i\in [n]}\mathbf{1}\{X_i-\xi\in[0,+\infty]\}\ge\frac{n}{2}            &\text{Median-Right Side (m-RS)}\\
\displaystyle\sum_{i\in [n]}\mathbf{1}\{\xi-X_i\in[0,\frac{C }{Ln}]\}\ge2       &\text{$\kappa=1$-Left Side (1-LS)}\\
\displaystyle\sum_{i\in [n]}\mathbf{1}\{X_i-\xi\in[0,\frac{C }{Ln}]\}\ge2        &\text{$\kappa=1$-Right Side (1-RS)}\\
\vdots\\
\displaystyle\sum_{i\in [n]}\mathbf{1}\{\xi-X_i\in[0,\frac{\lfloor \frac{Lnr}{2C} \rfloor C }{Ln}]\}\ge \lfloor \frac{Lnr}{2C} \rfloor+1        &\text{$\kappa=\lfloor \frac{Lnr}{2C} \rfloor$-Left Side ($\lfloor \frac{Lnr}{2C} \rfloor$-LS)}\\
\displaystyle\sum_{i\in [n]}\mathbf{1}\{X_i-\xi\in[0,\frac{\lfloor \frac{Lnr}{2C} \rfloor C }{Ln}]\}\ge \lfloor \frac{Lnr}{2C} \rfloor+1        &\text{$\kappa=\lfloor \frac{Lnr}{2C} \rfloor$-Right Side ($\lfloor \frac{Lnr}{2C} \rfloor$-RS)}
\end{cases}\]  
Our algorithm, as described in Algorithm \ref{alg:typicalHamming}, runs based on the following greedy choice:
\begin{center}\textcolor{black}{\emph{In order to transform $X$ to $X'$ with the least number of changes $\min d_H(X,X')$\\
s.t $\xi$ is a well-concentrated left-median of $X'$,  }}
\[\begin{cases}
\textcolor{black}{\emph{If an LS/RS constraint is violated, we change firstly the further element of $\xi$ from the right/left side}. }\\
\textcolor{black}{\emph{If we change an element of $X$, we set it always $\xi$}. }\\
\end{cases}
\]\end{center}
It is easy to check using classical exchange arguments for both part (A) and part (B) that the algorithm outputs an optimal solution.
\textcolor{black}{Indeed, let $X_{\mathcal{O}}$ be an optimal data-set for the above combinatorial problem, in other words:
\[X_{\mathcal{O}} \in \displaystyle\arg\min_{\begin{matrix}X'\in\mathcal{H},m(X')=\xi\\\xi\in[-R-r/2,R+r/2]\end{matrix}} d_H(X,X')\]
We show that for any optimal data-set $X_{\mathcal{O}}$, different from the the output $X'$ of the Algorithm \ref{alg:typicalHamming},
we can transform $X_{\mathcal{O}}$ to the output $X'$ of the Algorithm \ref{alg:typicalHamming} without increasing the Hamming distance from the input data-set $X$. \textcolor{black}{To describe uniquely $X_{\mathcal{O}}$, it suffices to list its modifications from the set $X$. Let's denote this list }
$\mathcal{M}_{\mathcal{O}}=\{X_{o_1}\to  v_{o_1}',\cdots,X_{o_k}\to  v_{o_k}'\}$.}

Firstly we  mention that the greedy choice of \textcolor{black}{\emph{setting any changed element of $X$ to $\xi$}}
is always as good as any other optimal choice, since $\xi$ is by definition the median of $X_{\mathcal{O}}$.
This holds because $\xi$ is the central value for any of the interval constraints.
Thus we can transform $\mathcal{M}_{\mathcal{O}}$ to $\mathcal{M}_{\mathcal{O}}'=\{X_{o_1}\to  \xi,\cdots,X_{o_k}\to  \xi\}$.

Secondly, it is easy to check because of the nesting nature of the constraints around the median that the greedy choice of  \textcolor{black}{\emph{ changing firstly the further element of $\xi$}} is again always as good as any other optimal choice. Indeed, 
let $
\mathcal{T}
$ be the list of modifications provided by our method, $\mathcal{M}_{\mathcal{T}}=\{X_{T_1}\to  \xi,\cdots,X_{T_\ell}\to \xi\}$. Indeed, if $\mathcal{M}_{\mathcal{T}}$ and $\mathcal{M}_{\mathcal{O}}$ differ in one change we can always follow the greedy's choice since it will satisfy \textbf{at least} the same number of constraints than
any other solution.
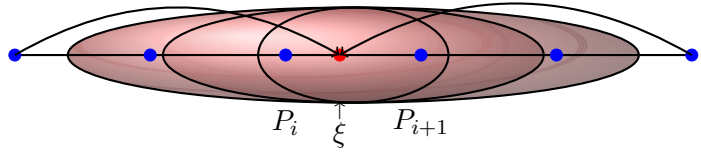
\begin{figure}[H]
\centering
    \begin{tikzpicture}[scale=0.9]
    \shade[ball color = red!40, opacity = 0.1] (1,0) ellipse (40pt and 20pt);
    \shade[ball color = red!40, opacity = 0.2] (1,0) ellipse (80pt and 20pt);
    \shade[ball color = red!40, opacity = 0.3] (1,0) ellipse (120pt and 20pt);
  
    \draw[thick] (1,0) ellipse (40pt and 20pt);
    \draw[thick] (1,0) ellipse (80pt and 20pt);
    \draw[thick] (1,0) ellipse (120pt and 20pt);
    \draw[thick] (-4,0) -- (6,0);
    \foreach \k in {-2, ...,2,3}
    {
        \filldraw[blue] (\k*2,0) circle (2.5pt);
    }
    \node at (0*2,-1) {$P_{i}$};
    \node at (1*2,-1) {$P_{i+1}$};
    \filldraw[red] (0.4*2,0) circle (2.5pt);
    \node at (0.4*2,-1) {$\overset{\uparrow}{\xi}$};
    \draw [->,thick]  (6,0) to [in=30,out=150]  (0.4*2,0)  ;
    \draw [->,thick]  (-4,0) to [in=150,out=30]  (0.4*2,0)  ;
    \end{tikzpicture}
\end{figure}

Thus without never worsening the optimal solution, we can eliminate any differences by changing inductively $X_{\mathcal{O}}$ to the solution of $\textsc{TypicalHamming}(X=(X_1,\ldots,X_n),\xi) $.  Finally the time complexity of the algorithm is $O(n\log n) + O(n)$ for part (A) and $O(Lnr) \times O(n)=O(n^2)$ for part (B). 
\end{proof}


\begin{claim}\label{claim:partition:H}
For a given data-set $X=(X_1,\ldots,X_n)$, there is a $\mathrm{poly}(n)$-time computable ordered collection of $O(n^2)$ points $\mathbb{P}=\{-R-r/2=P_0,P_1,\cdots,P_k,P_{k+1}=R+r/2]\}$ such that for every open interval $I=(P_{i-1},P_{i})$ , it holds that :
\begin{equation*} 
 \text{For any $\xi_1,\xi_2 \in I$ : $\textsc{TypicalHamming}(X,\xi_1)=\textsc{TypicalHamming}(X,\xi_2)$.}
 \end{equation*}
\end{claim}
\begin{proof}
We will start by describing the set $\mathbb{P}$. For each point of our data-set we create
$2 \lfloor \frac{Lnr}{2C} \rfloor+1$ \textbf{anchor}-points. More precisely for the point $X_i$, we set as anchor points:
\[\operatorname{anchors}(X_i)=\left\{X_i\pm\kappa \times \frac{C }{Ln}\Big|\forall \kappa \in \{0,1,\cdots,  \lfloor \frac{Lnr}{2C} \rfloor \} \right\}\] 
Let $\mathbb{P}$ be the union of all the anchor points plus the limit points of the interval $[-R-r/2,R+r/2]$: 
\[\mathbb{P}=\bigcup_{i\in[n]}\operatorname{anchors}(X_i)\cup\{-R-r/2,R+r/2\}\]
It is easy to see, since $Lr\leq 1/2, C>1$ that we need $O(n^2)$ time for its construction:
\begin{algorithm}
$\mathbb{P}\gets\{-R-r/2,R+r/2\}$\\
\For{$i\in\left[n\right]$}
{
    $\mathbb{P}\gets X_i\cup \mathbb{P}$\\
    \For{$\kappa\in\left\{1\cdots \lfloor \frac{Lnr}{2C} \rfloor\right\}$}{
        $\mathbb{P}\gets (X_i+\kappa \times \frac{C }{Ln}) \cup \mathbb{P}$\\
        $\mathbb{P}\gets (X_i-\kappa \times \frac{C }{Ln}) \cup \mathbb{P}$
    }
}
\Return $\mathbb{P}$
\caption{$\textsc{Construction}-\mathbb{P}(X=(X_1,\ldots,X_n)) $}
\label{alg:anchorpoints}
\end{algorithm}

Finally, we show that for any consecutive points of $\mathbb{P}$, let $p_i,p_{i+1}$, the algorithm $\textsc{TypicalHamming}(X,\xi)$ outputs the
same value for any $\xi \in (p_i,p_{i+1})$.
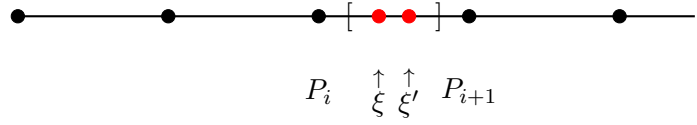
\begin{figure}[H]
\centering
    \begin{tikzpicture}
    \draw[thick] (-4,0) -- (5,0);
    \foreach \k in {-2, ...,2}
    {
        \filldraw (\k*2,0) circle (2.5pt);
    }
    \node at (0*2,-1) {$P_{i}$};
    \node at (0.2*2,0) {$[$};
    \node at (1*2,-1) {$P_{i+1}$};
    \node at (0.8*2,0) {$]$};
    \filldraw[red] (0.4*2,0) circle (2.5pt);
    \filldraw[red] (0.6*2,0) circle (2.5pt);
    \node at (0.4*2,-1) {$\overset{\uparrow}{\xi}$};
    \node at (0.6*2,-1) {$\overset{\uparrow}{\xi'}$};
    \end{tikzpicture}
\end{figure}

Indeed, let's assume any $\xi\neq\xi'$ such that $\xi,\xi'\in(p_i,p_{i+1})$, where $p_i,p_{i+1}\in\mathbb{P}$.
Since between $\xi,\xi'$ we assumed that there is no point of $\mathbb{P}$, \textcolor{black}{ by definition we have that}
\begin{equation}\label{eq:oblivious}
    \mathbf{1}\{p_a\le \xi\le p_b\}=\mathbf{1}\{p_a\le \xi'\le p_b\}\ \ \forall p_a,p_b\in \mathbb{P} \tag{Oblivious Property}
\end{equation}

By inspection of the algorithm, it follows that all the decision of the algorithm have been taken based on the queries of the form:
\[
\begin{Bmatrix}
Q_1(\xi)&:\displaystyle\sum_{i\in [n]}\mathbf{1}\{\xi\ge X_i\}\ge\frac{n}{2} &            
Q_2(\xi)&:\displaystyle\sum_{i\in [n]}\mathbf{1}\{X_i\ge \xi\}\ge\frac{n}{2} \\         
Q_3(\kappa,\xi)&:\displaystyle\sum_{i\in [n]}\mathbf{1}\{X_i \le \xi \le X_i-\frac{\kappa C }{Ln}]\}>\kappa+1 &
Q_4(\kappa,\xi)&:\displaystyle\sum_{i\in [n]}\mathbf{1}\{\frac{\kappa C }{Ln} \le X_i \le \xi \}>\kappa+1
\end{Bmatrix}
\]
More precisely the queries $Q_1(\xi),Q_2(\xi)$ are used in Part(A) and $Q_3(\xi),Q_4(\xi)$ in the Part (B).
Additonally, it is easy to check that any query calculates sums of indicators of the form of \cref{eq:oblivious}.
Therefore for any $\xi,\xi'$ that belong to the open interval $(p_i,p_{i+1})$,
it holds that:
\[Q_i(\xi)=\textsc{True}\Leftrightarrow Q_i(\xi')=\textsc{True} \ \ \forall i\in\{1,2,3,4\}\]
which implies that the output of the optimal algorithm is the same and therefore $\textsc{TypicalHamming}(X,\xi)=\textsc{TypicalHamming}(X,\xi')$
\end{proof}

From the above claim, we showed that $\textsc{TypicalHamming}(X,\xi)$ stays constant inside an open subcover of $[-R-r/2,R+r/2]$.
To complete our proof for \Cref{appendix:definition:TypicalHamming}, we just extend the definition also in the breakpoints $\mathbb{P}$, i.e

\[
\textsc{TypicalHamming}(X,\xi)=
\begin{cases}
\textsc{TypicalHamming}(X,P_0)&\xi\in I_1:=[P_0,P_0]\equiv -R-r/2\\
\textsc{TypicalHamming}(X,\frac{P_0+P_1}{2})&\xi\in I_2:=(P_0,P_1)\\
\textsc{TypicalHamming}(X,P_1)&\xi\in I_3:=[P_1,P_1]\equiv P_1\\
\textsc{TypicalHamming}(X,\frac{P_1+P_2}{2})&\xi\in I_4:=(P_1,P_2)\\
\textsc{TypicalHamming}(X,P_2)&\xi\in I_5:=[P_2,P_2]\equiv P_2\\
\quad\quad\vdots\\
\textsc{TypicalHamming}(X,P_{O(n^2)})&\xi\in  I_{O(n^2)}:=[P_{O(n^2)},P_{O(n^2)}]\equiv R+r/2
\end{cases}
\]

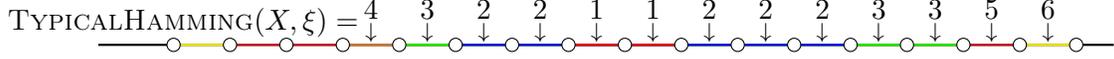
\begin{figure}[!h]
\centering
    \begin{tikzpicture}
    \draw[thick] (-7,0) -- (6.5,0);
    \draw[thick,yellow] (4*1.5,0) -- (-4*1.5,0);
    \draw[thick,purple] (3.5*1.5,0) -- (-3.5*1.5,0);
    \draw[thick,green] (3*1.5,0) -- (-2*1.5,0);
    \draw[thick,blue] (2*1.5,0) -- (-1.5*1.5,0);
    \draw[thick,red] (0.5*1.5,0) -- (-0.5*1.5,0);
    \draw[thick,brown] (-2.5*1.5,0) -- (-2*1.5,0);
    \foreach \k in {-4,-3.5,...,4}
    {
    \draw[fill=white] ({\k*1.5},0) circle (2.5pt);
    }
    \node at (-3.75*1.5-0.25,0.3) {$\textsc{TypicalHamming}(X,\xi)=$};
    \node at (3.75*1.5,0.3) {$\underset{\downarrow}{6}$};
    \node at (3.25*1.5,0.3) {$\underset{\downarrow}{5}$};
    
    \node at (2.75*1.5,0.3) {$\underset{\downarrow}{3}$};
    \node at (2.25*1.5,0.3) {$\underset{\downarrow}{3}$};
    \node at (1.75*1.5,0.3) {$\underset{\downarrow}{2}$};
    \node at (1.25*1.5,0.3) {$\underset{\downarrow}{2}$};
    \node at (.75*1.5,0.3) {$\underset{\downarrow}{2}$};
    \node at (0.25*1.5,0.3) {$\underset{\downarrow}{1}$};

    \node at (-2.25*1.5,0.3) {$\underset{\downarrow}{4}$};
    \node at (-1.75*1.5,0.3) {$\underset{\downarrow}{3}$};
    \node at (-1.25*1.5,0.3) {$\underset{\downarrow}{2}$};
    \node at (-.75*1.5,0.3) {$\underset{\downarrow}{2}$};
    \node at (-0.25*1.5,0.3) {$\underset{\downarrow}{1}$};
    \end{tikzpicture}
    \caption{\emph{\textcolor{black}{Solely} for illustrative purpose, we present a ``feasibly'' representative sketch of the decomposition of our interval $\mathcal{I}=[-B,B]$ based on  $\textsc{TypicalHamming}(X,\xi)$}}
\end{figure}

\end{proof}

\begin{remark}
One can establish more results on the structure of \emph{\textsc{TypicalHamming}} than the one presented in Lemma \ref{appendix:definition:TypicalHamming}; for example \emph{\textsc{TypicalHamming}} can be proven to be left or right semi-continuous depending on the position of the left empirical median $m(X)$, implying an additional structure on the intervals where it is of constant value.  For example, for the simple case of $\mathcal{H}=\mathbb{Z}^n$
and $X=\{1,2,3,4,5\}$ \textsc{TypicalHamming} is constant in $[1,2)[2,3),[3,3],(3,4],(4,5]$.
Nevertheless, in order to avoid the complexity of the description of the rules of the continuity in the potentially more complex typical set $\mathcal{H}$ that we examine in our work,
we just included in the proof of Lemma \ref{appendix:definition:TypicalHamming} separately all the breakpoints as separated intervals where the \emph{\textsc{TypicalHamming}} is trivially of constant value.
\end{remark}
An immediate consequence of \Cref{appendix:definition:TypicalHamming} is that the following definition is well-defined.
\begin{definition}\label{appendix:corollary:TypicalHamming}
For any data-set $X$,  there is a partition of $[-R-r/2,R+r/2]$ to a collection $\mathbb{I}$ of $\operatorname{poly}(n)$ consecutive intervals, $\mathbb{I}=\{ I_1,\cdots,I_{\operatorname{poly}(n)} \}$ for which we denote
\begin{equation*} 
 \text{ $\textsc{TypicalHamming}(X,I):=\textsc{TypicalHamming}(X,\xi)$,}
 \end{equation*}for arbitrary $\xi \in I$.
\end{definition}

We also define the following collection of intervals and points:~
\begin{definition}\label{definition:end-points}
For each $k\in[n]$, let us define $\mathbb{I}_k$ the union of the intervals $I$ where $\textsc{TypicalHamming}(X,I)=k$:  \[\mathbb{I}_k=\displaystyle\bigcup_{I\in\mathbb{I},\textsc{TypicalHamming}(X,I)=k} I\]
Additionally, let us define the limit points $\xi_{k,\inf},\xi_{k,\sup}$ such that $\textsc{TypicalHamming}(X,I)=k$:
\[
\xi_{k,\inf}=\texttt{left-endpoint}(\mathbb{I}_k)=\inf_{\xi\in \mathbb{I}_k} \xi ,\quad \xi_{k,\sup}=\texttt{right-endpoint}(\mathbb{I}_k)=\sup_{\xi\in \mathbb{I}_k} \xi
\]
\end{definition}
\emph{Notice that for every $k\in[n]$ we can compute in $O(n^4)$ time the values of  $\xi_{k,\inf},\xi_{k,\sup}$}.
We are ready now to prove our main technical lemma:
\begin{lemma}[Restated \Cref{lemma:algo:concatenation-laplacian}]\label{appendix:lemma:algo:concatenation-laplacian}
For any given data-set $X$, there is a partition of $[-B,B]$ to a collection $\mathbb{J}$ of $\operatorname{poly}(n)$ consecutive intervals, $\mathbb{J}=\{ J_1,\cdots,J_{\operatorname{poly}(n)} \}$ such that :
\begin{equation*} 
 \text{For any $J_i \in \mathbb{J}\quad :\textsc{UnNormalized}(X,\omega)=
 \exp\left(\alpha_i \omega+\beta_i\right) \quad \forall \omega\in J_i$}
 \end{equation*}
 Moreover, we can compute exactly $\mathbb{J}$ and the constants $(\alpha_i,\beta_i)$ for every $J_i\in\mathbb{J}$ in $O(\operatorname{poly}(n))$ time.\footnote{ Again, for the interested reader, both the number of intervals is at most $7n^2$ and the exact time complexity is $O(n^4)$, but we will keep for simplicity the general expression $\operatorname{poly}(n)$ in our formal statement}
\end{lemma}
\begin{proof}
Using $\textsc{TypicalHamming}$ we can reduce the optimization problem for a given $\omega$ as:
\begin{align*}
\textsc{UnNormalized}(X,\omega)&=\exp\left(\min_{k\in[n]}\displaystyle\inf_{
\begin{matrix}X' \in \mathcal{H}\\m(X')=\xi\\\xi\in [-R-r/2,R+r/2]\\d_H(X,X')=k\end{matrix}}
\left[ \left(\frac{\epsilon}{2} k\right) -\frac{\epsilon}{4} \min\left\{ \mystar  \left|\xi-\omega\right|,\mysquare \right\} \right]\right)\\
&=\exp\left(\min_{k\in [n]} \inf_{\underset{\xi\in[-R-r/2,R+r/2]}{\textsc{TypicalHamming}(X,\xi)=k}}
\left[ \left(\frac{\epsilon}{2} k\right) -\frac{\epsilon}{4} \min\left\{ \mystar  \left|\xi-\omega\right|,\mysquare \right\} \right]\right)\\
(\text{By \Cref{appendix:corollary:TypicalHamming}})
&=\exp\left( \displaystyle\min_{k\in [n]} \displaystyle \inf_{\xi\in \mathbb{I}_k}\left[  \frac{\epsilon}{2}k -\frac{ \epsilon}{4} \min\left\{ \mystar  \left|\xi-\omega\right|,\mysquare \right\} \right]
\right)\\
(\text{By \Cref{definition:end-points}})
&=\exp\left(\frac{\epsilon}{2}\min_{k\in [n]}\underbrace{\left\{ k - 
   \frac{1}{2}\cdot \mystar \displaystyle   \min \left\{  \max_{p\in\{\xi_{k,\inf},\xi_{k,\sup}\}} \left|p-\omega\right|,3Cr \right\} \right\}}_{h_k(\omega)}\right)\Leftrightarrow\\
\end{align*}
\begin{equation}
\textsc{UnNormalized}(X,\omega)=\exp\left(\frac{\epsilon}{2}\min_{k\in [n]}h_k(\omega)\right) \label{eq:un-normalized-equals-min-of-hs}
\end{equation}
We now focus on understanding the structure of the functions $h_k(\omega)$ 
\begin{figure}[h!]
    \centering
\begin{tikzpicture}[scale=0.9]
\def\normaltwo{\x,{5/exp(min(abs(\x-3),2.5)/2)}}
\def\uniform{\x,{5/exp(2.5/2)}}
\def\pulseA{\x,{7-min( max(abs(\x-3.5),abs(\x+4)), 5.5)}}
\def\pulseB{\x,{7-min( max(abs(\x-3.5),0), 5.5)}}
\def\pulseC{\x,{7-min( max(0,abs(\x+4)), 5.5)}}
\def\xmax{8}
\def\xcen{-0.25}
\def\xmin{-8}
\def\xlapmax{1.5}
\def\xlapmin{-2}
\def\pulseleft{{7-min( max(abs(\xlapmin-3.5),abs(\xlapmin+4)), 5.5)}}
\def\pulseright{{7-min( max(abs(\xlapmax-3.5),abs(\xlapmax+4)), 5.5)}}
\def\pulsecen{{7-min( max(abs(\xcen-3.5),abs(\xcen+4)), 5.5)}}

\draw[thick,color=blue,domain=\xmin:\xmax,samples=100] plot (\pulseA) node[right] {};
\draw[dashed,color=blue,domain=\xmin:\xmax,samples=100] plot (\pulseB) node[right] {};
\draw[dashed,color=blue,domain=\xmin:\xmax,samples=100] plot (\pulseC) node[right] {};

\draw[-,dashed] (\xmin,0) -- (\xmin,\pulseleft) node[left] {$-B$};
\draw[-,dashed] (\xmax,0) -- (\xmax,\pulseright) node[right] {$B$};
\draw[-,dashed] (3.5,0) -- (3.5,6) node[above] {$\xi_{\sup,k}$};
\draw[-,dashed] (-4,0) -- (-4,6) node[above] {$\xi_{\inf,k}$};
\draw[-,dashed] (\xcen,0) -- (\xcen,\pulsecen) node[above] {$\frac{\xi_{\inf,k}+\xi_{\sup,k}}{2}$};
\draw[-,dashed] (\xlapmax,0) -- (\xlapmax,\pulseright) node[above] {};
\draw[-,dashed] (\xlapmin,0) -- (\xlapmin,\pulseleft) node[above] {};
\node at (\xlapmax,-1) [above] {$\xi_{\inf,k}+3Cr$};
\node at  (\xlapmin,-1) [above] {$\xi_{\sup,k}-3Cr$};

\draw[->,thick] (0,0) -- (\xmin-1,0) node[right] {};
\draw[->,thick] (0,0) -- (\xmax+1,0) node[right] {};
\end{tikzpicture}
\caption{An example of $h_{i}$}
\label{fig:H_is}
\end{figure}
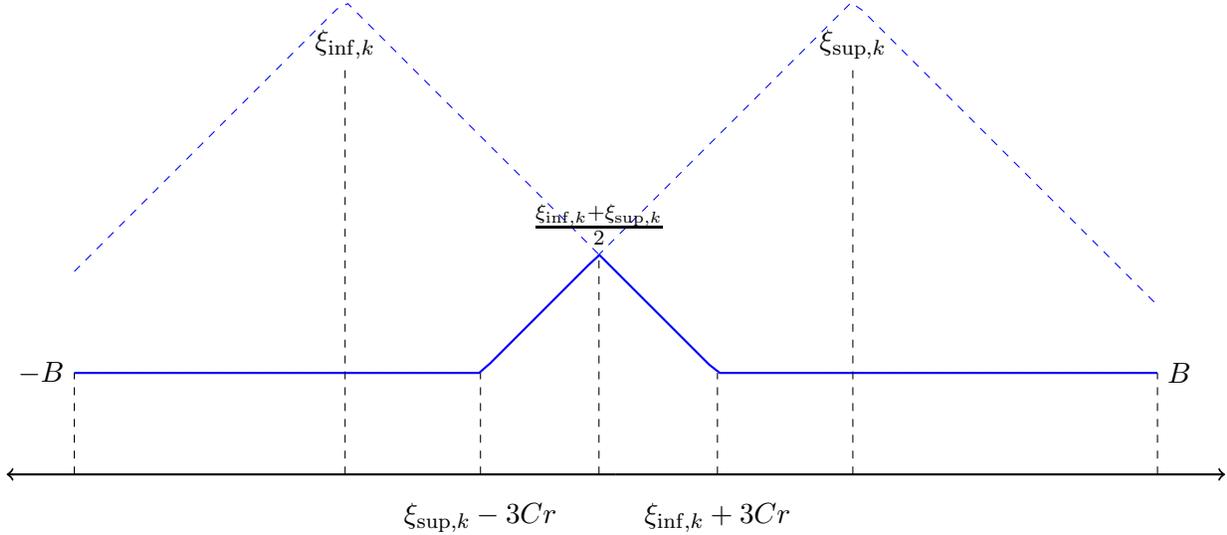

By inspection (See \Cref{fig:H_is}),
it is not difficult to check that every $h_{k}(\omega)$ 
is a piecewise linear function with at most four 
\footnote{The existence/size of the triangular pulse depends actually
by the distance of $\xi_{\inf,k},\xi_{\sup,k}$. Similarly, the existence/duration of the constant pulse depends on the relation of $\{\xi_{\inf,k}+3Cr,\xi_{\sup,k}-3Cr\}$ with $\{-B,B\}$.}
linear pieces
\footnote{The constant pulse can be achieved by setting the corresponding $\alpha_i^{[k]}$ zero.}:
\[\forall k\in[n]:
h_k(\omega)=\begin{cases}
\alpha_{1}^{[k]}\omega+\beta_{1}^{[k]} & \omega\in[\rho_1^{[k]},\rho_2^{[k]}]\equiv[-B,\xi_{\sup,k}-3Cr]\\\\
\alpha_{2}^{[k]}\omega+\beta_{2}^{[k]} & \omega\in[\rho_2^{[k]},\rho_3^{[k]}]\equiv[\xi_{\sup,k}-3Cr,\frac{\xi_{\inf,k}+\xi_{\sup,k}}{2}]\\\\
\alpha_{3}^{[k]}\omega+\beta_{3}^{[k]} & \omega\in[\rho_3^{[k]},\rho_4^{[k]}]\equiv[\frac{\xi_{\inf,k}+\xi_{\sup,k}}{2},\xi_{\inf,k}+3Cr]\\\\
\alpha_{4}^{[k]}\omega+\beta_{4}^{[k]} & \omega\in[\rho_4^{[k]},\rho_5^{[k]}]\equiv[\xi_{\inf,k}+3Cr,B]
\end{cases}
\]
where $BP^{[k]}=\{\rho_1^{[k]},\cdots,\rho_5^{[k]}\}$ are the breakpoints of $h_k$ in which the function changes its expression. 
Since each $h_k(\omega)$ is a piece-wise function with at most 4 pieces, its minimum is also a piece-wise linear function.
Consequently, for any given data-set $X$ there is a partition of $[-B,B]$ to a collection $\mathbb{J}$ consecutive intervals, $\mathbb{J}=\{ J_1,\cdots,J_{|\mathbb{J}|} \}$ such that :
\begin{equation*} 
 \text{For any $J_i \in \mathbb{J}\quad :\min_{k\in [n]}h_k(\omega)=
 \alpha_i \omega+\beta_i \quad \forall \omega\in J_i$}
 \end{equation*}
or equivalently using \cref{eq:un-normalized-equals-min-of-hs}:
\begin{equation*} 
 \text{For any $J_i \in \mathbb{J}\quad :\textsc{UnNormalized}(X,\omega)=
 \exp\left(\alpha_i \omega+\beta_i\right) \quad \forall \omega\in J_i$}
 \end{equation*}
To conclude the proof of our lemma it suffices to show that the size of the collection $\mathbb{J}$ is polynomial and to describe 
a $\operatorname{poly}(n)$-time procedure that computes the  piece-wise linear function $\min_{k\in [n]}h_k$.

Observe that to calculate $\min_{k\in [n]}h_k(\omega)$, it suffices to know the relative ordering of $h_1(\omega),\cdots,h_{n}(\omega)$ for each $\omega\in[-B,B]$.
By continuity of $h_{k}$, the relative order of a pair $h_k,h_\ell$ can not change inside an interval $J$ where there is no solution of the equation $h_k(\omega)=h_\ell(\omega)$.
Let us have $\Pi^{[k,\ell]}=\pi_{1}^{[k,\ell]},\cdots,\pi_{c_{k,\ell}}^{[k,\ell]}$ be the ordered set of solutions of $h_k(\omega)=h_\ell(\omega)$ for every $k\neq\ell$.
Let us denote denote also as $\Pi=\bigcup_{k\neq \ell} \Pi^{[k,\ell]}\cup\{-B,B\}$. Observe that in any interval $J$ of consecutive points of $\Pi$,
the relative ordering of $h_k(\omega),h_{\ell}(\omega)$ for any  $k,\ell\in[n]$ does not change inside $J$.
Thus the minimizing function $h_\star$ does not change inside the interval $J$. We can further subdivide the interval $J$, if it is needed into sub-intervals based on its breakpoints
where $h_{\star}$ is linear.

We present the algorithm to construct the set $\mathbb{J}$:
\begin{algorithm}
Sort $\Phi=\displaystyle\bigcup_{k\neq \ell} \Pi^{[k,\ell]}\cup\{-B,B\}\cup \bigcup_{k} BP^{[k]}$\\
\texttt{//}$\Phi=\{\phi_0=-B\le\pi_1\le\cdots\le\phi_m\le\phi_{m+1}=B\}$\\
$\mathbb{J}\gets\emptyset$\\
    \For{$i\in \left\{1,\cdots, m+1\right\}$}{
    Set $J_{i}=[\phi_{i-1},\phi_i]$\\
    $\mathbb{J}\gets\mathbb{J}\cup J_i$
    }
\Return $\mathbb{J}$
\caption{$\textsc{Construction}-\mathbb{J}(X=(X_1,\ldots,X_n)) $}
\label{alg:J-construction}
\end{algorithm}
    \begin{algorithm}[h!]
    Sort $\Phi=\displaystyle\bigcup_{k\neq \ell} \Pi^{[k,\ell]}\cup\{-B,B\}\cup \bigcup_{k} BP^{[k]}$\\
    \texttt{//}$\Phi=\{\phi_0=-B\le\pi_1\le\cdots\le\phi_m\le\phi_{m+1}=B\}$\\
    Compute $h_1(-B),\cdots,h_n(-B)$\\
    Keep a list of the current order $h_{k}$: $\Sigma=\{h_{i_1}\le\cdots\le h_{i_n}\}$\\
    Keep a list of the current liner form of $h_{k}$: $T=\{h_{1}:(\alpha_1^{[1]},\beta_1^{[1]}) \cdots h_{n}:(\alpha_1^{[n]},\beta_1^{[n]}) \}$\\
    \For{$i\in \left\{1,\cdots, m+1\right\}$}{
        Set $J_{i}=[\phi_{i-1},\phi_i]$\\
        Set $(\alpha_i,\beta_i)\gets T[\min h_{k}]$\\
        \If{$\phi_i\in \Pi$ \texttt{// The next point is crossing point of two functions}}
        {
            Update the relative order $\Sigma$\\
        }
        \If{$\phi_i=\rho_{j}^{[k]}$ \texttt{// The next point is some of the 4 breakpoints of $h_k$}}
        {
            Update $T[h_k]=(\alpha_j^{[k]},\beta_j^{[k]})$\\
        }

    }
    \Return $\{(J_i,\alpha_i,\beta_i)\}_m$
    \caption{$\textsc{Construction}-(\mathbb{J},\pmb{\alpha},\pmb{\beta})(X=(X_1,\ldots,X_n)) $}
\label{alg:J-a-b-construction}
\end{algorithm}
\begin{itemize}
    \item To evaluate the size of $\mathbb{J}$, it suffices to understand the size of $\Pi$ and $\bigcup_{k} BP^{[k]}$.
    On the one hand, the size of $\bigcup_{k} BP^{[k]}$ is at most $4n$, since we have at most 4 pieces in $h_{k}$.
    On the other hand, $\Pi$ includes all the solutions of $h_k(\omega)=h_\ell(\omega)$ for every $k\neq\ell$.
    This size is upper bound by $\binom{n}{2} \times \max_{k\neq\ell} \Pi^{[k,\ell]}$. Observe that to solve $h_k(\omega)=h_\ell(\omega)$, we need to solve at most 16 linear equations
    \footnote{
    In the general case for two piece-wise linear functions with $4$ pieces, the maximum number of intersections Leveraging their specific ``triangular-pulse'' form,it is easy to verify that the maximum number  is $4$ and with an even more detailed case study the maximum number of intersections is actually $2$, because the 
slopes of the  ``triangular-pulse'' is the same for all $h_k$ functions.}, each having at most one solution, since the pairs $(\alpha_{j}^{k},\beta_{j}^{k})$ are unique. Thus $\mathbb{J}=O(n^2)+O(n)=O(n^2)$.
    \item The time complexity of construction $\mathbb{J}$ is the sum of the elapsed time to compute all the $4n$ breakpoints $BP^{[k]}$, which can be done in $O(n^4)$ and
    compute the solutions of $O(n^2)$ linear equations and sorting their union $O(n^2\log n)$. Thus, indeed we can construct $\mathbb{J}$ in $\operatorname{poly}(n)$ time.
    \item To compute the constants $(\alpha_i,\beta_i)$ : $\min_{k}h_k(\omega)=\alpha_i\omega+\beta_i \ \ \forall \omega\in J_i$, we can start from left to right.
    We compute the relative ordering of $h_1(-B),\cdots,h_n(-B)$. For each solution $h_k(\omega)=h_\ell(\omega)$, we update the relative ordering and for each
    breakpoint of $h_{k}$ we update its current linear form.
\end{itemize}
\end{proof}
\begin{lemma}
For a given a date-set $X=(X_1,\cdots,X_n)$ and its median $m(X)$, there exists a $\operatorname{poly}(n)-$time protocol that generates a sample from $\mathcal{D}_{\text{general}}$.
\end{lemma} 
\begin{proof}
We use the decomposition described in \Cref{appendix:lemma:algo:concatenation-laplacian}.
In the first round we  sample from a discrete distribution on a $\operatorname{poly}(n)$-cardinality domain  to decide among the regions $J_1,\cdots,J_{\operatorname{poly}(n)} $. In the second round we  simply apply conditional sampling from the corresponding either exponential or uniform  which will
be truncated on the region which was the outcome of the first round.
More specifically:

\begin{algorithm}[h!]
Execute $\textsc{Construction}-(\mathbb{J},\pmb{\alpha},\pmb{\beta})(X=(X_1,\ldots,X_n))$\\
Let $\{(J_i,\alpha_i,\beta_i)\}_m$ be the $\operatorname{poly}(n)$-decomposition of $\mathcal{I}=[-B,B]$ as described in \cref{appendix:lemma:algo:concatenation-laplacian}\\
//\texttt{  \text{For any }}$J_i \quad :\textsc{UnNormalized}(X,\omega)=
 \exp\left(\alpha_i \omega+\beta_i\right) \quad \forall \omega\in J_i$ \\
 \For{$i\in[m]$}
 {
 $p_i=\displaystyle\int_{J_i}\exp\left(\alpha_i \omega+\beta_i\right)\mathrm{d}\omega$
 }
Let $p=\sum_{i\in [m]} p_i$\\
Toss a $m$-nary coin $c:=\{1,\cdots,m\}$ with probability $(\tfrac{p_1}{p},\cdots,\tfrac{p_m}{p})$, correspondingly.\\
\If{$c$ outputs $\mathcal{C}$}
{
      $s \gets \text{Sample from }\textsc{TruncatedExponential}\left[\alpha=\alpha_i,\beta=\beta_i,I=J_i\right]$ 
}
\Return $s$
\caption{$\text{Sample from }\mathcal{D}_{\text{general}}$}
\label{alg:GeneralEstimator}
\end{algorithm}
The correctness of our estimator is derived by \Cref{appendix:lemma:algo:concatenation-laplacian}.As we showed, the size of $J=\{(J_i,\alpha_i,\beta_i)\}_m$ is  $\operatorname{poly}(n)$ and we need $\operatorname{poly}(n)$ time to compute $\{(J_i,\alpha_i,\beta_i)\}_m$. Our final step is to toss a $m$-nary coin which needs
an extra $\Theta(m)=\operatorname{poly}(n)$ time.
\end{proof}
\subsubsection{Average-case Time Complexity of \textsc{PrivateMedian}  }\label{appendix:average-case}
In this section we discuss about the average-case time complexity of the algorithm $\textsc{PrivateMedian}(X=(X_1,\ldots,X_n)) $ which implements the algorithm $\mathcal{A}$ defined in \eqref{alg:general}. We consider running the algorithm for $C=C_0 \log n$ for some appropriate constant $C_0>0.$ Note that this is different from the worst-case analysis section where it discusses the algorithm for a sufficiently large constant $C>0$ and establishes its rate-optimality. We establish the following Corollary establishing that the algorithm remains rate-optimal up to logarithmic factors for this value of $C$. 
\begin{coro}\label{cor:upper_bound}
Suppose $C=C_0 \log n$ for some constant $C_0>0.$ Suppose also $\epsilon \in (0,1)$, $\mathcal{D}$ is an admissible distribution and $\mathcal{A}$ is the $\epsilon$-differentially private algorithm defined in \eqref{alg:general} for this value of $C$. Then for any $\alpha \in (0,r)$ and $ \beta \in (0,1)$ for some $$ n=\tilde{O}\left( \frac{\log \left( \frac{1}{\beta}\right)}{L^2 \alpha^2}+\frac{ \log \left(\frac{1}{\beta}\right)}{\epsilon L\alpha }+\frac{\log \left(\frac{R}{ \alpha}+1\right)}{\epsilon Lr}\right)$$ it holds $\quad\quad \quad \mathbb{P}_{X_1,X_2,\ldots,X_n \overset{iid}{\sim} \mathcal{D}} [ |\mathcal{A}(X_1,\ldots,X_n)-m\left(\mathcal{D}\right)|  \geq \alpha] \leq \beta.$
\end{coro}The corollary follows directly by applying for this value of $C$ the Theorem \ref{thm:upper_bound0} (a slight generalization of Theorem \ref{thm:upper_bound}, using a constant $C>0$ possible scaling with $n$) and using the basic asymptotic formula that for $A=\Omega(1),$ it holds $n/\log n =\Omega \left(A\right)$ if and only if $n =\Omega \left(A \log A\right)$.  In particular, Corollary \ref{cor:upper_bound} combined with the results of Section \ref{sec:LB} allows us to conclude that the algorithm where $C$ scaling logarithmically with $n$ remains rate-optimal up to logarithmic factors.

We now show that for appropriate tuning of the parameter $C_0>0$ and under a weak bound on the sampling size $n$, satisfied for instance by the third term in the almost optimal rate of Corollary \ref{cor:upper_bound}, the algorithm runs in average-case in $\tilde{O}(n)$ time.

\begin{theorem}
There exists constants $C_0,D_0>0$ such that if $C=C_0 \log n$ and $T$ is the termination of time of $\textsc{PrivateMedian}(X=(X_1,\ldots,X_n)) $ for this choice of $C>0$ the following holds. If $n \geq D_0 \frac{1}{Lr} \log  (\frac{1}{Lr}),$ then \begin{equation*}
    \mathbb{E}\left[T\right]=O\left(n \log n\right).
\end{equation*}\end{theorem}

\begin{proof}
We have by the law of total expectation,
\begin{align}
      \mathbb{E}\left[T\right]&= \mathbb{P}\left[X \in \mathcal{H} \right] \mathbb{E}\left[T | X \in \mathcal{H}\right]  +\mathbb{P}\left[X \not \in \mathcal{H} \right]  \mathbb{E}\left[T | X \not \in \mathcal{H}\right] \nonumber \\
      & \leq   \mathbb{E}\left[T | X \in \mathcal{H}\right]  +\mathbb{P}\left[X \not \in \mathcal{H} \right]  \mathbb{E}\left[T | X \not \in \mathcal{H}\right] \label{eq:totalprob}.
\end{align}We now consider each term separately.

In the case where $X \in \mathcal{H}$ we need to check that $X \in \mathcal{H}$ to confirm the inclusion to the set $\mathcal{H}$ and then sample from the $\mathcal{D}_{\text{restricted}}$ distribution. For the first part, standard $O(n \log n)$ counting binary searches and locating the value of the median $m(X)$ suffices to certify that $X \in \mathcal{H}.$ Furthermore, given Lemma \ref{lem:restr_time}, since we have located the median $m(X)$, it takes $O(1)$-time to sample from $\mathcal{D}_{\text{restricted}}.$ In conclusion \begin{equation}\label{first_trivial}
     \mathbb{E}\left[T | X \in \mathcal{H}\right] =O\left(n \log n\right).
\end{equation}

For the second part, since the algorithm runs in worst case polynomial-time there is a constant $M>0$ such that \begin{equation}\label{eq:poly}
    \mathbb{E}\left[T | X \not \in \mathcal{H}\right] =O\left(n^M\right).
\end{equation} Now we focus on bounding $\mathbb{P}\left[X \not \in \mathcal{H} \right].$ 
Notice that by identical reasoning as the derivation of inequality (14) of \cite{Subgaussian} in the proof of Lemma 3 in \cite{Subgaussian} we have for $t=r/2 \in [0,r]$ 
\begin{equation*} \label{eq:marco2} \mathbb{P}\left(|m(X)-m(\mathcal{D})| \geq r/2\right) \leq 2e^{-nL^2r^2/8}.
\end{equation*}Since $|m\left(\mathcal{D}\right)| \leq R$ we conclude 
\begin{equation} \label{eq:marco22} \mathbb{P}\left( m(X) \not \in [-R-r/2,R+r/2]\right) \leq 2e^{-nL^2r^2/8}.
\end{equation}Furthermore, notice that since $C=C_0 \log n$ by assuming $n$ is bigger than a constant, Lemma \ref{lemma:typical_set_general} can be applied with $T=1$. Note that for $T=1$, all counting constraints of $\mathcal{H}$ are considered and therefore, combined with \eqref{eq:marco22} we have \begin{equation*}
    \mathbb{P}\left[X \not \in \mathcal{H} \right] \leq O\left(  C_0\log n\exp\left( -\frac{C_0\log n}{8}\right)+e^{-\Theta \left( L^2r^2n\right)}\right).
\end{equation*}Combining with \eqref{eq:poly} we have 
\begin{equation*}
    \mathbb{P}\left[X \not \in \mathcal{H} \right]\mathbb{E}\left[T | X \not \in \mathcal{H}\right]  \leq O\left( n^M C_0\log n\exp\left( -\frac{C_0\log n}{8}\right)+n^M e^{-\Theta \left( L^2r^2n\right)}\right).
\end{equation*}Now using standard asymptotics there are  $C_0,D_0>0$ sufficiently large constants dependent on $M$ for which if $n \geq D_0 \frac{1}{Lr} \log  (\frac{1}{Lr}),$ then 
\begin{equation*}
    \mathbb{P}\left[X \not \in \mathcal{H} \right]\mathbb{E}\left[T | X \not \in \mathcal{H}\right]  \leq O\left(1\right).
\end{equation*}Combining with \eqref{first_trivial} and \eqref{eq:totalprob} we conclude
\begin{align}
      \mathbb{E}\left[T\right] \leq O\left(n \log n\right).
      \end{align}This completes the proof of the Theorem.
\end{proof}
\section*{Acknowledgements}

\textcolor{black}{I.Z. is grateful to Adam Smith for bringing the private median estimation problem to his attention.  E.V would like to thank Lampros Flokas, Ana-Andrea Stoica and Myrto Pantelaki for the helpful discussions at camera-ready stage of this project.}

\begin{flushleft}
\bibliographystyle{alpha}
\bibliography{references}

\newcommand{\etalchar}[1]{$^{#1}$}
\begin{thebibliography}{KRWY16}

\bibitem[AD20]{Duch20near}
Hilal Asi and John~C. Duchi.
\newblock Near instance-optimality in differential privacy, 2020.

\bibitem[BA19]{Subgaussian}
Victor{-}Emmanuel Brunel and Marco Avella{-}Medina.
\newblock Differentially private sub-gaussian location estimators.
\newblock {\em arXiv}, 2019.

\bibitem[BA20]{ptr20}
Victor{-}Emmanuel Brunel and Marco Avella{-}Medina.
\newblock Propose, test, release: Differentially private estimation with high
  probability.
\newblock {\em CoRR}, abs/2002.08774, 2020.

\bibitem[BBDS13]{BlockiBDS13}
Jeremiah Blocki, Avrim Blum, Anupam Datta, and Or~Sheffet.
\newblock Differentially private data analysis of social networks via
  restricted sensitivity.
\newblock In Robert~D. Kleinberg, editor, {\em Innovations in Theoretical
  Computer Science, {ITCS} '13, Berkeley, CA, USA, January 9-12, 2013}, pages
  87--96. {ACM}, 2013.

\bibitem[BCS15]{BorgsNeurIPS15}
Christian Borgs, Jennifer~T. Chayes, and Adam Smith.
\newblock Private graphon estimation for sparse graphs.
\newblock In {\em Proceedings of the 28th International Conference on Neural
  Information Processing Systems - Volume 1}, NIPS’15, page 1369–1377,
  Cambridge, MA, USA, 2015. MIT Press.

\bibitem[BCSZ18a]{borgs2018private}
Christian Borgs, Jennifer Chayes, Adam Smith, and Ilias Zadik.
\newblock Private algorithms can always be extended.
\newblock {\em arXiv preprint arXiv:1810.12518}, 2018.

\bibitem[BCSZ18b]{borgs2018revealing}
Christian Borgs, Jennifer Chayes, Adam Smith, and Ilias Zadik.
\newblock Revealing network structure, confidentially: Improved rates for
  node-private graphon estimation.
\newblock In {\em 2018 IEEE 59th Annual Symposium on Foundations of Computer
  Science (FOCS)}, pages 533--543. IEEE, 2018.

\bibitem[BCSZ18c]{BorgsCSZ18}
Christian Borgs, Jennifer~T. Chayes, Adam~D. Smith, and Ilias Zadik.
\newblock Revealing network structure, confidentially: Improved rates for
  node-private graphon estimation.
\newblock In Mikkel Thorup, editor, {\em 59th {IEEE} Annual Symposium on
  Foundations of Computer Science, {FOCS} 2018, Paris, France, October 7-9,
  2018}, pages 533--543. {IEEE} Computer Society, 2018.

\bibitem[BKSW19]{PHS}
Mark Bun, Gautam Kamath, Thomas Steinke, and Zhiwei~Steven Wu.
\newblock Private hypothesis selection.
\newblock {\em CoRR}, abs/1905.13229, 2019.

\bibitem[CD20]{CummingsD20}
Rachel Cummings and David Durfee.
\newblock Individual sensitivity preprocessing for data privacy.
\newblock In Shuchi Chawla, editor, {\em Proceedings of the 2020 {ACM-SIAM}
  Symposium on Discrete Algorithms, {SODA} 2020, Salt Lake City, UT, USA,
  January 5-8, 2020}, pages 528--547. {SIAM}, 2020.

\bibitem[CDC19]{mediancdc}
CDC.
\newblock {\em Principles of Epidemiology in Public Health Practice, Third
  Edition An Introduction to Applied Epidemiology and Biostatistics}, 2019.

\bibitem[CKS20]{DiscreteGaussian}
Cl{\'{e}}ment~L. Canonne, Gautam Kamath, and Thomas Steinke.
\newblock The discrete gaussian for differential privacy.
\newblock {\em CoRR}, abs/2004.00010, 2020.

\bibitem[Dav95]{david1995first}
Herbert~A David.
\newblock First (?) occurrence of common terms in mathematical statistics.
\newblock {\em The American Statistician}, 49(2):121--133, 1995.

\bibitem[DCKP06]{dicker2006principles}
Richard~C Dicker, Fatima Coronado, Denise Koo, and R~Gibson Parrish.
\newblock Principles of epidemiology in public health practice; an introduction
  to applied epidemiology and biostatistics.
\newblock {\em Centers for Disease Control and Prevention and others}, 2006.

\bibitem[DE13]{DankarE13}
Fida~Kamal Dankar and Khaled~El Emam.
\newblock Practicing differential privacy in health care: {A} review.
\newblock {\em Trans. Data Priv.}, 6(1):35--67, 2013.

\bibitem[DG07]{davies2007breakdown}
PL~Davies and Ursula Gather.
\newblock The breakdown point—examples and counterexamples.
\newblock {\em REVSTAT Statistical Journal}, 5(1):1--17, 2007.

\bibitem[DL09]{DworkL09}
Cynthia Dwork and Jing Lei.
\newblock Differential privacy and robust statistics.
\newblock In Michael Mitzenmacher, editor, {\em Proceedings of the 41st Annual
  {ACM} Symposium on Theory of Computing, {STOC} 2009, Bethesda, MD, USA, May
  31 - June 2, 2009}, pages 371--380. {ACM}, 2009.

\bibitem[DMNS06]{DworkMNS06}
Cynthia Dwork, Frank McSherry, Kobbi Nissim, and Adam~D. Smith.
\newblock Calibrating noise to sensitivity in private data analysis.
\newblock In Shai Halevi and Tal Rabin, editors, {\em Theory of Cryptography,
  Third Theory of Cryptography Conference, {TCC} 2006, New York, NY, USA, March
  4-7, 2006, Proceedings}, volume 3876 of {\em Lecture Notes in Computer
  Science}, pages 265--284. Springer, 2006.

\bibitem[DR14]{DworkR14}
Cynthia Dwork and Aaron Roth.
\newblock The algorithmic foundations of differential privacy.
\newblock {\em Foundations and Trends in Theoretical Computer Science},
  9(3-4):211--407, 2014.

\bibitem[DSSU17]{dwork2017exposed}
Cynthia Dwork, Adam Smith, Thomas Steinke, and Jonathan Ullman.
\newblock Exposed! a survey of attacks on private data.
\newblock {\em Annual Review of Statistics and Its Application}, 4(1):61--84,
  2017.

\bibitem[DWJ16]{DuchiWJ16}
John~C. Duchi, Martin~J. Wainwright, and Michael~I. Jordan.
\newblock Minimax optimal procedures for locally private estimation.
\newblock {\em CoRR}, abs/1604.02390, 2016.

\bibitem[Dwo06]{Dwork06}
Cynthia Dwork.
\newblock Differential privacy.
\newblock In Michele Bugliesi, Bart Preneel, Vladimiro Sassone, and Ingo
  Wegener, editors, {\em Automata, Languages and Programming, 33rd
  International Colloquium, {ICALP} 2006, Venice, Italy, July 10-14, 2006,
  Proceedings, Part {II}}, volume 4052 of {\em Lecture Notes in Computer
  Science}, pages 1--12. Springer, 2006.

\bibitem[FLJ{\etalchar{+}}14]{FredriksonLJLPR14}
Matthew Fredrikson, Eric Lantz, Somesh Jha, Simon Lin, David Page, and Thomas
  Ristenpart.
\newblock Privacy in pharmacogenetics: An end-to-end case study of personalized
  warfarin dosing.
\newblock In Kevin Fu and Jaeyeon Jung, editors, {\em Proceedings of the 23rd
  {USENIX} Security Symposium, San Diego, CA, USA, August 20-22, 2014}, pages
  17--32. {USENIX} Association, 2014.

\bibitem[G{\etalchar{+}}82]{galton1882report}
Francis Galton et~al.
\newblock Report of the anthropometric committee.
\newblock In {\em Report of the 51st Meeting of the British Association for the
  Advancement of Science}, volume 1881, pages 245--260, 1882.

\bibitem[GKK{\etalchar{+}}20]{LocallyPrivateHypothesisSelection}
Sivakanth Gopi, Gautam Kamath, Janardhan Kulkarni, Aleksandar Nikolov,
  Zhiwei~Steven Wu, and Huanyu Zhang.
\newblock Locally private hypothesis selection.
\newblock {\em CoRR}, abs/2002.09465, 2020.

\bibitem[HR81]{huberbreakdown}
PJ~Huber and EM~Ronchetti.
\newblock Breakdown point.
\newblock {\em Robust Statistics}, page~8, 1981.

\bibitem[HRRS11]{hampel2011robust}
Frank~R Hampel, Elvezio~M Ronchetti, Peter~J Rousseeuw, and Werner~A Stahel.
\newblock {\em Robust statistics: the approach based on influence functions},
  volume 196.
\newblock John Wiley \& Sons, 2011.

\bibitem[Hub81]{Huber81}
Peter~J. Huber.
\newblock {\em Robust Statistics}.
\newblock Wiley Series in Probability and Statistics. Wiley, 1981.

\bibitem[KLSU19]{Kamath0SU19}
Gautam Kamath, Jerry Li, Vikrant Singhal, and Jonathan Ullman.
\newblock Privately learning high-dimensional distributions.
\newblock In Alina Beygelzimer and Daniel Hsu, editors, {\em Conference on
  Learning Theory, {COLT} 2019, 25-28 June 2019, Phoenix, AZ, {USA}}, volume~99
  of {\em Proceedings of Machine Learning Research}, pages 1853--1902. {PMLR},
  2019.

\bibitem[KNRS13]{KasiviswanathanNRS13}
Shiva~Prasad Kasiviswanathan, Kobbi Nissim, Sofya Raskhodnikova, and Adam~D.
  Smith.
\newblock Analyzing graphs with node differential privacy.
\newblock In Amit Sahai, editor, {\em Theory of Cryptography - 10th Theory of
  Cryptography Conference, {TCC} 2013, Tokyo, Japan, March 3-6, 2013.
  Proceedings}, volume 7785 of {\em Lecture Notes in Computer Science}, pages
  457--476. Springer, 2013.

\bibitem[KRWY16]{Kearns913}
Michael Kearns, Aaron Roth, Zhiwei~Steven Wu, and Grigory Yaroslavtsev.
\newblock Private algorithms for the protected in social network search.
\newblock {\em Proceedings of the National Academy of Sciences},
  113(4):913--918, 2016.

\bibitem[KSSU19]{Mixtures}
Gautam Kamath, Or~Sheffet, Vikrant Singhal, and Jonathan Ullman.
\newblock Differentially private algorithms for learning mixtures of separated
  gaussians.
\newblock {\em CoRR}, abs/1909.03951, 2019.

\bibitem[KSU20]{HeavyTailed}
Gautam Kamath, Vikrant Singhal, and Jonathan Ullman.
\newblock Private mean estimation of heavy-tailed distributions.
\newblock {\em CoRR}, abs/2002.09464, 2020.

\bibitem[KV18]{KarwaV18}
Vishesh Karwa and Salil~P. Vadhan.
\newblock Finite sample differentially private confidence intervals.
\newblock In Anna~R. Karlin, editor, {\em 9th Innovations in Theoretical
  Computer Science Conference, {ITCS} 2018, January 11-14, 2018, Cambridge, MA,
  {USA}}, volume~94 of {\em LIPIcs}, pages 44:1--44:9. Schloss Dagstuhl -
  Leibniz-Zentrum f{\"{u}}r Informatik, 2018.

\bibitem[LXJ{\etalchar{+}}20]{liu2020local}
Peng Liu, YuanXin Xu, Quan Jiang, Yuwei Tang, Yameng Guo, Li-e Wang, and
  Xianxian Li.
\newblock Local differential privacy for social network publishing.
\newblock {\em Neurocomputing}, 391:273--279, 2020.

\bibitem[MEO13]{MalinEO13}
Bradley~A. Malin, Khaled~El Emam, and Christine~M. O'Keefe.
\newblock Biomedical data privacy: problems, perspectives, and recent advances.
\newblock {\em J. Am. Medical Informatics Assoc.}, 20(1):2--6, 2013.

\bibitem[RG20]{rigaki2020survey}
Maria Rigaki and Sebastian Garcia.
\newblock A survey of privacy attacks in machine learning.
\newblock {\em arXiv preprint arXiv:2007.07646}, 2020.

\bibitem[SU19]{SU19}
Adam Sealfon and Jonathan Ullman.
\newblock Efficiently estimating erdos-renyi graphs with node differential
  privacy.
\newblock {\em CoRR}, abs/1905.10477, 2019.

\bibitem[TC12]{task2012guide}
Christine Task and Chris Clifton.
\newblock A guide to differential privacy theory in social network analysis.
\newblock In {\em 2012 IEEE/ACM International Conference on Advances in Social
  Networks Analysis and Mining}, pages 411--417. IEEE, 2012.

\bibitem[YK{\"{O}}17]{YukselKO17}
Buket Y{\"{u}}ksel, Alptekin K{\"{u}}p{\c{c}}{\"{u}}, and {\"{O}}znur
  {\"{O}}zkasap.
\newblock Research issues for privacy and security of electronic health
  services.
\newblock {\em Future Gener. Comput. Syst.}, 68:1--13, 2017.

\bibitem[ZKKW20]{MRF}
Huanyu Zhang, Gautam Kamath, Janardhan Kulkarni, and Zhiwei~Steven Wu.
\newblock Privately learning markov random fields.
\newblock {\em CoRR}, abs/2002.09463, 2020.

\bibitem[ZLZ{\etalchar{+}}17]{ZhangLZHS17}
Jiajun Zhang, Xiaohui Liang, Zhikun Zhang, Shibo He, and Zhiguo Shi.
\newblock Re-dpoctor: Real-time health data releasing with w-day differential
  privacy.
\newblock In {\em 2017 {IEEE} Global Communications Conference, {GLOBECOM}
  2017, Singapore, December 4-8, 2017}, pages 1--6. {IEEE}, 2017.

\bibitem[ZLZP17]{zhu2017differentially}
Tianqing Zhu, Gang Li, Wanlei Zhou, and S~Yu Philip.
\newblock Differentially private social network data publishing.
\newblock In {\em Differential Privacy and Applications}, pages 91--105.
  Springer, 2017.

\end{thebibliography}
\end{flushleft}
\clearpage
\appendix
\section{Basic Definitions and the Extension Lemma}
\label{section:Appendix-Preliminaries}
\subsection{Basic Definitions}

A similar notion to $\epsilon$-differential privacy (Definition \ref{definition:epsilon-privacy}) is the notion of $(\epsilon,\delta)$-differential privacy.
\textcolor{black}{\begin{definition}\label{definition:epsilon-delta-privacy}
A randomized algorithm $\mathcal{A}$ is $(\epsilon,\delta)$-differential private if for all subsets $S \in \mathcal{F}$ of the output measurable space $(\Omega, \mathcal{F})$ and $n$-tuples of samples $X_1,X_2 \in \mathbb{R}^n$, such that $d_H(X_1,X_2)\le 1$ it holds \begin{equation} \label{delta-privdfn}\mathbb{P}\left(\mathcal{A}(X_1) \in S\right) \leq e^{\epsilon}\mathbb{P}\left(\mathcal{A}(X_2) \in S\right)+\delta.\end{equation}
\end{definition}It is rather straightforward that $(\epsilon,\delta)$-differential privacy is a \emph{weaker} notion to $\epsilon$-differential privacy, in the sense that for any $\epsilon,\delta>0$ any $\epsilon$-differentially private algorithm is also an $(\epsilon,\delta)$-differentially private algorithm.}

Of importance to us in the design of the algorithm desicred in the main body of this work is the notion of \emph{the left (empirical) median} of an $n$-tuple of samples.
\begin{definition}\label{definition:median}
For $x=(x_1,\ldots,x_n)\in\mathbb{R}^n$, we denote by $x_{(1)}, \ldots, x_{(n)}$ the reordered coordinates of $x$ in nondecreasing order, i.e. $\DS \min_{1\leq i\leq n} x_i=x_{(1)}\leq \ldots\leq x_{(n)}=\max_{1\leq i\leq n} x_i$. We let $\ell=\lfloor n/2\rfloor$ and $ m(x)=x_{(l)}$ be the empirical left median of $x$. 
\end{definition}

\subsection{On the Extension Lemma}\label{prem:extension}

In this section we provide more details on the Extension Lemma as stated in Proposition \ref{extension}.\textcolor{black}{We repeat it here for convenience.}

\begin{proposition}[``The Extension Lemma'' Proposition 2.1, \cite{borgs2018private}] \label{extensionApp} 
Let $\hat{\mathcal{A}}$ be an $\epsilon$-differentially private algorithm designed for input from $\mathcal{H} \subseteq \mathbb{R}^n$ with arbitrary output measure space $(\Omega,\mathcal{F})$. Then there exists a randomized algorithm $\mathcal{A}$ defined on the whole input space $\mathbb{R}^n$ with the same output space which is $2\epsilon$-differentially private and satisfies that for every $X \in \mathcal{H}$, $\mathcal{A}(X) \overset{d}{=}  \hat{\mathcal{A}}(X)$.
\end{proposition} 

Albeit the generality of the result, as mentioned in the conclusion of \cite{borgs2018private}, the Extension Lemma does not provide any guarantee that the extension of the algorithm $\hat{\mathcal{A}}$ can be made in a computationally efficient way. In this work, as described in the main body of the paper, we show that the Extension Lemma is applicable in the context of median estimation, and furthermore the extension can be implemented in polynomial time. Hence, naturally, to perform this extension in a computationally efficient way we don't use the Extension Lemma as a ``blackbox'' but rather need to use the specific structure of the extended private algorithm, by digging into the proof of it in \cite[Proposition 2.1.]{borgs2018private}.

\paragraph{The density of the extended algorithm - general} For simplicity, we present the density of the extended algorithm under certain assumptions that will be correct in the context of our work. As everywhere in this paper, let us first focus on the case the input space is $\mathcal{M}=\mathbb{R}^n$ equipped with the Hamming distance and the output space is the real line equipped with the Lebesgue measure. Furthermore let us assume also that for any $X \in \mathcal{H}$ the randomised restricted algorithm $\hat{\mathcal{A}}(X)$ follows a real-valued continuous distribution with a density $f_{\hat{\mathcal{A}}(X)}$ with respect to the Lebesque measure, given by Equation \eqref{restricted}. We repeat here the definitions for convenience;
\begin{equation} \label{restricted_app} f_{\hat{\mathcal{A}}\left(X\right)} (\omega)=\frac{1}{\hat{Z}} \exp\left(-\frac{\epsilon}{4} \min\left\{ \mystar  \left|m(X)-\omega\right|,\mysquare \right\} \right), \omega \in \mathcal{I}:=\myregion \end{equation}
where the normalizing constant is \begin{equation} \label{norm_restr_app} \hat{Z}=\int_{\mathcal{I}} \exp\left(-\frac{\epsilon}{4} \min\left\{  \mystar \left|m(X)-\omega\right|,\mysquare \right\} \right) \mathrm{d}\omega.\end{equation}As claimed in Section \ref{sec:optimalalg} the normalizing constant $\hat{Z}$ does not have an index $X$ because the right hand side of \eqref{norm_restr_app} takes the same value for any $X \in \mathcal{H}.$ This follows by the following Lemma.

\textcolor{black}{\emph{Note:} For readability reasons, we omit the proofs of the below propositions in the following section}
\begin{lemma}\label{norm_same}
Suppose $C>1/2$. Then for any $X \in \mathcal{H},$
$$\int_{\mathcal{I}} \exp\left(-\frac{\epsilon}{4} \min\left\{  \mystar \left|m(X)-\omega\right|,\mysquare \right\} \right) \mathrm{d}\omega=\int_{\mathcal{I}} \exp\left(-\frac{\epsilon}{4} \min\left\{  \mystar \left|\omega\right|,\mysquare \right\} \right) \mathrm{d}\omega.$$
\end{lemma}

Then from the proof of the Extension Lemma [Section 4, \cite{borgs2018private}] we have that the ``extended'' $\epsilon$-differentially private algorithm $\mathcal{A}$ on input $X \in \mathbb{R}^n$ admits also a density given by
\begin{equation}\label{alg:extended}f_{\mathcal{A}(X)}(\omega) =\frac{1}{Z_X} \inf_{X' \in \mathcal{H}} \left[ \exp\left(\frac{\epsilon}{4} d_H(X,X')\right) f_{\hat{\mathcal{A}}(X')}(\omega) \right], \omega \in \mathbb{R}\end{equation} where \begin{equation*}Z_X:=\int_{\mathbb{R}} \inf_{X' \in \mathcal{H}} \left[ \left(\frac{\epsilon}{4} d_H(X,X')\right) f_{\hat{\mathcal{A}}(X)'}(\omega) \right] d \omega.\end{equation*}

For reasons of completeness we state the corollary of the Extension Lemma that establishes that in our setting the algorithm $\mathcal{A}$ satisfies the desired properties of the Extension Lemma. \begin{proposition}\label{prop_replica}
Under the above assumptions, the algorithm $\mathcal{A}$ is $\epsilon$-differentially private and for every $X' \in \mathcal{H}$, $\mathcal{A}(X')\overset{d}{=}\hat{\mathcal{A}}(X')$.
\end{proposition} 

As a technical remark note that for the density in equation \eqref{alg:extended} to be well-defined we require that for every $X \in \mathbb{R}^n$ the ``unnormalized'' density function \begin{equation} \label{eq:Gfunction} G_X(\omega):=  \inf_{X' \in \mathcal{H}} \left[ \exp\left(\frac{\epsilon}{2} d_H(X,X')\right) f_{\hat{\mathcal{A}}(X')}(\omega) \right],  \omega  \in \mathbb{R}  \end{equation} is integrable and has a finite integral as well. 
Both conditions follows by the following Lemma, which establishes - among other properties - that $G_X$ is a continuous function almost everywhere and \textcolor{black}{therefore it} has a finite integral.
\begin{lemma}\label{well_posed}
Suppose the above assumptions hold and fix $X \in \mathbb{R}^n$. Then, \begin{itemize} 
\item $G_X(\omega)=0$ for all $\omega \not \in \mathcal{I}=\myregion$ 
\item $G_X$ is $\mathcal{R}$-Lipschitz on $\mathcal{I}$ with Lipschitz constant $\mathcal{R}=e^{\frac{\epsilon n}{2}}\frac{\epsilon L n}{12C \hat{Z}},$ where $\hat{Z}$ is given in \eqref{norm_restr_app}.
\end{itemize}Furthermore, it holds $0 \leq \int_{\omega \in \mathbb{R}} G_X(\omega) d\omega \leq 1.$
\end{lemma}

Plugging in now the densities of the restricted algorithm (Equations \eqref{restricted_app}) to the general extended algorithm (Equations \eqref{alg:general_app}) and finally using also that according to Lemma \ref{norm_same} the normalizing constant for the restricted algoriths is independent of $X' \in \mathcal{H}$, we have \begin{equation}\label{alg:general_app} f_{\mathcal{A}\left(X\right)}(\omega)=\frac{1}{Z_{X}}\exp\left(\displaystyle\inf_{X' \in \mathcal{H}} \left[ \frac{\epsilon}{2} d_H(X,X') -\frac{\epsilon}{4} \min\left\{ \mystar  \left|m(X')-\omega\right|,\mysquare \right\}\right] \right), \end{equation}where $\omega \in \myregion$ and $Z_X$ is the appropriate normalizing constant.


\subsection{Omitted Proofs}

\subsubsection{Proofs of \Cref{norm_same}}

\begin{proof}[Proof of Lemma \ref{norm_same}] Since $X=0 \in \mathcal{H}$ and the left hand side for $X=0$ evaluates to the right hand side, it suffices to show that the left hand side does not depend on the value of $m(X).$

By denoting $\mathcal{I}-m(X):=[-R-4Cr-m(X),R+4Cr-m(X)]$ we have
\begin{align*} &\int_{\mathcal{I}} \exp\left(-\frac{\epsilon}{4} \min\left\{  \mystar \left|m(X)-\omega\right|,\mysquare \right\} \right) \mathrm{d}\omega=\int_{\mathcal{I}-m(X)} \exp\left(-\frac{\epsilon}{4} \min\left\{  \mystar \left|\omega\right|,\mysquare \right\} \right) \mathrm{d}\omega\\
=&\int_{(\mathcal{I}-m(X))\cap\{|\omega| \leq 3Cr\}} \exp\left(-\frac{\epsilon}{4}   \mystar \left|\omega\right| \right) \mathrm{d}\omega+\int_{(\mathcal{I}-m(X))\cap\{|\omega| > 3Cr\}} \exp\left(-\frac{\epsilon}{4} \mysquare \right) \mathrm{d}\omega.\\
\end{align*}Now since $X \in \mathcal{H}$ we have $m(X) \in [-R-r/2,R+r/2].$ Hence, using $C>1/2$ we have $[-3Cr,3Cr] \subseteq \mathcal{I}-m(X) .$ Therefore the last summation of integrals simplifies to
\begin{align*} &\int_{|\omega| \leq 3Cr} \exp\left(-\frac{\epsilon}{4}   \mystar \left|\omega\right| \right) \mathrm{d}\omega+(|\mathcal{I}|-6Cr) \exp\left(-\frac{\epsilon}{4} \mysquare \right) \mathrm{d}\omega,
\end{align*}which does not depends on $X$ as we wanted.
\end{proof}

\subsubsection{Proof of  \Cref{prop_replica}}

For the proof of Proposition \ref{prop_replica}, we start with using the \cite[Lemma 4.1.]{borgs2018private} applied in our setting, which gives the following Lemma.
\begin{lemma}\label{lemma:Radon-Nikodyn-equivalence}
Let $\mathcal{A}'$ be a real-valued randomized algorithm designed for input from $\mathcal{H}' \subseteq \mathbb{R}^n$. Suppose that for any $X \in \mathcal{H}'$, $\mathcal{A}'(X)$ admits a density function with respect to the Lebesque measure $f_{\mathcal{A}'(X)}$. Then the following are equivalent \begin{itemize} \item[(1)] $\mathcal{A}'$ is $\epsilon$-differentially private on $\mathcal{H}$; \item[(2)] For any $X,X' \in \mathcal{H}$ \begin{equation}\label{prime}  f_{\mathcal{A}'(X)}(\omega) \leq e^{\epsilon d_H(X,X') } f_{\mathcal{A}'(X')}(\omega), \end{equation} almost surely with respect to the Lebesque measure. 
\end{itemize}
\end{lemma}

We now proceed with the proof of Proposition \ref{prop_replica}.

\begin{proof}[Proof of Proposition \ref{prop_replica}]

We first prove that $\mathcal{A}$ is $\epsilon$-differentially private over all pairs of input from $\mathbb{R}^n$. Using \Cref{lemma:Radon-Nikodyn-equivalence} it suffices to prove that for any $X_1,X_2 \in \mathcal{H}$,
 \begin{align*}
f_{\mathcal{A}(X_1)}(\omega) \leq \exp \left( \epsilon d_H(X_1,X_2)\right) f_{\mathcal{A}(X_2)}(\omega),
\end{align*} almost surely with respect to the Lebesque measure. We establish it in particular for every $\omega \in \mathbb{R}$. Notice that if $\omega \not \in \mathcal{I},$ both sides are zero from Lemma \ref{well_posed}. Hence let us assume $\omega \in \mathcal{I}$. Let $X_1,X_2 \in \mathbb{R}^n$. Using triangle inequality we obtain for every $\omega \in \mathcal{I}$, \begin{align*}\inf_{X' \in \mathcal{H}} \left[ \exp \left( \frac{\epsilon}{2} d_H(X_1,X')\right) f_{\hat{\mathcal{A}}(X')}(\omega) \right] & \leq \inf_{X' \in \mathcal{H}} \left[ \exp\left(\frac{\epsilon}{2} \left[d_H(X_1,X_2)+d_H(X_2,X')\right] \right) f_{\hat{\mathcal{A}}(X')}(\omega) \right]\\
&=\exp \left( \frac{\epsilon}{2} d_H(X_1,X_2)\right) \inf_{X' \in \mathcal{H}} \left[ \exp \left(\frac{\epsilon}{2} d_H(X,X')\right) f_{\hat{\mathcal{A}}(X')}(\omega) \right],
\end{align*}which implies that for any $X_1,X_2 \in \mathcal{M}$, 
\begin{align*}
Z_{X_1}&=\int_{\Omega} \inf_{X' \in \mathcal{H}} \left[ \exp\left(\frac{\epsilon}{2} d(X_1,X')\right) f_{\hat{\mathcal{A}}(X')}(\omega) \right] d \omega \\
&\leq \exp\left(\frac{\epsilon}{2} d(X_1,X_2) \right)\int_{\Omega} \inf_{X' \in \mathcal{H}} \left[ \exp \left( \frac{\epsilon}{2} d(X_2,X') \right) f_{\hat{\mathcal{A}}(X')}(\omega) \right] d  \omega\\
&=\exp \left( \frac{\epsilon}{2}d(X_1,X_2) \right) Z_{X_2}. 
\end{align*}Therefore using the above two inequalities we obtain that for any $X_1,X_2 \in \mathbb{R}^n$ and $\omega \in \mathcal{I}$,
 \begin{align*}
f_{\mathcal{A}(X_1)}(\omega) &=\frac{1}{Z_{X_1}} \inf_{X' \in \mathcal{H}} \left[ \exp \left( \frac{\epsilon}{2} d_H(X_1,X') \right) f_{\hat{\mathcal{A}}(X')}(\omega) \right] \\
&\leq \frac{1}{\exp\left(-\frac{\epsilon}{2} d_H(X_2,X_1)\right)Z_{X_2}}\exp\left(\frac{\epsilon}{2}d_H(X_1,X_2)\right) \inf_{X' \in \mathcal{H}}  \left[ \exp \left( \frac{\epsilon}{2} d(X_2,X') \right) f_{\hat{\mathcal{A}}(X')}(\omega) \right] \\
&=\exp \left( \frac{\epsilon}{2} d(X_1,X_2)\right) \frac{1}{Z_{X_2}} \inf_{X' \in \mathcal{H}}  \left[ \exp \left( \frac{\epsilon}{2}d_H(X_2,X') \right) f_{\hat{\mathcal{A}}(X')}(\omega) \right]\\
&=\exp \left( \epsilon d_H(X_1,X_2)\right) f_{\mathcal{A}(X_2)}(\omega),
\end{align*}as we wanted. 

Now we prove that for every $X \in \mathcal{H}$, $\mathcal{A}(X) \overset{d}{=}  \hat{\mathcal{A}}(X)$. Consider an arbitrary $X \in \mathcal{H}$. We know that $\hat{\mathcal{A}}$ is $\epsilon/2$-differentially private which based on Lemma \ref{lemma:Radon-Nikodyn-equivalence} implies that for any $X,X' \in \mathcal{H}$ \begin{equation}  f_{\hat{\mathcal{A}}(X)}(\omega) \leq \exp \left(\frac{\epsilon}{2} d_H(X,X') \right) f_{\hat{\mathcal{A}}(X')}(\omega),   \end{equation} almost surely with respect to the Lebesque measure. Observing that the above inequality holds almost surely as equality if $X'=X$ we obtain that for any $X \in \mathcal{H}$ it holds $$f_{\hat{\mathcal{A}}(X)}(\omega)=\inf_{X' \in \mathcal{H}}  \left[ \exp \left( \frac{\epsilon}{2} d_H(X,X') \right) f_{\hat{\mathcal{A}}(X')}(\omega) \right],$$ almost surely with respect to the Lebesque measure. Using that $f_{\hat{\mathcal{A}}(X)}$ is a probability density function  we conclude that in this case $$Z_{X}=\int  f_{\hat{\mathcal{A}}(X)}(\omega) d\omega =1.$$ Therefore $$f_{\hat{\mathcal{A}}(X)}(\omega)=\frac{1}{Z_X}\inf_{X' \in \mathcal{H}}  \left[ \exp \left( \epsilon d_H(X,X') \right) f_{\hat{\mathcal{A}}(X')}(\omega) \right],$$ almost surely with respect to the Lebesque measure and hence $$ f_{\hat{\mathcal{A}}(X)}(\omega)=f_{\mathcal{A}(X)}(\omega),$$ almost surely with respect to the Lebesque measure. This suffices to conclude that $\hat{\mathcal{A}}(X) \overset{d}{=}  \mathcal{A}(X)$ as needed. 

The proof of  Proposition \ref{prop_replica} is complete.
\end{proof}

\subsubsection{Proof of \Cref{well_posed}}

\begin{proof}[Proof of Lemma \ref{well_posed}] First notice that if $\omega \not \in \mathcal{I},$ from \eqref{restricted} for any $X' \in \mathcal{H},$ $f_{\hat{\mathcal{A}}(X')}(\omega)=0$. Therefore indeed $$0 \leq f_{\mathcal{A}(X)}(\omega) \leq \exp\left(\frac{\epsilon}{2} d_H(X,X')\right) f_{\hat{\mathcal{A}}(X')}(\omega)=0.$$
We prove now that for all $X' \in \mathcal{H},$  the function $\exp\left(\frac{\epsilon}{2} d_H(X,X')\right) f_{\hat{\mathcal{A}}(X')}(\omega)$ is $\mathcal{R}$-Lipschitz on $\mathcal{I}$. The claim then follows by the elementary real analysis fact that the pointwise infimum over an arbitrary family of $\mathcal{R}$-Lipschitz functions is an $\mathcal{R}$-Lipschitz function.

Now recall that for all $a,b>0$ by elementary calculus, $|e^{-a}-e^{-b}| \leq |a-b|.$ Hence, for fixed $X'\in \mathcal{H},$ using the definition of the density in equation \eqref{restricted_app}, we have for any $\omega, \omega' \in \mathcal{I},$ \begin{align*}
|f_{\hat{\mathcal{A}}(X')}(\omega)-f_{\hat{\mathcal{A}}(X')}(\omega')| \leq  \frac{\epsilon}{4\hat{Z}} |\min\left\{ \mystar  \left|m(X)-\omega\right|,\mysquare \right\}-\min\left\{ \mystar  \left|m(X)-\omega'\right|,\mysquare \right\}|
\end{align*}Now combining with Property \ref{property:triangular-inequality-in-min} we conclude 
\begin{align*}
|f_{\hat{\mathcal{A}}(X')}(\omega)-f_{\hat{\mathcal{A}}(X')}(\omega')| \leq  \frac{\epsilon}{4\hat{Z}} |\min\left\{ \mystar  \left|\omega- \omega'\right|,\mysquare \right\}| \leq \frac{\epsilon L n}{12C \hat{Z}}|\omega-\omega'|.
\end{align*} In particular, $\exp\left(\frac{\epsilon}{2} d_H(X,X')\right) f_{\hat{\mathcal{A}}(X')}(\omega)$ is $\mathcal{R}= e^{\frac{\epsilon n}{2} }\frac{\epsilon L n}{12C \hat{Z}}$-Lipschitz since $\exp\left(\frac{\epsilon}{2} d_H(X,X')\right)$ is a constant independent of $\omega$ with $\exp\left(\frac{\epsilon}{2} d_H(X,X')\right) \leq \exp(\frac{\epsilon n}{2} ). $ The proof of the Lipschitz continuity is complete. The final part follows from the fact that $G$ is non-negative by definition and again by definition for arbitrary fixed $X' \in \mathcal{H},$ $f_{\hat{\mathcal{A}}(X')}$
integrates to one and upper bounds pointwise the function $G_X.$

\end{proof}
\label{section:Extension-Lemma}
\section{Proofs for Section \ref{sec:optimalalg}: The Rate-Optimal Algorithm}
\label{section:Median}

\subsection{Proof of Lemma \ref{lemma:sensitivity-median}}

\begin{proof}
We assume $\kappa:=d_H(X,Y) \geq 1$ as if it equals \textcolor{black}{to} zero the Lemma follows. 

Consider two data-sets $\begin{cases} 
X&:=\{x_1,x_2,\ldots,x_{\kappa},C_1,\ldots,C_{n-\kappa}\}\\ 
Y&:=\{y_1,y_2,\ldots,y_{\kappa},C_1,\ldots,C_{n-\kappa}\}  \end{cases}
\in \mathcal{H}$.  Without loss of generality we assume that the common part is sorted in an increasing order, i.e. it holds $C_1\leq\ldots\leq C_{n-\kappa}.$

Notice that the left median of data-set can be altered by the addition of $M'$ new data-points only to some data-point among the $\lfloor M'/2 \rfloor +1$ data-points on the right and the $\lfloor M'/2 \rfloor +1$ data-points on the left of it. Hence the interval $[C_{\lfloor \frac{n-\kappa}{2} \rfloor},m(X)] \cup [m(X),C_{\lfloor \frac{n-\kappa}{2} \rfloor}]$ (respectively the interval $[C_{\lfloor \frac{n-\kappa}{2} \rfloor},m(Y)] \cup [m(Y),C_{\lfloor \frac{n-\kappa}{2} \rfloor}]$) there can be at most $\lfloor \frac{\kappa }{2}\rfloor+1$ data-points of the data-set $X$ (respectively of the data-set $Y$). 

Therefore, leveraging the definition of sensitivity set, we have that $$|m(X)-C_{\lfloor \frac{n-\kappa}{2} \rfloor}|\le (\lfloor \frac{\kappa }{2}\rfloor+1)\frac{ C }{Ln} $$ and  $$|m(Y)-C_{\lfloor \frac{n-\kappa}{2} \rfloor}|\le (\lfloor \frac{\kappa }{2}\rfloor+1)\frac{ C }{Ln} ,$$ which implies by the triangle inequality
$$ \left|m(X)-m(Y)\right| \leq (2\lfloor \frac{\kappa }{2}\rfloor+2)\frac{ C }{Ln} \leq 3\kappa\frac{ C}{Ln}.$$
\end{proof}

\subsection{Proof of Lemma \ref{lem:restricted-DP}}

\begin{proof}

\textcolor{black}{We start with the following elementary property is important in what follows for the proof of Lemma \ref{lem:restricted-DP}.
\begin{property}[Lemma 9.2. in \cite{BorgsCSZ18}] \label{property:triangular-inequality-in-min}For any $a, b > 0$ the function $f : \mathbb{R}\to \mathbb{R}$, with $f(x) = \min\{a|x|, b\}$, for all $x \in \mathbb{R}$,
satisfies the triangle inequality, $f(x + y) \le f(x) + f(y)$ for all $x, y \in \mathbb{R}$.
\end{property}
Additionally, using Lemma \ref{lemma:sensitivity-median} we have for any $X,Y \in \mathcal{H}$, 
\begin{equation}\label{equation:semiLip-median}
\min\left\{  \mystar \left|m(X)-m(Y)\right|,\mysquare  \right\} \leq d_H(X,Y).
\end{equation}
}

We firstly analyze the non-normalized ratio corresponding to two data-sets
$X,Y$:

\textcolor{black}{
\begin{align*}
\frac{Z_Xf_{\hat{\mathcal{A}}\left(X\right)} (q)}{Z_Yf_{\hat{\mathcal{A}}\left(Y\right)} (q)}&=
\frac{\exp\left(-\frac{\epsilon}{4} \min\left\{  \mystar  \left|m(X)-q\right|,\mysquare  \right\} \right)}{\exp\left(-\frac{\epsilon}{4} \min\left\{  \mystar  \left|m(Y)-q\right|,\mysquare  \right\}  \right)}\\
&\le
\exp\left(\frac{\epsilon}{4} \min\left\{  \mystar  \left|m(X)-m(Y)\right|,\mysquare  \right\} \right)
\ &\text{\Cref{property:triangular-inequality-in-min}}\\
&\le \exp\left(\frac{\epsilon}{4}d_H(X,Y)\right)
\ & \text{\Cref{equation:semiLip-median}}
\end{align*}
Furthermore for the ratio of the two normalizing constants we have
}
\begin{align*}
\frac{Z_Y}{Z_X}&=\frac{
\displaystyle\int_{\mathcal{I}} \exp\left(-\frac{\epsilon}{4} \min\left\{  \mystar  \left|m(Y)-q\right|,\mysquare  \right\} \right) \mathrm{d}q
}{
\displaystyle\int_{\mathcal{I}} \exp\left(-\frac{\epsilon}{4} \min\left\{  \mystar  \left|m(X)-q\right|,\mysquare  \right\} \right) \mathrm{d}q
}\\
&\le \frac{
\displaystyle\int_{\mathcal{I}} \exp\left(-\frac{\epsilon}{4} \min\left\{  \mystar  \left|m(Y)-q\right|,\mysquare  \right\} \right) \mathrm{d}q
}{
\displaystyle\int_{\mathcal{I}} \exp\left(-\frac{\epsilon}{4}\begin{cases} \min\left\{  \mystar  \left|m(X)-m(Y)\right|,\mysquare  \right\}\\+\\\min\left\{  \mystar  \left|m(Y)-q\right|,\mysquare  \right\}\end{cases} \right) \mathrm{d}q
}&\text{\Cref{property:triangular-inequality-in-min}}\\
&=
\exp\left(\frac{\epsilon}{4} \min\left\{  \mystar  \left|m(X)-m(Y)\right|,\mysquare  \right\} \right)\\
&\le 
\exp\left(\frac{\epsilon}{4}d_H(X,Y)\right) \ & \text{\Cref{equation:semiLip-median}}
\end{align*}
Combining the two final inequalities above we conclude
\[
f_{\hat{\mathcal{A}}\left(X\right)} (q)\le \exp\left(\frac{\epsilon}{2} d_H(X,Y)\right)f_{\hat{\mathcal{A}}\left(Y\right)} (q), \ \forall q\in \mathcal{I}
\] The proof is complete.

\end{proof}

\subsection{Proof of Lemma \ref{lemma:accuracy-typical0}}

Lemma \ref{lemma:accuracy-typical0} follows from the following more general lemma where $C>1$ is allowed to grow with $n$ while in the context of Lemma \ref{lemma:accuracy-typical0} is considered in the asymptotic analysis as a constant.
\begin{lemma}[]\label{lemma:accuracy-typical}
Suppose $C>1$, possibly scaling with $n$, and $\alpha \in (0,r), \beta \in (0,1)$. For some $$n=O\left(C\frac{\log \left(\frac{1}{\beta}\right)}{\epsilon L \alpha}+\frac{\log \left( \frac{R}{\alpha}+1\right)}{\epsilon L r}\right)$$it holds
\begin{align*}
    \max_{X \in \mathcal{H}} \mathbb{P} [ |\hat{\mathcal{A}}(X)-m(X)|  \geq \alpha] \leq \beta,
\end{align*}where the probability is with respect to the randomness of the algorithm $\hat{\mathcal{A}}(X).$
\end{lemma}
\begin{proof}
$$$$
\begin{itemize}
    \item \textbf{Step 1}: Observe: Since $C>1/4$ it holds $R+4Cr>R+r$. Hence, by change of variables $q=q-m(X)$,
\begin{align*}
    Z_X&= \int_{\myregion} \exp\left(-\frac{\epsilon}{4} \min\left\{  \mystar \left|m(X)-q\right|,\mysquare \right\} \right) \mathrm{d}q\\
    & \geq \int_{0}^{R+r} \exp(-\frac{\epsilon}{4} \min\{ \mystar q,\mysquare\} )\mathrm{d}q \\
    & = \int_{0}^{\min\{3rC,R+r\}} \exp(-\frac{\epsilon}{4} \cdot \mystar q)\mathrm{d}q+\int_{\min\{3rC,R+r\}}^{R+r} \exp(-\frac{\epsilon}{4}Lrn)\mathrm{d}q\\
    & \geq \int_{0}^{\min\{3rC,R+r\}} \exp(-\frac{\epsilon}{4} \cdot \mystar q)\mathrm{d}q \\
    & =  \frac{1}{\frac{\epsilon}{4} \cdot \mystar}\left(1- \exp(-\frac{\epsilon}{4} \cdot \mystar \min\{3Cr, R+r\})  \right)\\
    & =  \frac{12C}{ n \epsilon L}\left(1- \exp(-\frac{n \epsilon L }{12C} \cdot  \min\{3Cr, R+r\}) )  \right)\\
    & \geq  \frac{12C}{ n \epsilon L}\left(1- \exp(-\Theta\left(\frac{n \epsilon Lr}{C}  \right) )  \right)\\
\end{align*}

    \item \textbf{Step 2}:Now using that $m(X) \in [- R-r/2, R+r/2]$ and by change of variables $q=q-m(X)$ we have:
\begin{align*}
    \mathbb{P} [ |\hat{\mathcal{A}}(X)-m(X)| \geq \alpha ] &\leq \frac{2}{Z_X} \int_{\alpha}^{2 R+3Cr+r} q \exp(-\frac{\epsilon}{4} \min\{ \mystar q,\mysquare\} )\mathrm{d}q \\
    &\leq \frac{2}{Z_X} \left( \int_{\alpha}^{+\infty}  \exp(-\frac{\epsilon}{4} \mystar q)\mathrm{d}q+ \int_{3Cr}^{2R+3Cr+r} \exp(-\frac{\epsilon}{4}\mysquare)\mathrm{d}q \right) \\
    &\leq \frac{2}{Z_X} \left(\frac{12C}{ n \epsilon L}\exp(-\frac{\epsilon}{4} \mystar \alpha)+ (2R+r) \exp(-\frac{Lr \epsilon n}{2})\right)
\end{align*}

Hence we conclude for all $X \in \mathcal{H}$

\begin{align*}
    \mathbb{P} [ |\hat{\mathcal{A}}(X)-m(X)|  \geq \alpha] \leq 2 \left(1- \exp(-\Theta\left(\frac{n \epsilon Lr}{C}  \right))\right)^{-1} \left(\exp(-\frac{\epsilon}{4} \mystar \alpha)+(2R+r)\frac{n\epsilon L}{12C} \exp(-\frac{\epsilon Lr n}{2} )\right)
\end{align*}From this and elementary asymptotics we conclude that for $C>1$ with $$n=O\left(C\frac{\log \left(\frac{1}{\beta}\right)}{\epsilon L \alpha}+\frac{C+\log \left( \frac{R}{r}+1\right)}{\epsilon L r}\right)$$ it holds for all $X \in \mathcal{H},$ $\mathbb{P} [ |\hat{\mathcal{A}}(X)-m(X)|  \geq \alpha] \leq \beta.$ Using that $\alpha \leq r$ the above sample complexity bound simplifies to $$n=O\left(C\frac{\log \left(\frac{1}{\beta}\right)}{\epsilon L \alpha}+\frac{\log \left( \frac{R}{\alpha}+1\right)}{\epsilon L r}\right).$$ The proof of the Lemma is complete.

\end{itemize}
 
\end{proof}

\subsection{Proof of Lemma \ref{lemma:typical-set-median}}


We need the following technical lemma.
\begin{lemma}\label{lemma:typical_set_general}
Let $\mathcal{D}$ be an admissible distribution and $X=(X_1,\ldots,X_n)$ consisting of i.i.d. samples from $\mathcal{D}$. Suppose $C>5,$ possibly scaling with $n$, which satisfies $4Ce\exp\left( -\frac{2C}{27}\right)<1/2$. Then for any $T \in [0,\frac{Lr}{4C}n]$ it holds

$$\displaystyle \mathbb{P}\left(\bigcap_{\kappa \in [T,\frac{Lr}{2C}n]\cap \mathbb{Z}} 
\sum_{i\in [n]}\mathbf{1}\{X_i-m(X)\in[0,\frac{C \kappa }{Ln}]\}\ge \kappa +1\right) \geq 1-O\left( (8Ce\exp\left( -\frac{C}{8}\right))^{\lceil T\rceil }+e^{-\Theta \left( L^2r^2n\right)}\right),$$

and

$$\displaystyle \mathbb{P}\left(\bigcap_{\kappa \in [T,\frac{Lr}{2C}n]\cap \mathbb{Z}} 
\sum_{i\in [n]}\mathbf{1}\{X_i-m(X)\in[-\frac{C \kappa }{Ln},0]\}\ge \kappa +1\right) \geq 1-O\left( (8Ce\exp\left( -\frac{C}{8}\right)))^{\lceil T\rceil }+e^{-\Theta \left( L^2r^2n\right)}\right),$$
\end{lemma}

Before proving the Lemma \ref{lemma:typical_set_general}, we show how it can be used to prove the following more general version of Lemma \ref{lemma:typical-set-median} where $C>0$ is allowed to potentially scale with $n$.
\begin{lemma}\label{lemma:typical-set-median11} Let $\mathcal{D}$ be an admissible distribution, $\beta \in (0,1),$ $n \geq 3$ and $X=(X_1,\ldots,X_n)$ consisting of i.i.d. samples from $\mathcal{D}$. Suppose $C>5$, possibly scaling with $n$, which satisfies $4Ce\exp\left( -\frac{2C}{27}\right)<1/2$. Then for some $E>0$, $C'=C'(C)>0$ with $C'(C)=\Theta(\frac{1}{C})$ as $C$ grows, if $n \geq E \frac{\log \frac{1}{\beta}}{L^2r^2}$ then it holds
$$\displaystyle \mathbb{P}\left( \exists X' \in \mathcal{H} \text{ s.t. } d_H\left(X , X' \right)  \leq C' \log \frac{1}{\beta}, m(X')=m(X) \right) \geq 1-\beta.$$
\end{lemma}
\begin{proof}[Proof of Lemma \ref{lemma:typical-set-median11}]
We begin by applying Lemma \ref{lemma:typical_set_general} for sufficiently large $E>0$.

 We will choose $T=\frac{C'(C)}{2} \log \frac{1}{\beta}$ for appropriate choice $C'(C)>0$ with $C'(C)=\Theta(1/C)$ as $C$ grows. We claim that for any choice of such a function $C'(C)$ and in particular choice of $T$ for sufficiently large $E>0$ it holds $$T=\frac{C'(C)}{2} \log \frac{1}{\beta}  \leq \frac{Lr}{4C}n $$which suffices for us to apply the Lemma. To prove this note that the equivalently we need $$ n \geq 2C \frac{ C'(C)}{Lr} \log \frac{1}{\beta} .$$  Since $Lr \leq 1/2$ for any constant $E>0,$ if $n \geq E \frac{\log \frac{1}{\beta}}{L^2r^2}$ then $n \geq \frac{E}{2} \frac{ \log \frac{1}{\beta}}{Lr}.$ Hence, using that $C'(C)=\Theta(1/C)$  by choosing $E>0$ sufficiently large but constant, the choice of $T$ is indeed valid for apply the Lemma \ref{lemma:typical_set_general}.

We choose $T=\frac{C'(C)}{2} \log \frac{1}{\beta}$ for appropriate $C'(C)>0$ with $C'(C)=\Theta(1/C)$ as $C$ grows, so that using Lemma \ref{lemma:typical_set_general} both the two probabilistic guarantees hold with probability at least $1-\frac{\beta}{2}-O\left(e^{-\Theta \left( L^2r^2n\right)}\right).$ Notice that this is possible since this is the exact scaling of $T$ for which $(8Ce\exp\left( -\frac{C}{8}\right))^{\lceil T\rceil }=\Theta\left(\beta \right).$

Taking now  $E>0$ further sufficiently large but still constant since $n \geq E \frac{\log \frac{1}{\beta}}{L^2r^2}$ we can make sure the probabilistic guarantees hold with probability at least $1-\beta$. Combining the above, with probability at least $1-\beta$ the following event holds: for all $\kappa \in [T,\frac{Lr}{2C}n]\cap \mathbb{Z},$ \begin{equation}\label{eq:1k} \sum_{i\in [n]}\mathbf{1}\{X_i-m(X)\in[-\frac{C \kappa }{Ln},0]\}\ge \kappa +1 \end{equation}and \begin{equation} \label{eq:2k} \sum_{i\in [n]}\mathbf{1}\{X_i-m(X)\in[-\frac{C \kappa }{Ln},0]\}\ge \kappa +1.\end{equation} In words, all but, the closest to $m(X)$, $T=\frac{C'(C)}{2} \log \frac{1}{\beta}$ intervals on the left of $m(X)$, and all but, the closest to $m(X)$, $T=\frac{C'(C)}{2} \log \frac{1}{\beta}$ interval on the right of $m(X)$, satisfy their corresponding constraints described in the typical set $\mathcal{H}$.

Now let us consider a data-set $X$ for which \eqref{eq:1k} and \eqref{eq:2k} hold. We show that given these events one can construct an $X' \in \mathcal{H}$ satisfying the conditions described in the event considered in Lemma \ref{lemma:typical-set-median11}. First recall that as $T \leq \frac{Lr}{4C}n$ since $Lr \leq 1/2$ and $C>5$, it necessarily holds $T \leq n/40.$ Since we assume $n \geq 3$ it holds $T<n/2-2.$ Now, as there are at least $n/2-2$ data-points on the left of $m(X)$ and $n/2-2$ data-points on the right of $m(X)$ we can modify $X$ by choosing arbitrary $T$ data-points on the right of $m(X)$ and arbitrary $T$ data-points on the left of $m(X)$ and change all their position to $m(X)$. Notice that such a change produces a new data-set $X'$ which has Hamming distance $2T=C'(C) \log \frac{1}{\beta}$ with $X$ and has the same median with $X$, $m(X)=m(X')$.

Notice that as we moved data-points closer to the median, all satisfied constraint from $\mathcal{H}$ by the data-set $X$ according to \eqref{eq:1k} and \eqref{eq:2k} continue to be satisfied from $X'$. On top of this, since at least $2T$ data-points of $X'$ take now exactly the value of the left empirical median, the data-set $X'$ necessarily satisfies also the, potentially violated by $X$, constraints corresponding to the $T$ intervals on the left of $m(X')=m(X)$ and the $T$ intervals on the right of $m(X')=m(X)$. We conclude that $X'$ satisfies all the constraints described in $\mathcal{H}$ and therefore, $X' \in \mathcal{H},$ which completes the proof. 
\end{proof}

For the rest part we focus on proving Lemma \ref{lemma:typical_set_general}.
\begin{proof}[Proof of Lemma \ref{lemma:typical_set_general}]

We start by noticing that by identical reasoning as the derivation of inequality (14) of \cite{Subgaussian} in the proof of Lemma 3 in \cite{Subgaussian} we have for $t=r/2 \in [0,r]$ 
\begin{equation} \label{eq:marco} \mathbb{P}\left(|m(X)-m(\mathcal{D})| \geq r/2\right) \leq 2e^{-nL^2r^2/8}.
\end{equation}

Now we focus on proving the first out of the two probabilistic guarantee, $$\displaystyle \mathbb{P}\left(\bigcap_{\kappa \in [T,\frac{Lr}{2C}n]\cap \mathbb{Z}} 
\sum_{i\in [n]}\mathbf{1}\{X_i-m(X)\in[0,\frac{C \kappa }{Ln}]\}\ge \kappa +1\right) \geq 1-O\left( (8Ce\exp\left( -\frac{C}{8}\right))^{\lceil T\rceil }+e^{-\Theta \left( L^2r^2n\right)}\right),$$as the second follows naturally by the symmetric argument around zero. We make the following three observations.

First, observe that generating i.i.d. samples from $\mathcal{D}$ and then determining the left empirical median $m(X)$ based on the realization of $X_1,\ldots,X_n$ is equal in law with first sampling the place of the left empirical median $m(X)$ based on its distribution and sampling uniformly at random an $X_i$ so that $X_i=m(X)$ and then generating $\lceil \frac{n-1}{2} \rceil-1(n \text{ is even})$ i.i.d. samples from $\mathcal{D}$ conditional to be on the left of $m(X)$ and generating $\lceil \frac{n-1}{2} \rceil$ i.i.d. samples from $\mathcal{D}$ conditional to be on the right of $m(X)$. To model the underlying the randomness we can define for each $i$, the trinary random variables where $C_i=-1$ if the sample $i$ is chosen to be on the left of $m(X)$, $C_i=0$ if the sample equals $m(X)$ and $C_i=1$ if they are chosen to be on the right.

Second, using the fact that $\mathcal{D}$ has density lower bounded by $L$ in $[m(\mathcal{D})-r,m(\mathcal{D})+r]$, observe that our distribution $\mathcal{D}$ can be decomposed as a mixture,
\begin{equation} \label{eq:decomp}
\mathcal{D}=\frac{rL}{2}\mathrm{Unif}[m(\mathcal{D})-r,m(\mathcal{D})+r]+\left(1-\frac{rL}{2}\right)\mathcal{D}'
\end{equation} for some other probability measure $\mathcal{D}'$ on the reals. Now using the first observation, we first sample  the position of $m(X)$ and describe the sampling of $X_i$ as follows. If $C_i=0$ we simply set $X_i=m(X)$. For the other two cases, we assume $C_i=1$ as the other case is symmetric. If $m(X)-m(\mathcal{D}) \geq r $ then we sample from the distribution $\left(1-\frac{rL}{2}\right)\mathcal{D}'$ conditional on being on the right of $m(X)$. If $m(X)-m(\mathcal{D}) \leq r $ then we first flip a coin $B_i\overset{d}{=}\mathrm{Bernoulli}(\frac{rL}{2})$ and with probability $\frac{rL}{2}$ (that is when $B_i=1$) we sample from $\mathrm{Unif}[m(\mathcal{D})-r,m(\mathcal{D})+r]$ conditional on being on the right of $m(X)$ and with probability $1-\frac{rL}{2}$ (that is when $B_i=0$) we sample from some distribution $\mathcal{D}'$  conditional on being on the right of $m(X)$.

Third, observe that both the event of interests in the statement Lemma \ref{lemma:typical_set_general} of becomes less probable when we restrict ourselves to only a subset of the $n$ samples. Hence, since we want to prove a lower bound on the probability of the event we can restrict ourselves to an arbitrary subset.

Now, using \eqref{eq:marco} by neglecting an event of probability  $2e^{-nL^2r^2/8} \leq e^{-\Theta \left( L^2r^2n\right)}$ from now on we condition on $|m(X)-m(\mathcal{D})| \leq r/2$. Now using the three observations above, we restrict ourselves only on the $n_1 \leq n$ samples that satisfy $B_i=C_i=1$, that is they are samples from $\mathrm{Unif}[m(\mathcal{D})-r,m(\mathcal{D})+r]$ conditional on being on the right of $m(X)$.  We denote these samples by $X_1,\ldots,X_{n_1}$ for simplicity. Notice that $\mathrm{Unif}[m(\mathcal{D})-r,m(\mathcal{D})+r]$ conditional on being on the right of $m(X)$ is just distributed as $\mathrm{Unif}[m(X),m(\mathcal{D})+r]$. Hence, conditioning on the value of $n_1$, the density of the conditional distribution of each $X_1,\ldots,X_{n_1}$ given the $m(X)$ can be straightforwardly check to satisfy for each $i=1,2,\ldots,n_1,$ 
\begin{equation}\label{eq:cond_density}
    \frac{2}{3r} \leq f_{X_i}(u|m(X),|m(X)-m(\mathcal{D})| \leq r/2) \leq \frac{2}{r},  u \in [m(X),m(\mathcal{D})+r].
\end{equation}

Combining the above, to prove our result it suffices to prove $$\displaystyle \mathbb{P}\left(\bigcap_{\kappa \in [T,\frac{Lr}{2C}n]\cap \mathbb{Z}}  
\sum_{i\in [n_1]}\mathbf{1}\{X_i-m(X)\in[0,\frac{C \kappa }{Ln}]\}\ge \kappa +1\right) \geq 1-O\left( (8Ce\exp\left( -\frac{C}{8}\right)))^{\lceil T\rceil }+e^{-\Theta \left( L^2r^2n\right)}\right).$$

Since $m(X)=X_i$ for some $i$ by definition of the left empirical median, it suffices to prove that for $$J_{\kappa}:=(m(X),m(X)+\frac{C\kappa }{Ln} ], \text{  } \kappa=\lceil T\rceil,\lceil T\rceil + 1,\ldots,\lfloor \frac{Lr}{2C}n \rfloor-1$$ and sequence of events $$A_{\kappa}:=\{\sum_{i\in [n_1]}\mathbf{1}\{X_i\in J_{\kappa}\}\ge \kappa \} , \text{  } \kappa=\lceil T\rceil,\lceil T\rceil + 1,\ldots,\lfloor \frac{Lr}{2C}n \rfloor-1$$ it holds $$ \mathbb{P}\left(\bigcap_{\kappa=\lceil T\rceil }^{\lfloor \frac{Lr}{2C}n \rfloor-1} A_{\kappa} \bigg{|} |m(X)-m(\mathcal{D})|\leq r/2\right) \geq 1-O\left( (8Ce\exp\left( -\frac{C}{8}\right)))^{\lceil T\rceil }+e^{-\Theta \left( L^2r^2n\right)}\right)).$$ Using a union bound it suffices to show $$ \sum_{\kappa=\lceil T\rceil }^{\lfloor \frac{Lr}{2C}n \rfloor-1}\mathbb{P}\left(A^c_{\kappa} \cap \bigcap_{s=\lceil T\rceil}^{\kappa-1} A_{s}\bigg{|} |m(X)-m(\mathcal{D})|\leq r/2\right) \leq  O\left( (8Ce\exp\left( -\frac{C}{8}\right)))^{\lceil T\rceil }+e^{-\Theta \left( L^2r^2n\right)}\right).$$

Observe that as we are conditioning on $|m(X)-m(\mathcal{D})|\leq r/2$ for all $\kappa$ of interest $$J_{\kappa} \subseteq [m(\mathcal{D})-r,m(\mathcal{D})+r].$$ Hence using  \eqref{eq:cond_density} we have that, conditioned on $n_1$, for each $i \in [n_1]$ and $\kappa$, \begin{equation}\label{eq:prob_bound} \mathbb{P}\left(X_i \in J_{\kappa}|m(X),|m(X)-m(\mathcal{D})|\leq r/2 \right) \in [\frac{2}{3r} |J_{\kappa}|, \frac{2}{r}|J_{\kappa}| ].\end{equation} Furthermore recall that $n_1\overset{d}{=}\mathrm{Binom}\left(\lfloor \frac{n-1}{2} \rfloor, \frac{Lr}{2}\right)$ and that it holds $\kappa \leq \frac{Lrn}{2C}.$ Since $C>5$ by standard concentration inequalities, we have that with probability $1-e^{-\Theta\left(L^2r^2n\right)},$ it holds for all $\kappa$ of interest 
\begin{equation}\label{eq:binom_cond}
    \kappa<\frac{Lrn}{9}<\frac{n_1}{2}<Lrn.
\end{equation}

In what follows we condition on the event \eqref{eq:binom_cond}.
Suppose $\kappa=\lceil T \rceil$. Then using \eqref{eq:prob_bound} by definition the probability of $A^c_{\lceil T \rceil}$ is at most the probability a sample from a binomial distribution with $N:=n_1$ draws and probability $\frac{2}{r}|J_{\lceil T \rceil}|=\frac{2C\lceil T \rceil}{nLr}$ is at most $\lceil T \rceil$. Now conditioning on $n_1$ and denoting the Binomial random variable by $Z$ we have by the additive form of the Chernoff's inequality,
\begin{align*}
 \mathbb{P}\left(A^c_{\lceil T \rceil} \bigg{|}m(X), |m(X)-m(\mathcal{D})|\leq r/2, n_1\right)  & \leq   \mathbb{P}\left(Z \leq \lceil T \rceil\right)\\
 & \leq \exp\left(-\frac{\left(\frac{2n_1C \lceil T \rceil }{Lnr}-\lceil T \rceil\right)^2}{2n_1 \frac{2C \lceil T \rceil}{Lnr}\left(1-\frac{2C\lceil T \rceil}{Lnr}\right) }  \right)
\end{align*}Using now that we condition on \eqref{eq:binom_cond} it holds $\frac{2Lnr}{9}<n_1<2Lrn$ and $T<\frac{Lnr}{4C}$ we have
\begin{align*}
 \mathbb{P}\left(A^c_{\lceil T \rceil} \bigg{|} |m(X)-m(\mathcal{D})|\leq r/2\right) &=\mathbb{E}_{m(X),n_1}\left[\mathbb{P}\left(A^c_{\lceil T \rceil} \bigg{|}m(X), |m(X)-m(\mathcal{D})|\leq r/2, n_1\right) \right]\\
 & \leq \exp\left(-\frac{\left(\frac{4}{9}C-1\right)^2 \lceil T \rceil ^2}{2C \lceil T \rceil } \right)
\end{align*}which since $C>5>\frac{9}{7}$ gives that 
\begin{align}\label{eq:firstbound}
 \mathbb{P}\left(A^c_{\lceil T \rceil} \bigg{|} |m(X)-m(\mathcal{D})|\leq r/2\right)  & \leq \exp\left(-\frac{C \lceil T \rceil}{8} \right).
\end{align}

Now for every $\kappa>\lceil T \rceil$ notice that the event $A^c_{\kappa} \cap \bigcap_{s=0}^{\kappa-1} A_{s}$ can happen if and only if $\kappa$ of the samples belong in $J_{\kappa -1}$ and the rest $n_1-\kappa$ samples belong to $J^c_{\kappa}$. Therefore using the above observations and \eqref{eq:prob_bound}, \begin{align*} \mathbb{P}\left(A^c_{\kappa} \cap \bigcap_{s=\lceil T \rceil}^{\kappa-1} A_{s} \bigg{|}m(X), |m(X)-m(\mathcal{D})|\leq r/2,n_1\right) & \leq \binom{ n_1}{\kappa} \left(\frac{2}{r}|J_{\kappa -1}|\right)^{\kappa} \left(1- \frac{2}{3r}|J_{\kappa}| \right)^{n_1-\kappa}\\
& \leq \binom{n_1}{\kappa} \left(\frac{2C(\kappa-1)}{nLr} \right)^{\kappa} \left(1- \frac{2C\kappa}{3Lrn} \right)^{n_1-\kappa}\\
& \leq \left(\frac{n_1e}{\kappa}\right)^{\kappa} \left(\frac{2C\kappa}{Lrn} \right)^{\kappa} \exp\left( -C\frac{2(n_1-\kappa)\kappa}{3Lrn}\right)\\
& \leq \left(2Ce\frac{n_1}{Lrn}\right)^{\kappa}  \exp\left( -C\frac{2(n_1-\kappa)\kappa}{3Lrn}\right),
\end{align*}where we used the elementary inequalities $\binom{m}{m'} \leq \left(\frac{me}{m'}\right)^{m'}$, $1+x \leq e^x$ and that by definition $n_1 \leq n.$

The last displayed inequality conditioned on the event \eqref{eq:binom_cond} implies

\begin{align} \mathbb{P}\left(A^c_{\kappa} \cap \bigcap_{s=\lceil T \rceil}^{\kappa-1} A_{s} \bigg{|} |m(X)-m(\mathcal{D})|\leq r/2\right) 
& =\mathbb{E}_{m(X),n_1}\left[ \mathbb{P}\left(A^c_{\kappa} \cap \bigcap_{s=0}^{\kappa-1} A_{s} \bigg{|}m(X), |m(X)-m(\mathcal{D})|\leq r/2,n_1\right)
\right] \nonumber \\
& \leq \mathbb{E}_{n_1}\left[\left(4Ce\right)^{\kappa}  \exp\left( -\frac{2C}{3}\frac{(n_1-\frac{n_1}{2})\kappa}{Lrn}\right)\right] \nonumber\\
& \leq \mathbb{E}_{n_1}\left[\left(4Ce\right)^{\kappa}  \exp\left( -\frac{2C}{3}\frac{\frac{n_1}{2}\kappa}{Lrn}\right)\right] \nonumber\\
& \leq \mathbb{E}_{n_1}\left[\left(4Ce\right)^{\kappa}  \exp\left( -\frac{2C}{27}\kappa\right)\right] \nonumber \\
& = \left(4Ce\exp\left( -\frac{2C}{27}\right)\right)^{\kappa} \label{eq:secondbound}
\end{align}

Using \eqref{eq:firstbound} for the first term and since $4Ce\exp\left( -\frac{2C}{27}\right)<1/2$ a geometric summation over $\kappa \geq \lceil T \rceil$ for the rest terms, we have that conditional on \eqref{eq:binom_cond},
\begin{align*}
\sum_{\kappa=\lceil T \rceil}^{\lfloor \frac{Lr}{2C}n \rfloor-1}\mathbb{P}\left(A^c_{\kappa} \cap \bigcap_{s=\lceil T \rceil}^{\kappa-1} A_{s} \bigg{|} |m(X)-m(\mathcal{D})|\leq r/2\right)  & \leq \exp\left(-\frac{C \lceil T \rceil }{8} \right)+( 8Ce\exp\left( -\frac{2C}{27}\right))^{\lceil T \rceil}.
\end{align*}
Taking now into account the probability of the conditioned event we have \eqref{eq:binom_cond},  
\begin{align*}
\sum_{\kappa=\lceil T \rceil}^{\lfloor \frac{Lr}{2C}n \rfloor-1}\mathbb{P}\left(A^c_{\kappa} \cap \bigcap_{s=\lceil T \rceil}^{\kappa-1} A_{s} \bigg{|} |m(X)-m(\mathcal{D})|\leq r/2\right)  & \leq \exp\left(-\frac{C \lceil T \rceil }{8} \right)+( 8Ce\exp\left( -\frac{2C}{27}\right))^{\lceil T \rceil}+e^{-\Theta\left(L^2r^2n\right)} \\
& \leq 2( 8Ce\exp\left( -\frac{C}{8}\right))^{\lceil T \rceil}+e^{-\Theta\left(L^2r^2n\right)} \\
\end{align*}where in the last line since $\frac{C}{8}>\frac{2C}{27}$ and $8Ce>1$ under our assumptions it holds for all $T \geq 0$
$$\exp\left(-\frac{C \lceil T \rceil }{8} \right)+( 8Ce\exp\left( -\frac{2C}{27}\right))^{\lceil T \rceil} \leq 2( 8Ce\exp\left( -\frac{C}{8}\right))^{\lceil T \rceil}.$$
The proof is complete.

\end{proof}

\subsection{Proof of Theorem \ref{thm:upper_bound}} 

We establish instead the following slightly more general result which does not assume that $C$ is a constant, but could scale with $n$.

\begin{theorem}\label{thm:upper_bound0}
Suppose $C>0$, possibly scaling with $n$, is bigger than a sufficiently large constant, $\epsilon \in (0,1)$, $\mathcal{D}$ is an admissible distribution and $\mathcal{A}$ is the $\epsilon$-differentially private algorithm defined above. Then for any $\alpha \in (0,r)$ and $ \beta \in (0,1)$ for some $$ n=O\left( \frac{\log \left( \frac{1}{\beta}\right)}{L^2 \alpha^2}+C\frac{ \log \left(\frac{1}{\beta}\right)}{\epsilon L\alpha }+\frac{\log \left(\frac{R}{ \alpha}+1\right)}{\epsilon Lr}\right)$$ it holds $\quad\quad \quad \mathbb{P}_{X_1,X_2,\ldots,X_n \overset{iid}{\sim} \mathcal{D}} [ |\mathcal{A}(X_1,\ldots,X_n)-m\left(\mathcal{D}\right)|  \geq \alpha] \leq \beta.$
\end{theorem}

\begin{proof} Let $X=(X_1,\ldots,X_n)$ the $n$-tuple of i.i.d. samples from $\mathcal{D}.$ Let us use a parameter $\gamma \in (0,1)$ which we later choose a polynomial function of $\beta.$
We consider the event $$\mathcal{T}_{\gamma}:=\{\exists X' \in \mathcal{H} \text{ s.t. } d_H\left(X , X' \right)  \leq C' \log \frac{1}{\gamma}, m(X')=m(X)\},$$where $C'=C'(C)$ is chosen to satisfy the conclusion of Lemma \ref{lemma:typical-set-median11}. In particular it holds $C'(C)=\Theta( \frac{1}{C})$ as $C$ grows to infinity.

Notice that since $ \alpha \leq r$ for some $n=O\left( \frac{\log \frac{1}{\gamma}}{L^2\alpha^2}\right)$ it holds $n \geq D\frac{\log \frac{1}{\gamma}}{L^2r^2}$ for the $D>0$ defined in Lemma \ref{lemma:typical-set-median11}. Furthermore, we can assume $C>5$ is sufficiently large such that $4Ce\exp\left( -\frac{2C}{27}\right)<1/2$. Hence, by applying Lemma \ref{lemma:typical-set-median11} we have
\begin{align}
    \mathbb{P}\left[ |\mathcal{A}\left(X\right)-m\left(\mathcal{D}\right)| \geq \alpha \right] & \leq \mathbb{P}\left[ |\mathcal{A}\left(X\right)-m\left(\mathcal{D}\right)| \geq \alpha , X \in \mathcal{T}_{\gamma}\right]+\mathbb{P}\left[ X \not \in \mathcal{T}_{\gamma} \right] \nonumber\\
    & \leq \mathbb{P}\left[ |\mathcal{A}\left(X\right)-m\left(\mathcal{D}\right)| \geq \alpha | X \in \mathcal{T}_{\gamma}\right]+\gamma. \label{First0}
\end{align} 
Now conditioning on $X \in \mathcal{T}_{\gamma}$ we have that there exists an $X' \in \mathcal{H}$ with $d_H\left(X , X' \right)  \leq C' \log \frac{1}{\gamma}$ and $m(X')=m(X)$. Since the algorithm $\mathcal{A}$ is $\epsilon$-differentially private by its definition, we have that $\mathcal{A}(X)$ and $\mathcal{A}(X')$ assigns to each output value the same probability mass up to a multiplicative factor of $e^{\epsilon d_H(X,X')} \leq e^{\epsilon C' \log \frac{1}{\gamma}} \leq e^{C' \log \frac{1}{\gamma}},$ where in the last inequality we use that $\epsilon<1.$ Hence
\begin{align}
\mathbb{P}\left[ |\mathcal{A}\left(X\right)-m\left(\mathcal{D}\right)| \geq \alpha | X \in \mathcal{T}_{\gamma}\right] & \leq \mathbb{E}\left[e^{\epsilon d_H(X,X')}\mathbb{P}\left[ |\mathcal{A}\left(X'\right)-m\left(\mathcal{D}\right)| \geq \alpha | X \in \mathcal{T}_{\gamma}\right] \right] \nonumber\\
& \leq  \frac{e^{C' \log \frac{1}{\gamma}}}{ \mathbb{P}\left[X \in \mathcal{T}_{\gamma} \right]} \max_{X' \in \mathcal{H}, m(X')=m(X)}\mathbb{P}\left[ |\mathcal{A}\left(X'\right)-m\left(\mathcal{D}\right)| \geq \alpha \right] \nonumber \\
&\leq  \frac{e^{C' \log \frac{1}{\gamma}}}{1-\gamma} \max_{X' \in \mathcal{H}, m(X')=m(X)}\mathbb{P}\left[ |\mathcal{A}\left(X'\right)-m\left(\mathcal{D}\right)| \geq \alpha \right] \nonumber\\
&= \frac{e^{C' \log \frac{1}{\gamma}}}{1-\gamma} \max_{X' \in \mathcal{H}, m(X')=m(X)}\mathbb{P}\left[ |\hat{\mathcal{A}}\left(X'\right)-m\left(\mathcal{D}\right)| \geq \alpha \right], \label{Final1}\\
\end{align}
where in the last line we use that for all $X' \in \mathcal{H},$ it holds $\mathcal{A}(X')\overset{d}{=}\hat{\mathcal{A}}(X').$

According to Lemma \ref{lemma:accuracy-typical} for some $n=O\left( C\frac{\log \frac{1}{\gamma}}{\epsilon L \alpha}+\frac{\log \left(\frac{R}{\alpha}+1\right)}{\epsilon L r}\right)$ we can guarantee
\begin{align}
   \max_{X' \in H, m(X')=m(X)} \mathbb{P}\left[ |\hat{\mathcal{A}}\left(X\right)-m\left(X\right)| \geq \frac{\alpha}{2} \right]  \leq \gamma.\label{eq:atypical}
\end{align} Using classical results we have that for some $n=O\left(\frac{\log(\frac{1}{\gamma})}{L^2\alpha^2} \right)$ it holds \begin{equation}\label{eq:typical3}
    \mathbb{P} [ |m(X)-m\left(\mathcal{D}\right)|  \geq \frac{\alpha}{2}] \leq \gamma.
\end{equation}Combining the \eqref{eq:atypical} and \eqref{eq:typical3} we have for some $n=O\left(\frac{\log(\frac{1}{\gamma})}{L^2\alpha^2}+C\frac{\log \frac{1}{\gamma}}{\epsilon L \alpha}+\frac{\log \left(\frac{R}{\alpha}+1\right)}{\epsilon L r}  \right)$ it holds
\begin{align}
   \max_{X' \in H, m(X')=m(X)} \mathbb{P}\left[ |\hat{\mathcal{A}}\left(X\right)-m\left(\mathcal{D}\right)| \geq \alpha \right]  \leq 2\gamma.\label{eq:atypical2}
\end{align} or, combining with \eqref{Final1}, 
\begin{align*}
\mathbb{P}\left[ |\mathcal{A}\left(X\right)-m\left(\mathcal{D}\right)| \geq \alpha | X \in \mathcal{T}_{\gamma}\right] & \leq  \frac{e^{C' \log \frac{1}{\gamma}}}{1-\gamma} 2 \gamma.
\end{align*} or, combining with \eqref{First0},
\begin{align*}
\mathbb{P}\left[ |\mathcal{A}\left(X\right)-m\left(\mathcal{D}\right)| \geq \alpha \right] \leq  \frac{e^{C' \log \frac{1}{\gamma}}}{1-\gamma} 2 \gamma+\gamma.
\end{align*}Since $C'(C)=\Theta(\frac{1}{C})$ we can assume $C>0$ sufficiently large so that $C'(C)<\frac{1}{2}$. Hence for these values of $C$ some $n=O\left(\frac{\log(\frac{1}{\gamma})}{L^2\alpha^2}+C\frac{\log \frac{1}{\gamma}}{\epsilon L \alpha}+\frac{\log \left(\frac{R}{\alpha}+1\right)}{\epsilon L r}  \right)$ it holds \begin{align*}
\mathbb{P}\left[ |\mathcal{A}\left(X\right)-m\left(\mathcal{D}\right)| \geq \alpha \right] \leq  \frac{\sqrt{\gamma}}{1-\gamma} 2 \gamma+\gamma.
\end{align*}Choosing $\gamma$ appropriately of the order $\gamma=\Theta\left(\beta^{\frac{2}{3}}\right)$ we conclude that $$n=O\left(\frac{\log(\frac{1}{\beta})}{L^2\alpha^2}+C\frac{\log \frac{1}{\beta}}{\epsilon L \alpha}+\frac{\log \left(\frac{R}{\alpha}+1\right)}{\epsilon L r}  \right)$$ it holds \begin{align*}
\mathbb{P}\left[ |\mathcal{A}\left(X\right)-m\left(\mathcal{D}\right)| \geq \alpha \right] \leq  \beta.
\end{align*} The proof is complete.

\end{proof}

\section{Proof for Section \ref{sec:LB}: Lower Bounds}
\begin{proof}[Proof of Proposition \ref{propLB1}]
We argue by contradiction and consider an algorithm $\mathcal{A}$ satisfying the negation of the statement of the proposition.

We first prove that \begin{align}\label{eq:LB11}n=\Omega \left(\frac{r^2 \log (\frac{1}{\beta})}{\alpha^2} \right)\end{align} and then \begin{align}\label{eq:LB12}n=\Omega\left( \frac{\log \left( \frac{1}{\beta}\right)}{L^2 \alpha^2}(1-2Lr)\right).\end{align} Notice that combined the lower bounds imply $$n=\Omega \left(\frac{r^2 \log (\frac{1}{\beta})}{\alpha^2} + \frac{\log \left( \frac{1}{\beta}\right)}{L^2 \alpha^2}(1-2Lr)\right)=\Omega \left(\frac{ \log (\frac{1}{\beta})}{\alpha^2}\frac{(1-Lr)^2}{L^2}\right)=\Omega\left( \frac{\log \left( \frac{1}{\beta}\right)}{L^2 \alpha^2}\right),$$where for the last equality we use $Lr \leq \frac{1}{2}.$  The last displayed equation yields the desired contradiction.

To prove \ref{eq:LB11} notice that all uniform distributions which are supported on an interval of width $2r$ inside $[-R-r,R+r]$ are admissible. Observe now that for a uniform distribution the mean and the median of it are identical. For this reason, $\mathcal{A}$ to satisfy the negation of our statement, it should learn the mean of these uniform distributions of width $r$ from $n$ samples with accuracy $\alpha$ with probability $1-\beta.$ Standard learning theory implies that it should hold $n=\Omega \left(\frac{r^2 \log (\frac{1}{\beta})}{\alpha^2} \right).$

We now turn to \ref{eq:LB12}. Notice that if $Lr=\frac{1}{2}$ the lower bound on $n$ is trivial. Hence, to establish \ref{eq:LB12}, we focus on the case where $Lr<\frac{1}{2}.$

Recall the standard fact that learning the parameter $p$ of a Bernoulli random variable $\mathrm{Bernoulli}\left(p\right)$ at accuracy $\gamma>0$ with probability $1-\beta$ requires $\Omega\left( \frac{\log \left( \frac{1}{\beta}\right)}{\gamma^2}\right)$ samples. 

We know fix $p$. We construct the admissible distribution $\mathcal{D}$ which assigns probability mass $(1-2Lr)(1-p)$ at $-2r$, probability mass $(1-2Lr) p$ at $2r$ and with probability $2Lr$ samples from the uniform distribution on $[-r,r].$ It can be easily checked that $\mathcal{D}$ is admissible with the assumed parameters and median $p\left(\frac{1}{L}-2r\right)+r-\frac{1}{2L}$. In particular, learning the median at accuracy $\alpha$ is equivalent with learning the parameter $p$ at accuracy $\alpha \frac{2L}{1-2Lr}$ which from our assumption on $Lr < \frac{1}{2}$ is $\Theta \left((1-2Lr)^{-1}L\alpha \right).$ Furthermore, notice that for learning the parameter $p$ using the samples from $\mathcal{D}$ one needs to focus only the samples from which are equal to either $2r$ or $-2r$. Therefore the task of learning the median of $\mathcal{D}$ with $n$ samples at accuracy $\alpha$, reduces to learning the parameter $p$ of a $\mathrm{Bernoulli}(p)$ distribution with $N_1=\mathrm{Binom}(n,1-2Lr)$ samples at accuracy $\Theta \left((1-2Lr)^{-1}L\alpha \right).$ Using the standard fact mentioned above and this equivalence, we have that conditional on the event, call it $\mathcal{E}_{\alpha,\beta}(\mathcal{D}),$ that we can learn the median of $\mathcal{D}$ with $n$ samples at accuracy $\alpha$ it holds $$N_1=\Omega\left( \frac{\log \left( \frac{1}{\beta}\right)}{\alpha^2L^2} (1-2Lr)^2\right).$$

Now recall that we assume that $\mathcal{A}$ can learn the median of $\mathcal{D}$ with $n$ samples at accuracy $\alpha$ holds with probability at least $1-\beta.$ Hence $\mathbb{P}\left(\mathcal{E}_{\alpha,\beta}(\mathcal{D})\right) \geq 1-\beta$. Combined with the last observation of the paragraph above, we conclude \begin{align*}
    \mathbb{E}[N_1] \geq \mathbb{P}\left(\mathcal{E}_{\alpha,\beta}(\mathcal{D})\right) \mathbb{E}[N_1|\mathcal{E}_{\alpha,\beta}(\mathcal{D})] \geq (1-\beta)\Omega\left( \frac{\log \left( \frac{1}{\beta}\right)}{\alpha^2L^2}(1-2Lr)^2\right).
\end{align*}Using that $\mathbb{E}[N_1]=(1-2Lr)n$ and that $\beta \in (0,\frac{1}{2})$ we have  $$(1-2Lr)n=\Omega\left( \frac{\log \left( \frac{1}{\beta}\right)}{L^2 \alpha^2}(1-2Lr)^2\right).$$ Using that $Lr < \frac{1}{2}$ we have $n=\Omega\left( \frac{\log \left( \frac{1}{\beta}\right)}{L^2 \alpha^2}(1-2Lr)\right).$ The proof is complete.

\end{proof}

\begin{proof}[Proof of Proposition \ref{propLB2}]
We argue by contradiction and consider an algorithm $\mathcal{A}$ satisfying the negation of our statement. Since $R>2\alpha$, for $\eta>0$ sufficiently small it holds $2\alpha+\eta<R.$ We consider a $\mathcal{D}_1$ which assigns mass $\frac{1}{2}-Lr$ to $-r-R,$ mass $\frac{1}{2}-Lr$ to $r+R$  and with probability $2Lr$ samples from the uniform distribution on $[-r,r]$ and a $\mathcal{D}_2$ which assigns mass  $\frac{1}{2}-Lr$ to $-r-R,$ mass $\frac{1}{2}-Lr$ to $r+R$ and with probability $2Lr$ samples from the uniform distribution on $[-r+2\alpha+\eta,r+2\alpha+\eta]$. Note that they are both admissible with the desired parameters and it also holds $m(\mathcal{D}_1)=0,m(\mathcal{D}_2)=2\alpha+\eta.$  

Using the accuracy guarantee of $\mathcal{A}$ and that the medians of the two distribution have distance strictly bigger than $2\alpha$, it necessarily holds
\begin{equation}\label{eq:inclusion}\mathbb{P}_{X_1,X_2,\ldots,X_n \overset{iid}{\sim} \mathcal{D}_1} [ |\mathcal{A}(X)-m\left(\mathcal{D}_2\right)|  \leq \alpha] \leq \mathbb{P}_{X_1,X_2,\ldots,X_n \overset{iid}{\sim} \mathcal{D}_1} [ |\mathcal{A}(X)-m\left(\mathcal{D}_1\right)|  \geq \alpha]  \leq \beta.
\end{equation}
The two distribution can be coupled in the following way: we first sample from $ \mathcal{D}_1.$ If it falls in $[-r+2\alpha+\eta,r]$ we keep the sample from $\mathcal{D}_1$ as is, and translate it by $+2r$ if it falls into $[-r,-r+2\alpha+\eta]$, to form a sample from $\mathcal{D}_2.$ In particular, notice that by sampling two $n$-tuples $X ,X' \in \mathbb{R}^n$ of samples in an i.i.d. sense from $\mathcal{D}_1$ and $\mathcal{D}_2$ respectively, under this coupling the $d_H(X,X')$ follows a $\mathrm{Binom}\left(n,(2\alpha+\eta) L \right)$. Hence using the definition of $\epsilon$-differential privacy 
\begin{equation}\label{eq:inclusion2}\mathbb{E}\left[e^{-\epsilon d_H(X,X')} \mathbb{P}_{X' \overset{iid}{\sim} \mathcal{D}_1} [ |\mathcal{A}(X')-m\left(\mathcal{D}_2\right)|  \leq \alpha] \right] \leq \mathbb{P}_{X_1,X_2,\ldots,X_n \overset{iid}{\sim} \mathcal{D}_1} [ |\mathcal{A}(X)-m\left(\mathcal{D}_2\right)|  \geq \alpha],
\end{equation}where the expectation is over the coupling mentioned above. Combining \eqref{eq:inclusion}, \eqref{eq:inclusion2} and the assumption on the accuracy performance of $\mathcal{A}$ we have 
\begin{equation*}\mathbb{E}\left[e^{-\epsilon d_H(X,X')}  \right] \leq \frac{\beta}{1-\beta}.
\end{equation*} Using the moment generating function of the Binomial distribution we conclude
\begin{equation*}\left(1-L(2\alpha+\eta)\left(1-e^{-\epsilon}\right)\right)^n \leq \frac{\beta}{1-\beta}
\end{equation*}which by basic asymptotics since $\beta<1/2$ translates to $n=\Omega\left( \frac{ \log \left(\frac{1}{\beta}\right)}{\epsilon L(\alpha+\eta) }\right).$ Since $\eta>0$ can be taken arbitrarily small, the proof of the proposition is complete.

\end{proof}

\begin{proof}[Proof of Proposition \ref{propLB3}]
We  argue by contradiction and consider an algorithm $\mathcal{A}$ satisfying the negation of our statement.

We consider the partition of the interval $[-R,R]$ into $N=\Omega\left(\frac{R}{\alpha}+1\right)$ consecutive intervals of width $3\alpha$ and let $m_i, i=1,2,\ldots,N+1$ be the endpoints of these intervals.

Now consider $N$ admissible distributions $\mathcal{D}_i,i=1,2,\ldots,N$ which for each $i,$ assign mass $\frac{1}{2}-Lr$ at each of the data-points $-2(R+r)$ and $2(R+r)$ and with probability $2Lr$ it draws a sample from the uniform distribution on $[m_i-r,m_i+r].$ 

By assumption for all $i=1,2,\ldots,N$ it holds
\begin{equation}\label{eq:accur}
\mathbb{P}_{X_1,X_2,\ldots,X_n \overset{iid}{\sim} \mathcal{D}_i} [ |\mathcal{A}(X_1,X_2,\ldots,X_n)-m_i|  \leq \alpha] \geq 1-\beta.
\end{equation}

Now since for all $i=2,\ldots,N+1$ ($i\not =1$) the distributions $\mathcal{D}_1,\mathcal{D}_i$ differ only on the interval that they are uniform on which they fall with probability $2Lr$, we can straightforwardly couple the $n$-tuples $X,X'$ sampled in an i.i.d. fashion from $\mathcal{D}_1,\mathcal{D}_i$ such that $d_H(X,X')$ follows a $\mathrm{Binom}\left(n,2Lr\right).$ Hence from the definition of $\epsilon$-differential privacy it holds for all $i=2,3,\ldots,N+1$
\begin{equation}\label{eq:priv} \mathbb{E}\left[e^{-\epsilon d_H(X,X')} \mathbb{P}_{X' \overset{iid}{\sim} \mathcal{D}_i} [ |\mathcal{A}(X')-m_i|  \leq \alpha] \right] \leq \mathbb{P}_{X \overset{iid}{\sim} \mathcal{D}_1} [ |\mathcal{A}(X)-m_i|,  \leq \alpha]  
\end{equation}where the expectation of the left hand side is under the aforementioned coupling.
Now using \eqref{eq:accur} and the moment generating function of the Binomial distribution, it holds for all $i=2,3,\ldots,N+1$
\begin{equation*} (1-\beta) \left(1-2Lr\left(1-e^{-\epsilon}\right)\right)^n  \leq \mathbb{P}_{X \overset{iid}{\sim} \mathcal{D}_1} [ |\mathcal{A}(X)-m_i|  \leq \alpha] 
\end{equation*}Now notice that the intervals $[m_i-\alpha,m_i+\alpha], i=2,3,\ldots,N$ are disjoint and therefore 
\begin{equation*} (1-\beta) N\left(1-2Lr\left(1-e^{-\epsilon}\right)\right)^n  \leq \mathbb{P}_{X \overset{iid}{\sim} \mathcal{D}_1} [\bigcup_{i=2}^{N+1} \{|\mathcal{A}(X)-m_i|  \leq \alpha\}] \leq 1. 
\end{equation*} By standard asymptotics and as $Lr \leq \frac{1}{2}$, $\left(1-2Lr\left(1-e^{-\epsilon}\right)\right)^n  =\Omega\left(e^{-Lr\epsilon n}\right).$ Hence combining with the last displayed inequality, it holds
\begin{equation*} \left(1-\beta\right)N=O\left(e^{Lr\epsilon n} \right).\end{equation*} As $\beta<1/2$ and $N=\Omega\left(\frac{R}{\alpha}+1\right)$ we conclude 
$$n=\Omega\left( \frac{\log \left(\frac{R}{ \alpha}+1\right)}{\epsilon Lr}\right).$$The proof is complete.

\end{proof}

\clearpage
\section{Comments about Truncated-Sampling}\label{appendix:truncated-sampling}
In this last subsection, for the sake of completeness, we will present how we can simulate a random generator 
for Categorical, Bernoulli, Uniform, Laplace, (Negative) Exponential Distribution and their truncated versions in an interval $I$ in $O(1)$ time and $O(1)$ random queries
of $\mathrm {Unif} [0,1]$ using the standard ``inverse-CDF'' sampling method.

The general inverse-CDF sampling method works as follows:

\begin{center}
\fbox{
\begin{minipage}{10cm}
//\texttt{Suppose that we want to simulate a distribution $\mathcal{D}$ s.t its cdf $F_{\mathcal{D}}$ is invertible.  }
\begin{enumerate}
    \item Generate a random number $u$ from the standard uniform distribution in the interval $\displaystyle [0,1]$, e.g. from $\displaystyle U\sim \mathrm {Unif} [0,1].$
    \item Find the inverse of the desired CDF, e.g. $\displaystyle F_{\mathcal{D}}^{-1}(x)$.
    \item Output $\displaystyle X=F_{\mathcal{D}}^{-1}(u)$.
\end{enumerate}
\end{minipage}
}
\end{center}

Indeed, the computed random variable ${\displaystyle X}$ follows the distribution $\mathcal{D}$, since
$$\Pr(X\leq x)=\Pr(F_{\mathcal{D}}^{-1}(U)\leq x) = \Pr(U\leq F_{\mathcal{D}}(x))=F_{\mathcal{D}}(x).$$

\subsection{\textsc{Uniform}}
We start with an arbitrary uniform distribution in an interval $I=[a,b]$:
\begin{algorithm}
$u\gets$ Sample from $\displaystyle U\sim \mathrm {Unif} [0,1]$\\
$s\gets a+(b-a)\times u$\\
\Return $s$
\caption{$\textsc{Uniform}\left[I=[a,b]\right]$}
\label{sampler:Uniform}
\end{algorithm}
\subsection{$k-$nary coin}
We continue with a $k-$nary coin, known as Categorical distribution, which
is a discrete probability distribution that describes the possible results of a random variable that can take on one of $k$ possible categories, 
with the probability of each category separately specified. 
\begin{algorithm}
$u\gets$ Sample from $\displaystyle U\sim \mathrm {Unif} [0,1]$\\
Compute CDF vector $(q_0,q_1,q_2,\cdots,q_{k-1},q_k)=(0,p_1,p_1+p_2,\cdots,\sum_{i=1}^{k-1}p_i,\sum_{i=1}^{k}p_i=1)$\\
\For{$i\in [k]$}
{
\If{$q_{i-1}\le s<q_{i}$}
{
\[\quad\]
$s\gets o_i$
}
}
\Return $s$
\caption{$k-$nary coin $\{o_1,\cdots,o_k\}$ with probability $p_1,\cdots,p_k$}
\label{sampler:coin}
\end{algorithm}
\subsection{\textsc{TruncatedExponential}}
Our next probability distibution is the truncated (Positive/Negative) Exponential distribution $\mathcal{D}_{\text{TruncExponential}}\sim 
 \exp\left(\alpha \omega+\beta\right) \quad \omega\in I=[\texttt{left},\texttt{right}]$. 
 \begin{itemize}
     \item If $\alpha=0$ then  the actual distribution is the uniform.
     \item If $\alpha\neq 0$ we get the CDF of $\mathcal{D}_{\text{TruncExponential}}$ is
 \[F_{\mathcal{D}_{\text{TruncExponential}}}(x|\alpha,\beta)=\dfrac{ \exp\left(\alpha x+\beta\right)- \exp\left(\alpha\cdot \texttt{left}+\beta\right)}{ \exp\left(\alpha\cdot \texttt{right}+\beta\right)- \exp\left(\alpha\cdot \texttt{left}+\beta\right)}.\]
 Thus we get that:
  \[F_{\mathcal{D}_{\text{TruncExponential}}}^{-1}(x|\alpha,\beta)=  \frac{1}{\alpha}\ln\Big(\exp\left(\alpha\cdot \texttt{left}+\beta\right)+ x\left( \exp\left(\alpha\cdot \texttt{right}+\beta\right)- \exp\left(\alpha\cdot \texttt{left}+\beta\right)\right)\Big)-\beta\]
 \end{itemize} 
\begin{algorithm}
$u\gets$ Sample from $\textsc{Uniform}\left[I=[0,1]=[F_{\mathcal{D}_{\text{TruncExponential}}}(\texttt{left}),F_{\mathcal{D}_{\text{TruncExponential}}}(\texttt{right})]\right]$\\
$s\gets F_{\mathcal{D}_{\text{TruncExponential}}}^{-1}(u)$\\
\Return $s$
\caption{$\textsc{TruncatedExponential}\left[\alpha:\text{scale},\beta:\text{shift},I=[\texttt{left},\texttt{right}]\right]$}
\label{sampler:Exponential}
\end{algorithm}
\subsection{\textsc{TruncatedLaplace}}
Our last probability distibution is the truncated Laplace distribution $D_{\text{TruncLaplace}}\sim 
 \frac{1}{2\sigma}\exp\left(-\frac{|\omega-\mu|}{\sigma}\right) \quad \omega\in I=[\texttt{left},\texttt{right}]$. 
 Here we firstly present the CDF and its inverse of the classical non-truncated version of Laplace distribution.
 \begin{itemize}
     \item  $ F_{\mathcal{D}_{\text{Laplace}}}(x|\mu,\sigma)=\frac{1}{2}+\frac{1}{2}\operatorname{sign}(x-\mu)\left(1-\frac{1}{\sigma}\exp\left(-\frac{|\omega-\mu|}{\sigma}\right)\right) $
     \item $ F_{\mathcal{D}_{\text{Laplace}}}^{-1}(x|\mu,\sigma)=\mu+\sigma\operatorname{sign}(x-\frac{1}{2})\ln\left(1-2|x-0.5|\right)$
 \end{itemize}
 \[F_{\mathcal{D}_{\text{TruncLaplace}}}(x|\mu,\sigma)=\frac{F_{\mathcal{D}_{\text{Laplace}}}(x|\mu,\sigma) - F_{\mathcal{D}_{\text{Laplace}}}(\texttt{left}|\mu,\sigma) }{F_{\mathcal{D}_{\text{Laplace}}}(\texttt{right}|\mu,\sigma) - F_{\mathcal{D}_{\text{Laplace}}}(\texttt{left}|\mu,\sigma)}.\]
 Thus we get that:
  \[F_{\mathcal{D}_{\text{TruncLaplace}}}^{-1}(x|\mu,\sigma)= F_{\mathcal{D}_{\text{Laplace}}}^{-1}\left( F_{\mathcal{D}_{\text{Laplace}}}(\texttt{left}|\mu,\sigma) + x\times\left( F_{\mathcal{D}_{\text{Laplace}}}(\texttt{right}|\mu,\sigma) - F_{\mathcal{D}_{\text{Laplace}}}(\texttt{left}|\mu,\sigma) \right)|\mu,\sigma\right) \]
\begin{algorithm}
$u\gets$ Sample from $\textsc{Uniform}\left[I=[0,1]=[F_{\mathcal{D}_{\text{TruncLaplace}}}(\texttt{left}),F_{\mathcal{D}_{\text{TruncLaplace}}}(\texttt{right})]\right]$\\
$s\gets F_{\mathcal{D}_{\text{TruncLaplace}}}^{-1}(u)$\\
\Return $s$
\caption{$\textsc{TruncatedLaplace}\left[\mu:\text{location},\sigma:\text{scale},I=[\texttt{left},\texttt{right}]\right]$}
\label{sampler:Laplace}
\end{algorithm}

\end{document}